\documentclass{amsart}

\usepackage[T1]{fontenc}
\usepackage[utf8]{inputenc}
\usepackage{mathrsfs}
\usepackage{amssymb}
\usepackage{tikz-cd}
\usepackage[french]{babel}
\usepackage{hyperref}

\newcommand{\bk}{\Bbbk}
\newcommand{\Z}{\mathbb{Z}}
\newcommand{\C}{\mathbb{C}}
\newcommand{\F}{\mathbb{F}}

\newcommand{\frg}{\mathfrak{g}}
\newcommand{\cO}{\mathcal{O}}
\newcommand{\gr}{{\mathrm{gr}}}
\newcommand{\scB}{\mathscr{B}}
\newcommand{\scG}{\mathscr{G}}

\newcommand{\cF}{\mathcal{F}}

\newcommand{\sA}{\mathscr{A}}
\newcommand{\sB}{\mathscr{B}}
\newcommand{\sC}{\mathscr{C}}
\newcommand{\Db}{D^{\mathrm{b}}}
\newcommand{\Kb}{K^{\mathrm{b}}}
\newcommand{\lf}{{\mathrm{lf}}}
\newcommand{\pf}{{\mathrm{pf}}}
\newcommand{\op}{{\mathrm{op}}}
\newcommand{\lmod}{\text{\normalfont-}\mathrm{mod}}
\newcommand{\lgmod}{\text{\normalfont-}\mathrm{gmod}}

\newcommand{\Ext}{\mathrm{Ext}}
\newcommand{\Sym}{\mathrm{Sym}}
\newcommand{\Hom}{\mathrm{Hom}}
\newcommand{\End}{\mathrm{End}}
\newcommand{\id}{\mathrm{id}}
\newcommand{\uHom}{\underline{\Hom}}
\newcommand{\uExt}{\underline{\Ext}}
\newcommand{\uEnd}{\underline{\End}}
\newcommand{\Irr}{\mathrm{Irr}}
\newcommand{\Proj}{\mathrm{Proj}}
\newcommand{\Inj}{\mathrm{Inj}}
\newcommand{\simto}{\xrightarrow{\sim}}
\newcommand{\la}{\langle}
\newcommand{\ra}{\rangle}
\newcommand{\lla}{\langle\hspace{-1pt}\langle}
\newcommand{\rra}{\rangle\hspace{-1pt}\rangle}

\newcommand{\Flag}{\mathscr{X}}
\newcommand{\Par}{\mathsf{Par}}
\newcommand{\Perv}{\mathsf{Perv}}
\newcommand{\Coh}{\mathsf{Coh}}
\newcommand{\Rep}{\mathsf{Rep}}
\newcommand{\Dmix}{D^{\mathrm{mix}}}
\newcommand{\dmix}{\Delta^{\mathrm{mix}}}
\newcommand{\nmix}{\nabla^{\mathrm{mix}}}
\newcommand{\mix}{\mathrm{mix}}
\newcommand{\bX}{\mathbf{X}}
\newcommand{\cN}{\mathcal{N}}
\newcommand{\tcN}{\widetilde{\mathcal{N}}}
\newcommand{\Gm}{\mathbb{G}_{\mathrm{m}}}

\newcommand{\Gr}{\mathrm{Gr}}
\newcommand{\Iw}{\mathrm{Iw}}
\newcommand{\Waff}{W_{\mathrm{aff}}}
\newcommand{\Wext}{W_{\mathrm{ext}}}
\newcommand{\Waffmin}{{}^0 \hspace{-1pt} W_{\mathrm{aff}}}
\newcommand{\Wextmin}{{}^0 \hspace{-1pt} W_{\mathrm{ext}}}

\newcommand{\coh}{\mathsf{H}}
\newcommand{\Fr}{\mathrm{Fr}}
\newcommand{\bc}{\mathbf{c}}

\newcommand{\puH}{{}^p \hspace{-1pt} \underline{H}}
\newcommand{\puM}{{}^p \hspace{-1pt} \underline{M}}
\newcommand{\puN}{{}^p \hspace{-1pt} \underline{N}}
\newcommand{\ph}{{}^p \hspace{-1pt} h}
\newcommand{\pmm}{{}^p \hspace{-1pt} m}
\newcommand{\pn}{{}^p \hspace{-1pt} n}
\newcommand{\Fl}{\mathrm{Fl}}
\newcommand{\IW}{\mathcal{IW}}

\newcommand{\ext}{\mathrm{ext}}

\newcommand{\bG}{\mathbf{G}}
\newcommand{\bB}{\mathbf{B}}
\newcommand{\bT}{\mathbf{T}}

\newcommand{\WKM}{\mathcal{W}}

\numberwithin{equation}{section}
\newtheorem{thm}{Théorème}[section]
\newtheorem{lem}[thm]{Lemme}
\newtheorem{prop}[thm]{Proposition}

\newtheorem{conj}[thm]{Conjecture}
\theoremstyle{definition}
\newtheorem{defn}[thm]{Définition}
\theoremstyle{remark}
\newtheorem{rmq}[thm]{Remarque}
\newtheorem{ex}[thm]{Exemple}

\title[Dualit\'e de Koszul et th\'eorie des repr\'esentations]{Dualit\'e de Koszul formelle et th\'eorie des repr\'esentations des groupes alg\'ebriques r\'eductifs en caract\'eristique positive}

\author{Pramod N. Achar}
\address{Department of Mathematics\\
  Louisiana State University\\
  Baton Rouge, LA 70803\\
  U.S.A.}
\email{pramod@math.lsu.edu}

\author{Simon Riche}
\address{Universit\'e Clermont Auvergne, CNRS, LMBP, F-63000 Clermont-Ferrand, France.
}
\email{simon.riche@uca.fr}

\thanks{P.A. was supported by NSF Grant No.~DMS-1500890.
This project has received funding from the European Research Council (ERC) under the European Union's Horizon 2020 research and innovation programme (grant agreement No 677147).
}

\makeatletter
\def\@@and{et}
\makeatother

\begin{document}

\begin{abstract}
Dans cet article nous pr\'esentons les grandes lignes de la preuve d'une formule de caract\`eres pour les repr\'esentations basculantes des groupes alg\'ebriques r\'eductifs sur un corps de caract\'eristique positive, obtenue partiellement en collaboration avec plusieurs auteurs. Nous unissons les diff\'erentes \'etapes de cette preuve dans la notion de ``dualit\'e de Koszul formelle'', et en pr\'esentons quelques applications.
\end{abstract}

\maketitle

\setcounter{tocdepth}{1}
\tableofcontents

\section{Introduction}

\subsection{Pr\'esentation}

Le but principal de cet article est de pr\'esenter les grandes lignes de la preuve d'une formule de caract\`eres pour les repr\'esentations basculantes des groupes alg\'ebriques r\'eductifs sur des corps alg\'ebriquement clos de caract\'eristique positive, conjectur\'ee par G. Williamson et le second auteur~\cite{rw}, et d\'emontr\'ee dans une suite de travaux partiellement en collaboration avec S.~Makisumi, C.~Mautner, L.~Rider, et G.~Williamson (voir notamment~\cite{arider, mr, prinblock, amrw2}).

Cette preuve implique la construction de plusieurs \'equivalences de cat\'egories qui ont une structure similaire, que nous avons essay\'e d'axiomatiser dans la notion de ``dualit\'e de Koszul formelle''. La dualit\'e de Koszul ``classique'' (\'etudi\'ee notamment dans un article fondateur de Be{\u\i}linson--Ginzburg--Soergel~\cite{bgs}) est une \'equivalence de cat\'egories qu'on peut construire quand on rencontre un \emph{anneau de Koszul}. Une dualit\'e de Koszul formelle est ``une dualit\'e de Koszul sans anneau de Koszul'', c'est-\`a-dire un foncteur poss\'edant certaines propri\'et\'es formellement similaires \`a celles des dualit\'es de Koszul (ou plus pr\'ecis\'ement de ``Ringel--Koszul'') classiques, mais qui peut exister m\^eme dans des situations o\`u aucun anneau de Koszul n'est pr\'esent. 

\subsection{Anneaux et dualit\'e de Koszul}

La notion d'\emph{anneau de Koszul} a \'et\'e introduite par Priddy, et \'etudi\'ee dans un cadre g\'en\'eral et tr\`es satisfaisant par Be{\u\i}linson--Ginzburg--Soergel~\cite{bgs}. Il s'agit d'un anneau $A$ gradu\'e par $\Z_{\geq 0}$, dont la composante de degr\'e $0$ est un anneau semi-simple, et qui v\'erifie la condition que
\begin{equation}
\label{eqn:def-Koszul-intro}
\Hom_{\Db(A\lgmod)}(A_0, A_0 \langle n \rangle[m])=0 \quad \text{si $n+m \neq 0$.}
\end{equation}
(Ici, $A\lgmod$ d\'esigne la cat\'egorie des $A$-modules gradu\'es de type fini, et $\langle n \rangle$ est le d\'ecalage de la graduation, avec des conventions pr\'ecis\'ees ci-dessous.)

Sous certaines conditions techniques, si $A$ est un anneau de Koszul l'anneau gradu\'e
\[
A^! := \left( \bigoplus_{n \in \Z} \Hom_{\Db(A\text{-gMod})}(A_0, A_0 \langle -n \rangle[n]) \right)^\op
\]
est encore un anneau de Koszul (appel\'e \emph{anneau de Koszul dual}), et (sous des hypoth\`ese techniques plus fortes) il existe une \'equivalence de cat\'egories triangul\'ees $\kappa$ entre les cat\'egories d\'eriv\'ees born\'ees des $A$-modules gradu\'es de type fini et des $A^!$-modules gradu\'es de type fini (appel\'ee \emph{dualit\'e de Koszul} associ\'ee \`a $A$). En un sens cette \'equivalence ``m\'elange'' la graduation des modules et la graduation cohomologique des complexes ; plus pr\'eci\'ement elle v\'erifie $\kappa \circ \langle 1 \rangle \cong \langle -1 \rangle[1] \circ \kappa$. En particulier, la condition~\eqref{eqn:def-Koszul-intro} se traduit par le fait que l'image par $\kappa$ d'un $A$-module simple gradu\'e concentr\'e en degr\'e $0$ est un $A^!$-module gradu\'e projectif, et dualement l'image inverse par $\kappa$ d'un $A^!$-module simple gradu\'e concentr\'e en degr\'e $0$ est un $A$-module gradu\'e injectif.

Le mod\`ele typique de cette construction est obtenu quand 
$A = \bigwedge V$ est l'algèbre extérieure d'un espace vectoriel $V$ de dimension finie (avec sa graduation naturelle) ; l'anneau dual est alors l'alg\`ebre 
symétrique $\mathrm{S}(V^*)$ de l'espace vectoriel dual $V^*$ (avec sa graduation naturelle, encore une fois).

\subsection{La dualit\'e de Koszul de Be{\u\i}linson--Ginzburg--Soergel}
\label{ss:bgs-intro}

L'exemple principal qui a motiv\'e l'\'etude g\'en\'erale de ces constructions par Be{\u\i}linson--Ginzburg--Soergel est celui o\`u $A$ est l'anneau qui ``gouverne'' un bloc r\'egulier de la cat\'egorie $\mathcal{O}$ d'une alg\`ebre de Lie semi-simple complexe (c'est-\`a-dire l'anneau oppos\'e \`a celui des endomorphismes d'un g\'en\'erateur projectif de cette cat\'egorie). Cet anneau n'est pas ``naturellement'' gradu\'e ; cependant Be{\u\i}linson--Ginzburg--Soergel montrent qu'il peut \^etre muni d'une graduation qui en fait un anneau de Koszul, isomorphe \`a son anneau de Koszul dual. La dualit\'e de Koszul associ\'ee permet d'``expliquer'' certaines formules v\'erif\'ees par les polyn\^omes de Kazhdan--Lusztig, et notamment la ``formule d'inversion'' de Kazhdan--Lusztig.

Cette construction a jou\'e un r\^ole pr\'epond\'erant en th\'eorie g\'eom\'etrique des repr\'esentations, et a depuis \'et\'e adapt\'ee dans divers contextes proches ; voir par exemple~\cite{so-icm,bac, riche-koszul, svv, es}.

\subsection{Dualit\'e de Koszul et anneaux quasi-h\'er\'editaires}

En plus du fait qu'il peut \^etre muni d'une graduation de Koszul, l'anneau $A$ consid\'er\'e au~\S\ref{ss:bgs-intro} poss\`ede une autre structure remarquable : il s'agit d'un anneau \emph{quasi-h\'er\'editaire} ; en d'autres termes, il poss\`ede des modules ``standards'' et des modules ``costandards'' qui exhibent des propri\'et\'es cohomologiques remarquables. Les anneaux gradu\'es qui poss\`edent ces deux structures, et pour lesquels les deux structures sont ``compatibles'' en un sens appropri\'e, sont appel\'es \emph{quasi-h\'er\'editairement de Koszul}.

Les anneaux quasi-h\'er\'editaires (gradu\'es ou non) poss\`edent \'egalement un autre type de ``dualit\'e'', la \emph{dualit\'e de Ringel}. En g\'en\'eral, le dual de Ringel d'un anneau quasi-h\'er\'editairement de Koszul n'est pas n\'ecessairement quasi-h\'er\'editairement de Koszul ; cependant, quand il l'est, des r\'esultats de Mazorchuk~\cite{mazorchuk} montrent que les dualit\'es de Koszul et de Ringel commutent en un sens appropri\'e, et leur composition devient une ``dualit\'e de Koszul--Ringel'' qui envoie les modules simples concentr\'es en degr\'e $0$ vers des modules \emph{basculants} de l'anneau dual.
En particulier, cette situation se produit dans le contexte de Be{\u\i}linson--Ginzburg--Soergel ; dans ce cas le dual de Ringel de l'anneau $A$ du~\S\ref{ss:bgs-intro} est \'egalement isomorphe \`a $A$, et la dualit\'e de Koszul--Ringel est donc une auto-\'equivalence de la cat\'egorie d\'eriv\'ee born\'ee des $A$-modules gradu\'es de type fini.

Cet anneau ``con\-tr\^ole'' \'egalement la cat\'egorie des faisceaux pervers sur la vari\'et\'e de drapeaux associ\'ee \`a notre alg\`ebre de Lie semi-simple, constructibles par rapport \`a la stratification de Bruhat (et \`a coefficients complexes). Une id\'ee fondamentale sugg\'er\'ee par Be{\u\i}linson--Ginzburg~\cite{bg} et mise en pratique par Bezrukavnikov--Yun~\cite{by} est que, dans ce cadre, la dualit\'e de Koszul--Ringel admet une vaste g\'en\'eralisation, qui relie les faisceaux pervers sur la vari\'et\'e de drapeaux d'un groupe de Kac--Moody aux faisceaux pervers sur la vari\'et\'e de drapeaux du groupe de Kac--Moody dual (au sens de Langlands).

\subsection{Objets \`a parit\'e et dualit\'e de Koszul formelle}

Dans le cadre de Be{\u\i}lin\-son--Ginzburg--Soergel, la ``Koszulit\'e'' est une propri\'et\'e d'un anneau, et permet la cons\-truction d'une dualit\'e de Koszul (ou de Koszul--Ringel). Une id\'ee exprim\'ee notamment dans~\cite{soergel-dualitat}\footnote{La pr\'epublication~\cite{soergel-dualitat} contient une erreur dans une preuve cruciale ; cette preuve a \'et\'e corrig\'ee, et la construction compl\'et\'ee, dans l'article~\cite{rsw}.} est que, dans le contexte de la th\'eorie des repr\'esentations sur des corps de caract\'eristique positive, on peut recontrer des \'equivalences qui ont les m\^emes propri\'et\'es formelles que les dualit\'es de Koszul, m\^eme en l'absence d'anneaux de Koszul. En d'autres termes, cette id\'ee sugg\`ere que la ``Koszulit\'e'' est une propri\'et\'e d'un \emph{foncteur} plut\^ot que d'un anneau (ou de sa cat\'egorie de modules).

Pour exprimer cette id\'ee plus pr\'ecis\'ement, nous avons besoin de la notion d'\emph{objets \`a parit\'e}. Cette notion a \'et\'e introduite et \'etudi\'ee par Juteau--Mautner--William\-son~\cite{jmw} dans le cadre des complexes de faisceaux sur des vari\'et\'es alg\'ebriques. Dans cet article nous en consid\'erons une variante dans le cadre des cat\'egories triangul\'ees munies d'une collection quasi-exceptionnelle (au sens de~\cite{bez}) ``\`a parit\'e''. Dans le cas o\`u cette collection quasi-exceptionnelle est en fait exceptionnelle (au sens de~\cite{bez:ctm}), et si de plus ses ob\-jets appartiennent au coeur de la t-structure associ\'ee, ce coeur poss\`ede une structure naturelle de \emph{cat\'egorie de plus haut poids} (c'est-\`a-dire une structure similaire \`a celle de la cat\'egorie des modules sur un anneau quasi-h\'er\'editaire), et on peut \'egalement consid\'erer les objets basculants dans cette cat\'egorie. Une \emph{dualit\'e de Koszul formelle} est (essentiellement) un foncteur entre deux cat\'egories triangul\'ees poss\'edant ces structures, qui envoie les objets \`a parit\'e de la premi\`ere cat\'egorie sur les objets basculants de la seconde et r\'eciproquement.

Dans la plupart des contextes rencontr\'es en pratique pour lesquels le corps de coefficients consid\'er\'e est de caract\'eristique $0$ (comme dans~\cite{bgs} ou~\cite{by}), les objets \`a parit\'e sont en fait des sommes directes de d\'ecal\'es d'objets simples dans le coeur de la t-structure associ\'ee \`a la collection exceptionnelle consid\'er\'ee. En ce sens les dualit\'es de Koszul--Ringel consid\'er\'ees dans~\cite{bgs,by} sont des dualit\'es de Koszul formelles. Cependant, cette notion recouvre \'egalement d'autres constructions qui ne s'\'enoncent pas naturellement en termes d'anneaux (par exemple celles de~\cite{abg}) ; de plus, dans des exemples importants pour les lesquels le corps de coefficients est de caract\'eristique positive, ces objets \`a parit\'e ne sont \emph{pas} des sommes de d\'ecal\'es d'objets simples, et il n'y a donc \emph{pas} d'anneau de Koszul ``sous-jacent'' \`a la dualit\'e de Koszul formelle consid\'er\'ee.

\subsection{Formule de caract\`eres pour les repr\'esentations basculantes des grou\-pes alg\'ebriques r\'eductifs}

Le contexte principal dans lequel de tels foncteurs apparaissent est le suivant. Consid\'erons un groupe alg\'ebrique r\'eductif connexe $\bG$ d\'efini sur un corps $\bk$ alg\'ebriquement clos de caract\'eristique $p>0$. La cat\'egorie $\Rep(\bG)$ des repr\'esentations alg\'ebriques de dimension finie de $\bG$ poss\`ede une structure naturelle de cat\'egorie de plus haut poids, et la d\'etermination des caract\`eres des repr\'esentations basculantes ind\'ecomposables associ\'ees est un probl\`eme crucial du domaine. (Il est notamment connu qu'\`a partir de ces caract\`eres on peut obtenir une formule de caract\`eres pour les repr\'esentations \emph{simples} de $\bG$.) 

Dans~\cite{rw}, G. Williamson et le second auteur ont propos\'e une formule conjecturale permettant de d\'eterminer ces caract\`eres, qui fait intervenir les \emph{$p$-polyn\^omes de Kazhdan--Lusztig} associ\'es au groupe de Weyl affine associ\'e \`a $\bG$. Ces polyn\^omes sont une ``trace combinatoire'' des objets \`a parit\'e dans une cat\'egorie de $\bk$-faisceaux sur la vari\'et\'e de drapeaux affines du groupe Langlands-dual. Cette conjecture \'enonce donc une \'egalit\'e entre des donn\'ees combinatoires provenant d'objets basculants et d'objets \`a parit\'e, ce qui sugg\`ere de la voir comme une trace combinatoire d'une dualit\'e de Koszul formelle. Cependant, sa preuve fait intervenir non pas une mais \emph{trois} foncteurs de type ``dualit\'e de Koszul formelle'' :
\begin{enumerate}
\item
une g\'en\'eralisation de la dualit\'e de Koszul de Bezru\-kavnikov--Yun au cas des $\bk$-faisceaux, d\'emontr\'ee en collaboration avec S.~Makisumi et G.~Williamson~\cite{amrw2} et reliant certains complexes de faisceaux ``d'Iwahori--Whit\-taker'' sur une vari\'et\'e de drapeaux affines \`a des complexes de faisceaux sur une Grassmannienne affine ;
\item
une adaptation au cadre modulaire d'une \'equivalence due \`a Arkhipov--Bezru\-kavnikov--Ginzburg~\cite{abg} reliant certains complexes de faisceaux sur une Grassmannienne affine \`a des complexes de faisceaux coh\'erents \'equivariants sur la r\'esolution de Springer du groupe dual, d\'emontr\'ee ind\'ependamment par L. Rider et le premier auteur~\cite{arider} et par C.~Mautner et le second auteur~\cite{mr} ;
\item
une adaptation au cadre modulaire d'une autre \'equivalence due \`a Arkhipov--Bezru\-kavnikov--Ginzburg~\cite{abg}, d\'emontr\'ee dans~\cite{prinblock} et reliant les complexes de faisceaux coh\'erents \'equivariants sur la r\'esolution de Springer du d\'ecal\'e de Frobenius $\dot\bG$ de $\bG$ \`a la cat\'egorie d\'eriv\'ee du bloc principal de $\bG$.
\end{enumerate}

\subsection{Contenu}

La Partie~\ref{pt:Koszul-classique} de cet article est d\'edi\'ee au rappel du point de vue ``classique'' sur la dualit\'e de Koszul, d\^u notamment \`a Be{\u\i}linson--Ginzburg--Soergel. En particulier, au~\S\ref{sec:anneaux-Koszul} nous rappelons les d\'efinitions et propri\'et\'es principales des anneaux de Koszul. Au~\S\ref{sec:anneaux-qh} nous consid\'erons les anneaux quasi-h\'er\'editaires, et l'interaction entre cette notion et la Koszulit\'e. Enfin, au~\S\ref{sec:dualite-Ringel} nous rappelons quelques r\'esultats concernant les relations entre la dualit\'e de Koszul et la dualit\'e de Ringel.

La Partie~\ref{pt:Koszul-formelle} est d\'edi\'ee \`a l'\'etude g\'en\'erale de la notion de dualit\'e de Koszul formelle. Au~\S\ref{sec:objets-parite} nous consid\'erons une variation de la notion de complexes \`a parit\'e de Juteau--Mautner--Williamson~\cite{jmw}. Puis au~\S\ref{sec:Koszul-formelle} nous d\'efinissons les dualit\'es de Koszul formelles, et \'enon{\c c}ons quelques propri\'et\'es qui d\'ecoulent directement de la d\'efinition.

Enfin, la Partie~\ref{pt:exemples} est d\'edi\'ee aux exemples de telles dualit\'es qui apparaissent en lien avec la th\'eorie des repr\'esentations modulaires des groupes alg\'ebriques r\'eductifs. Au~\S\ref{sec:exemples} nous pr\'esentons ces exemples. Au~\S\ref{sec:traces-comb} nous \'enon{\c c}ons quelques cons\'equences combinatoires de l'existence de ces foncteurs, dans le cadre des bases $p$-canoniques de Williamson. Enfin, au~\S\ref{sec:applications} nous pr\'esentons quelques applications de ces r\'esultats :
\begin{enumerate}
\item
\`a la preuve de la formule de caract\`eres conjectur\'ee dans~\cite{rw} ;
\item
\`a une preuve de la \emph{conjecture de Humphreys} sur les vari\'et\'es de support des $\bG$-modules basculants en grande caract\'eristique (obtenue dans~\cite{ahr}, en collaboration avec W.~Hardesty) ;
\item
\`a des formules ``de Steinberg'' et ``de Donkin'' pour la base $p$-canonique (qui n'avaient pas \'et\'e consid\'er\'ees jusque-l\`a).
\end{enumerate}

\subsection{Remerciements}

Comme expliqu\'e ci-dessus,
les r\'esultats pr\'esent\'es dans ce texte sont issus de travaux effectu\'es partiellement en collaboration avec S. Makisumi, C. Mautner, L. Rider, et G. Williamson. Nous les remercions tous les quatre pour l'aide et les id\'ees qu'ils nous ont apport\'ees. Nous remercions \'egalement la Soci\'et\'e Math\'ematique de France et les organisateurs du congr\`es de Lille pour avoir donn\'e l'occasion \`a l'un de nous de pr\'esenter ces travaux lors du 2\`eme congr\`es de la S.M.F.

L'essentiel du travail sur cet article a \'et\'e effectu\'e alors que le second auteur \'etait membre du Freiburg Institute for Advanced Studies, dans le cadre du programme ``Cohomology in Algebraic Geometry and Representation Theory'' organis\'e par A. Huber-Klawitter, S. Kebekus et W. Soergel.

\part{Dualité de Koszul classique} 
\label{pt:Koszul-classique}

\section{Anneaux de Koszul et dualit\'e}
\label{sec:anneaux-Koszul}

\subsection{D\'efinitions}
\label{ss:Koszul-def}

Soit $\bk$ un corps, et soit $A = \bigoplus_{n \ge 0} A_n$ une $\bk$-algèbre positivement graduée.  Notons $A\lgmod$ la catégorie des $A$-modules gradués de type fini, et soit $\Db(A\lgmod)$ sa catégorie dérivée bornée.  Si $M \in A\lgmod$, nous noterons $M\la n\ra$ le module obtenu par décalage de la graduation, où $(M\la n\ra)_m = M_{n+m}$.  Pour $M, N \in A\lgmod$, soit $\uExt^k_A(M,N)$ l'espace vectoriel gradué défini par
\begin{equation}\label{eqn:gr-ext}
\uExt_A^k(M,N)_n = \Hom_{\Db(A\lgmod)}(M,N\la n\ra[k]).
\end{equation}

Supposons maintenant de plus que chaque composante de $A$ est de dimension finie sur $\bk$, et que la composante $A_0$ est un anneau semi-simple. Notons
\[
\Irr_0(A)
\]
l'ensemble des classes d'isomorphisme de $A$-modules simples dont la graduation est concentrée en degré $0$. (Cet ensemble est en bijection avec celui des classes d'isomorphisme de $A_0$-modules simples; en particulier il est fini.) Notons
\[
\Proj_0(A)
\]
l'ensemble des classes d'isomorphisme de $A$-modules projectifs indécomposables dont l'unique quotient simple est concentré en degré $0$. Il y a évidemment une bijection canonique entre $\Proj_0(A)$ et $\Irr_0(A)$, envoyant $P$ sur son unique quotient simple. Il sera souvent pratique d'introduire un nouvel ensemble de paramètres $I$ avec une bijection fix\'ee
\begin{equation}\label{eqn:irr-param}
I \simto \Irr_0(A) = \Proj_0(A).
\end{equation}
Pour $i \in I$, notons
\[
L_i, \qquad\text{resp.} \qquad P_i
\]
un $A$-module gradu\'e simple, resp.~projectif, dans la classe d'isomorphisme correspondante. En particulier, il y un morphisme surjectif $P_i \twoheadrightarrow L_i$.

Considérons maintenant l'anneau
\[
E := \uExt_A^\bullet(L,L) \qquad \text{où $L = \bigoplus_{i \in I} L_i$.}
\]
Cet anneau est $\Z^2$-gradué: on a
\[
E = \bigoplus E^n_m \quad \text{où $E^n_m = \Hom_{\Db(A\lgmod)}(L,L\la m\ra[n])$.}
\]
Il est clair que la graduation en exposant (la graduation ``cohomologique'') est positive. La graduation en indice (ou ``interne'') vérifie certaines bornes aussi. En effet, nos hypoth\`eses impliquent que le module semi-simple $L$ admet une résolution projective
\[
\cdots \to P^{-1} \to P^0 \to L \to 0
\]
où chaque terme $P^k = \bigoplus_{n \in \Z} P^k_n$ est un $A$-module gradué projectif qui vérifie
\[
P^k_n = 0 \qquad \text{si $n < -k$.}
\]
Par conséquent, l'algèbre $\Z^2$-graduée $E$ vérifie
\[
E^k_n = 0 \qquad\text{si $n > -k$.}
\]

L'algèbre $A$ est dite \emph{de Koszul} si $E$ vérifie la condition
\[
E^k_n \neq 0 \quad \Rightarrow \quad n = -k.
\]
Si $A$ est de Koszul, soit
\[
A^! := E^\op
\]
l'anneau opposé de $E$, muni de la graduation d\'efinie par $A^!_n = E^n_{-n}$. Cet anneau gradu\'e v\'erifie les m\^emes hypoth\`eses que celles impos\'ees sur $A$ ; en particulier il est de Koszul (voir~\cite[Proposition~2.9.1 et Theorem~2.10.1]{bgs}), et on l'appelle  le \emph{dual} (au sens de Koszul) de $A$. Il est à noter que l'opération $A \mapsto A^!$ n'est pas une involution. Cependant, elle n'est pas loin de l'être, ce qui justifie l'utilisation du mot ``dual'' ; voir la Remarque~\ref{rmq:involution}.
On remarquera \'egalement que
\[
A^!_0 = \bigoplus_{i \in I} \End_{A_0}(L_i)^{\op}
\]
est un anneau semi-simple, dont les classes d'isomorphisme de modules simples sont en bijection canonique avec $I$.

\begin{rmq}\phantomsection
\label{rmq:A0-ss}
\begin{enumerate}
\item
Soit $A$ comme ci-dessus (positivement gradu\'ee, avec composantes gradu\'ees de dimension finie), mais ne supposons pas a priori que $A_0$ est semi-simple. D\'efinissons encore $I$ comme l'ensemble des classes d'isomorphisme de $A_0$-modules simples, ou de fa{\c c}on \'equivalente de $A$-modules simples concentr\'es en degr\'e $0$. Alors la condition que $\Hom_{\Db(A\lgmod)}(L_i, L_j \langle n \rangle [k])=0$ pour $n \neq -k$ implique en particulier que
\[
\Ext^1_{A_0}(L_i,L_j) = \Hom_{\Db(A\lgmod)}(L_i, L_j [1])=0
\]
pour tous $i,j \in I$. Donc $A_0$ est automatiquement un anneau semi-simple.
\item
\label{it:dual-quad}
Un anneau de Koszul est en particulier un anneau quadratique (voir~\cite[Corollary~2.3.3]{bgs}), et l'anneau $A^!$ est canoniquement isomorphe \`a l'anneau quadratique dual ; voir~\cite[Theorem~2.10.1]{bgs}.
\end{enumerate}
\end{rmq}

\begin{ex}\label{ex:sym-ext}
Soit $V$ un espace vectoriel de dimension finie sur $\bk$, et soit $A = \Sym(V)$ l'algèbre symétrique associ\'ee. Munissons-la de la graduation ``naturelle'', où $A_n = \Sym^n(V)$.  Il est bien connu que l'algèbre $E = \Ext_A^\bullet(\bk,\bk)$ s'identifie à l'algèbre extérieure de l'espace vectoriel dual, et que $E^k_n = 0$ si $n \ne -k$. Cet anneau est isomorphe à son anneau opposé, et on a donc $A^! \cong \bigwedge V^*$.

Il est également possible d'échanger les rôles des deux algèbres dans cet exemple: voir l'Exemple~\ref{ex:sym-ext2}.
\end{ex}

\subsection{Dualit\'e de Koszul}

Pour un anneau gradué $A$, on notera
\[
\Db_\lf(A\lgmod) \qquad\text{resp.}\qquad \Db_\pf(A\lgmod)
\]
la sous-catégorie triangul\'ee pleine de $\Db(A\lgmod)$ engendrée par les modules de longueur finie, resp.~par les modules projectifs de type fini. (On appelle les objets de $\Db_\pf(A\lgmod)$ les \emph{complexes parfaits}.)   L'énoncé suivant explique, dans le cas o\`u $A$ est de Koszul, la relation entre les $A$-modules gradu\'es et les $A^!$-modules gradu\'es.

\begin{prop}
\label{prop:koszul-pre}
Soit $A$ comme au~\S{\rm \ref{ss:Koszul-def}}, et supposons $A$ de Koszul.
Il existe une équivalence de catégories triangul\'ees
\[
\kappa: \Db_\lf(A\lgmod) \simto \Db_\pf(A^!\lgmod)
\]
telle que
\[
\kappa(M\la 1\ra) \cong \kappa(M)\la -1\ra[1]
\]
pour tout $M$ dans $\Db_\lf(A\lgmod)$.
De plus, en identifiant $\Irr_0(A)$ est $\Irr_0(A^!)$ comme ci-dessus, et en utilisant~\eqref{eqn:irr-param} (pour l'alg\`ebre $A^!$) on a
\[
\kappa(L_i) \cong P^!_i
\]
pour tout $i \in I$.
\end{prop}

Pour une preuve de cette proposition, voir~\cite[Theorem~2.12.1]{bgs}. (En fait, ce dernier résultat est beaucoup plus général que notre Proposition~\ref{prop:koszul-pre}: les hypothèses sur les anneaux ainsi que les conditions ``lf'' et ``pf'' ont été affaiblies. Par contre, la version ci-dessus est plus facile à énoncer et suffira pour nos buts.)

Sous certaines hypothèses de finitude supplémentaires, on peut supprimer les indices ``lf'' et ``pf'' dans la Proposition~\ref{prop:koszul-pre}. Si $A$ est de dimension finie, alors la catégorie $A\lgmod$ a assez d'injectifs.  Notons
\[
\Inj_0(A)
\]
l'ensemble des classes d'isomorphisme de $A$-modules gradu\'es injectifs indécomposables dont l'unique sous-module simple est concentré en degré $0$.  Il y a encore une fois une bijection canonique $\Irr_0(A) = \Inj_0(A)$.  Pour $i \in I$, notons $J_i$ le module injectif correspondant. On a alors le résultat suivant.

\begin{thm}[Dualité de Koszul, 1ère version]
\label{thm:koszul1}
Soit $A$ comme au~\S{\rm \ref{ss:Koszul-def}}, suppos\'ee de plus
de dimension finie sur $\bk$, de dimension globale finie, et de Koszul.  Alors le dual $A^!$ est de dimension finie sur $\bk$, de dimension globale finie (et de Koszul), et il existe une équivalence de catégories triangul\'ees
\[
\kappa: \Db(A\lgmod) \simto \Db(A^!\lgmod)
\]
telle que
\[
\kappa(M\la 1\ra) \cong \kappa(M)\la -1\ra[1]
\]
pour tout $M$ dans $\Db(A\lgmod)$.
Via l'identification $\Irr_0(A) \cong \Irr_0(A^!)$ consid\'er\'ee ci-dessus, on a
\[
\kappa(L_i) \cong P_i^!, \qquad
\kappa(J_i) \cong L_i^!
\]
pour tout $i \in I$.
\end{thm}

\begin{proof}
L'anneau $A^!$ est de dimension finie car $A$ est de dimension globale finie. Puisque tout $A$-module est de longueur finie, le foncteur d'inclusion $\Db_\lf(A\lgmod) \to \Db(A\lgmod)$ est une équivalence de catégories. On a alors le foncteur $\kappa: \Db(A\lgmod) \to \Db_\pf(A^!\lgmod)$ de la Proposition~\ref{prop:koszul-pre}.  En considérant les espaces
\[
\Hom_{\Db(A\lgmod)}(L_i,J_j \langle n \rangle [k]) \cong \Hom_{\Db(A^!\lgmod)}(P_i^!, \kappa(J_j) \langle -n \rangle[n+k]),
\]
on peut v\'erifier que $\kappa(J_j)$ est l'unique quotient simple de $P_i^!$. En particulier, tout $A^!$-module simple est un complexe parfait, ce qui implique que $A^!$ est de dimension globale finie.
Cette propri\'et\'e nous assure que l'inclusion $\Db_\pf(A^!\lgmod) \to \Db(A^!\lgmod)$ est une équivalence de catégories, et on peut alors consid\'erer $\kappa$ comme une \'equivalence $\Db(A\lgmod) \simto \Db(A^!\lgmod)$.
\end{proof}

\begin{rmq}\phantomsection
\label{rmq:involution}
\begin{enumerate}
\item
Pour une version plus ``cat\'egorique'' de cette construction, on pourra consulter~\cite[Theorem~2.4]{ar:kdsf}.
\item
\label{it:hyp}
Il est clair d'apr\`es la preuve du th\'eor\`eme que les hypoth\`eses r\'eellement n\'ecessaires sont que $A$ est de dimension finie et que $A^!$ est de dimension globale finie. Cependant, imposer les deux conditions sur $A$ permet d'obtenir une version plus sym\'etrique de l'\'enonc\'e.
\item
On peut remarquer que les rôles des modules projectifs et injectifs dans le Théorème~\ref{thm:koszul1} n'est pas symétrique, et se demander s'il existe une version modifiée de cet énoncé qui échange leurs rôles. C'est bien le cas : si $A$ est de Koszul, son anneau opposé l'est aussi (voir~\cite[Proposition~2.2.1]{bgs}), et on peut poser
\[
{}^!\!A := ((A^\op)^!)^\op.
\]
Avec cette notation,~\cite[Theorem~2.10.2]{bgs} implique qu'il existe des isomorphismes canoniques
\[
A \cong {}^!(A^!) \cong({}^!\!A)^!.
\]
Si $A$ vérifie de plus les hypothèses du Théorème~\ref{thm:koszul1}, alors il existe une équivalence de catégories triangul\'ees $\kappa': \Db(A\lgmod) \simto \Db({}^!A\lgmod)$ qui envoie les modules projectifs (resp.~simples) sur des modules simples (resp.~injectifs).
\end{enumerate}
\end{rmq}

\begin{ex}\label{ex:sym-ext2}
Soit $A = \bigwedge V^*$ l'algèbre extérieure du dual d'un espace vectoriel $V$.  Comme on l'a vu dans l'Exemple~\ref{ex:sym-ext}, on a ${}^!A \cong A^! \cong \Sym(V)$. (Ici on a utilisé le fait que les algèbres symétriques et extérieures sont toutes les deux isomorphes, de fa{\c c}on naturelle, à leurs anneaux opposés.) Cet exemple ne rentre pas dans le cadre du Théorème~\ref{thm:koszul1} car $A$ n'est pas de dimension globale finie ; mais sa preuve reste valable dans ce cas particulier\footnote{En fait, ce cas particulier est justement celui qui a motiv\'e le d\'eveloppement de cette th\'eorie.} (voir la Remarque~\ref{rmq:involution}\eqref{it:hyp}),\ et on obtient une equivalence de catégories triangul\'ees
\[
\kappa: \Db(\textstyle \bigwedge V^* \lgmod) \simto \Db(\Sym(V) \lgmod).
\]
Rappelons maintenent que l'anneau $A$ est auto-injectif: ses modules injectifs coïncident avec ses modules projectifs. Le foncteur $\kappa$ envoie alors
\begin{align*}
L_0 &\mapsto P_0^!, \\
P_0\la \dim V\ra &\mapsto L_0^!,
\end{align*}
où $L_0$ est l'unique module simple concentr\'e en degré $0$ (et de m\^eme pour $L_0^!$). La dualité de Koszul pour les algèbres symétrique et extérieure jouit donc d'une symétrie supplémentaire, qui ne vaut pas pour les anneaux de Koszul g\'en\'eraux.
\end{ex}

\subsection{Quelques exemples}
\label{ss:carquois}

Soit $R$ l'algèbre des chemins (à coefficients dans $\bk$) du carquois suivant :
\begin{equation}
\label{eqn:carquois-ex}
\begin{tikzcd}
\hbox{\textcircled{\raisebox{-0.9pt}{1}}} \ar[r, shift left=0.5ex, "\alpha"] &
\hbox{\textcircled{\raisebox{-0.9pt}{2}}} \ar[l, shift left=0.5ex, "\beta"]
\end{tikzcd}
\end{equation}
De même, soit
$R^*$ l'algèbre des chemins du carquois obtenu en inversant les flèches:
\[
\begin{tikzcd}
\hbox{\textcircled{\raisebox{-0.9pt}{1}}} \ar[r, shift right=0.5ex, "\beta^*"'] &
\hbox{\textcircled{\raisebox{-0.9pt}{2}}} \ar[l, shift right=0.5ex, "\alpha^*"']
\end{tikzcd}
\]
(La concat\'enation des chemins sera not\'ee avec les m\^emes conventions que pour les compositions d'applications.)
Munissons et $R$ et $R^*$ de la graduation donnée par la longueur des chemins. (En particulier, les idempotents correspondant aux sommets sont de degré $0$, et les éléments $\alpha$, $\beta$, $\alpha^*$, et $\beta^*$ correspondant aux flèches sont tous de degré $1$). Pour rendre explicite cette graduation, nous décrirons les représentations du carquois comme des espaces vectoriels gradués, de fa{\c c}on compatible avec 
les degr\'es des
flèches. Par exemple, les diagrammes suivants correspondent \`a des $R$-modules gradués :
\[
\begin{tikzcd}
\bk \ar[r, shift left=0.5ex, "\id"] &
\bk\la -1\ra \ar[l, shift left=0.5ex, "0"]
\end{tikzcd}
\qquad
\begin{tikzcd}
\bk\la -1\ra \ar[r, shift left=0.5ex, "0"] &
\bk. \ar[l, shift left=0.5ex, "\id"]
\end{tikzcd}
\]

Les anneaux de la colonne de gauche du tableau ci-dessous sont de Koszul, et leurs duaux de Koszul sont les anneaux qui figurent dans la colonne de droite :
\[
\begin{array}{ccc}
A && A^! \\
\cline{1-1}  \cline{3-3}
R && R^*/(\alpha^*\beta^*, \beta^*\alpha^*) \\
R/(\alpha\beta) && R^*/(\alpha^*\beta^*) \\
R/(\alpha\beta, \beta\alpha) && R^*
\end{array}
\]
En fait, pour v\'erifier cela on peut proc\'eder de la fa{\c c}on suivante. On \'ecrit $R$ comme l'alg\`ebre tensorielle du bimodule sur $R_0 = \bk e_1 \oplus \bk e_2$ donn\'e par $V=\bk \alpha \oplus \bk \beta$ (avec $e_1 \alpha=0$, $e_2 \alpha=\alpha$, etc.). L'alg\`ebre tensorielle d'un bimodule $V$ sur un anneau semisimple est toujours de Koszul, ce qui r\`egle la premi\`ere et la troisi\`eme ligne de ce tableau. Pour la deuxi\`eme ligne, on voit que $R/(\alpha \beta)$ est l'anneau quadratique d\'efini par $V$ et les relations $\bk \cdot (\alpha \otimes \beta) \subset V \otimes_{R_0} V$. S'il est de Koszul, son anneau dual est l'anneau quadratique dual (voir la Remarque~\ref{rmq:A0-ss}\eqref{it:dual-quad}). Avec les conventions adopt\'ee dans~\cite{bgs}, cet anneau quadratique dual est d\'efini par le $R_0$-bimodule $V^* = \Hom_{R_0}(V,R_0)=\bk \alpha^* \oplus \bk \beta^*$ (o\`u $\alpha^*(\alpha)=e_2$, $\alpha^*(\beta)=0=\beta^*(\alpha)$, $\beta^*(\beta)=e_1$) et les relations $\bk \cdot (\alpha^* \otimes \beta^*) \subset V^* \otimes_{R_0} V^*$. Il ne reste donc plus qu'\`a voir que $R/(\alpha\beta)$ est de Koszul.

Un module gradu\'e sur cet anneau peut être vu comme une représentation gradu\'ee du carquois~\eqref{eqn:carquois-ex} qui vérifie la relation $\alpha\beta=0$. Nous allons donc écrire les modules dans le langage des carquois. Les $R/(\alpha\beta)$-modules simples en degr\'e $0$ sont :
\[
L_1 = \begin{tikzcd}
\bk \ar[r, shift left=0.5ex, "0"] &
0 \ar[l, shift left=0.5ex, "0"]
\end{tikzcd},
\qquad
L_2 = \begin{tikzcd}
0 \ar[r, shift left=0.5ex, "0"] &
\bk \ar[l, shift left=0.5ex, "0"]
\end{tikzcd}
\]
Leurs couvertures projectives et enveloppes injectives sont:
\[
\begin{gathered}
P_1 = I_1\la -2\ra = 
\begin{tikzcd}[ampersand replacement=\&]
(\bk\la -2\ra \oplus \bk)  \ar[r, shift left=0.5ex, "{\left[\begin{smallmatrix} 0 & 1\end{smallmatrix}\right]}"]\&
\bk\la -1\ra \ar[l, shift left=0.5ex, "{\left[\begin{smallmatrix} 1 \\ 0\end{smallmatrix}\right]}"]
\end{tikzcd},
\\
P_2 = \begin{tikzcd}
\bk\la -1\ra \ar[r, shift left=0.5ex, "0"] &
\bk \ar[l, shift left=0.5ex, "\id"]
\end{tikzcd},
\qquad
I_2 = \begin{tikzcd}
\bk\la 1\ra \ar[r, shift left=0.5ex, "\id"] &
\bk \ar[l, shift left=0.5ex, "0"]
\end{tikzcd}.
\end{gathered}
\]
En particulier, on a des r\'esolutions projectives des objets simples de la forme
\begin{equation}
\label{eqn:res-proj}
0 \to P_2 \la -1 \ra \to P_1 \to L_1 \to 0 \ \text{ et } \ 0 \to P_2 \la -2 \ra \to P_1 \la -1 \ra \to P_2 \to L_2 \to 0,
\end{equation}
qui montrent que $R$ est bien de Koszul.

On peut comprendre plus concr\`etement la relation $\alpha^* \beta^*=0$ dans l'anneau dual de la fa{\c c}on suivante.
En utilisant les r\'esolutions projectives ci-dessus on peut v\'erifier que $\uExt^1(L_1,L_2)$ et $\uExt^1(L_2,L_1)$ sont de dimension~$1$; plus pr\'ecis\'ement, il existe des suites exactes courtes non scindées
\begin{gather*}
0 \to L_2\la -1\ra \to I_2\la -1\ra \to L_1 \to 0, \\
0 \to L_1\la -1\ra \to P_2 \to L_2 \to 0.
\end{gather*}
Soient $\alpha^* \in \Ext^1(L_1, L_2\la -1\ra)$  et $\beta^* \in \Ext^1(L_2, L_1\la -1\ra)$ les éléments correspondants à ces suites exactes courtes.  Le produit $\beta^* \cdot \alpha^*$ au sens de Yoneda correspond à la suite exacte
\begin{equation}\label{eqn:ext2-carquois}
0 \to L_1\la -2\ra \to P_2\la -1\ra \to I_2\la -1\ra \to L_1 \to 0.
\end{equation}
Il existe une suite de sous-modules $L_1\la -2\ra \subset P_2\la -1\ra \subset P_1$ telle que $P_1/L_1\la -2\ra \cong I_2\la -1\ra$ et $P_1/P_2\la -1\ra \cong L_1$.  Autrement dit, la classe de la suite~\eqref{eqn:ext2-carquois} dans le groupe $\Ext^2(L_1,L_1\la -2\ra)$ est nulle.

\section{Anneaux quasi-héréditaires et catégories de plus haut poids}
\label{sec:anneaux-qh}

\subsection{Anneaux quasi-héréditaires}
\label{ss:qhered}

Comme au~\S\ref{ss:Koszul-def},
soit $A = \bigoplus_{n \ge 0} A_n$ une $\bk$-algèbre positivement graduée, qu'on supposera ici de dimension finie pour simplifier.
Soit $I$ un ensemble paramétrant les $A_0$-modules simples (ou, de fa{\c c}on \'equivalente, les $A$-modules simples concentrés en degré $0$).  Supposons que $I$ est muni d'un ordre partiel $\preceq$.
On appellera ``id\'eal'' de $I$ un sous-ensemble $J \subset I$ tel que si $i \in I$, $j \in J$ et $i \preceq j$ alors $i \in J$. Pour tout id\'eal $T \subset I$, on notera $A\lgmod_T$ la sous-catégorie pleine de $A\lgmod$ form\'ee par les modules dont tous les facteurs de composition sont de la forme $L_i\la n\ra$ avec $i \in T$ et $n \in \Z$. On consid\'erera en particulier cette sous-cat\'egorie dans le cas o\`u $T=\{\preceq i\}:=\{j \in I \mid j \preceq i\}$, et dans celui o\`u $T=\{\prec i\}:=\{j \in I \mid j \prec i\}$.

L'alg\`ebre $A$ est dite \emph{quasi-héréditaire graduée}
si pour tout $i \in I$ il existe des objets $\Delta_i, \nabla_i \in A\lgmod_{\{\preceq i\}}$ et des morphismes $\Delta_i \to L_i \to \nabla_i$ qui vérifient les conditions suivantes:
\begin{enumerate}
\item 
\label{it:qh-1}
pour tout module simple $L_i$, on a $\uEnd(L_i) \cong \bk$ ;
\item
\label{it:qh-2}
pour tout id\'eal $T \subset I$ et tout $i \in T$ maximal, $\Delta_i$ est la couverture projective de $L_i$ dans $A\lgmod_T$, et $\nabla_i$ est son enveloppe injective ;
\item 
\label{it:qh-3}
le noyau de $\Delta_i \to L_i$ et le conoyau de $L_i \to \nabla_i$ appartiennent à $A\lgmod_{\{\prec i\}}$ ;
\item 
\label{it:qh-4}
pour tout $i,j \in I$, on a $\uExt^2_A(\Delta_i, \nabla_j) = 0$.
\end{enumerate}

Les modules de la forme $\Delta_i\la n\ra$, resp.~$\nabla_i\la n\ra$, sont dits \emph{standards}, resp.~\emph{costandards}. 
Notons que toute alg\`ebre quasi-héréditaire est de dimension globale finie ; voir~\cite[Corollary~3.2.2]{bgs}. Notons \'egalement qu'on a automatiquement
\begin{equation}
\label{eqn:qh-Hom-D-N}
\Hom_{\Db(A\lgmod)}(\Delta_i, \nabla_j \langle n \rangle [k]) \cong \begin{cases}
\bk & \text{si $i=j$ et $n=k=0$ ;}\\
0  & \text{sinon,}
\end{cases}
\end{equation}
voir par exemple~\cite[Corollary~7.6]{riche-hab}.

\begin{rmq}
\begin{enumerate}
\item
Les objets $\Delta_i$ et $\nabla_i$ sont bien s\^ur uniques (\`a isomorphisme pr\`es) s'ils existent.
\item
La notion d'alg\`ebre quasi-h\'er\'editaire est due \`a Cline--Parshall--Scott ; voir notamment~\cite{cps}. La d\'efinition que nous consid\'erons ici est diff\'erente de celle \'etudi\'ee dans~\cite{cps} ; cependant, 
il est possible de déduire de~\cite[Theorem~3.6]{cps} et de~\cite[Theorem~3.2.1]{bgs} que les deux notions de quasi-hérédité coïncident. La définition adopt\'ee ci-dessus (qui est sugg\'er\'ee dans~\cite{bgs}) a l'avantage d'être ``catégorique''; c'est-à-dire qu'elle s'adapte bien au cadre des catégories abéliennes plus générales, qui ne sont pas d\'efinies naturellement comme une cat\'egorie de modules sur un anneau.
\item
Notons bien que dans cette partie les alg\`ebres que nous consid\'erons sont \emph{positivement} gradu\'ees. On peut bien s\^ur \'etudier des alg\`ebres quasi-h\'er\'editaires gradu\'ees par $\Z$ (aucun des axiomes ci-dessus n'utilise nos restrictions sur la graduation), mais nous ne le ferons pas ici.
\item
Ci-dessous nous consid\'ererons \'egalement des alg\`ebres quasi-h\'er\'editaires \emph{non gradu\'ees}. Pour les d\'efinir, il suffit de remplacer les $\uHom$ et les $\uExt$ dans les axiomes ci-dessus par les $\Hom$ et $\Ext$ ordinaires, et d'omettre les décalages de graduation.
\end{enumerate}
\end{rmq}

\subsection{Anneaux quasi-h\'er\'editairement de Koszul et dualit\'e de Koszul}
\label{ss:anneau-qh-koszul}

Soit $A$ une $\bk$-algèbre 
quasi-héréditaire gradu\'ee (avec les conventions du~\S\ref{ss:qhered}). On dira que $A$ est \emph{quasi-héréditairement de Koszul} si pour tous $i, j \in I$, on a
\begin{equation}\label{eqn:qh-koszul-defn}
\Hom_{\Db(A\lgmod)}(\Delta_i,L_j\la n\ra[k]) = \Hom_{\Db(A\lgmod)}(L_i, \nabla_j\la n\ra[k]) = 0
\end{equation}
si $n \ne -k$.

\begin{rmq}
Dans la terminologie de~\cite[Definition~2.14.1]{bgs}, la condition que $\Hom_{\Db(A\lgmod)}(\Delta_i,L_j\la n\ra[k]) = 0$ si $n +k \neq 0$ peut s'interpr\'eter comme disant que les modules gradu\'es $\Delta_i$ sont des \emph{modules de Koszul}.
\end{rmq}

Cette terminologie est justifi\'ee par le r\'esultat suivant. (Pour une preuve dans le langage de modules sur un anneau, voir~\cite{adl}.  Pour une preuve dans un langage catégorique plus proche de celui utilis\'e ici, voir~\cite[Proposition~2.13]{modrap3} et la Remarque~\ref{rmq:A0-ss}.)

\begin{prop}
Si $A$ est une $\bk$-alg\`ebre quasi-héréditairement de Koszul, alors elle est de Koszul.
\end{prop}

\begin{rmq}
Cette proposition implique en particulier que si $A$ est une $\bk$-alg\`ebre quasi-héréditairement de Koszul alors l'anneau $A_0$ est semi-simple.
\end{rmq}

Consid\'erons encore une $\bk$-alg\`ebre quasi-h\'er\'editaire gradu\'ee $A$ comme au~\S\ref{ss:qhered}.
Si $M$ est un $A$-module gradu\'e, une \emph{filtration standard}, resp.~\emph{costandard}, est une filtration finie $0 = M_0 \subset M_1 \subset \cdots \subset M_k = M$ telle que tout sous-quotient $M_i/M_{i-1}$ est un module standard, resp.~costandard.
Une propriété très importante des anneaux quasi-héréditaires est que tout module projectif admet une filtration standard (voir~\cite[Theorem~3.2.1]{bgs}). En un sens, les modules standards servent donc à ``interpoler'' entre les modules simples et les modules projectifs. De même, les modules costandards interpolent entre les modules simples et les modules injectifs. Il est donc naturel de se demander comment les modules (co)standards se comportent par rapport au foncteur du Théorème~\ref{thm:koszul1}.

Le r\'esultat suivant est d\'emontr\'e dans~\cite{adl}. Ci-dessous nous en donnons une autre preuve.

\begin{thm}[Dualité de Koszul, 2ème version]
\label{thm:koszul2}
Soit $A$ une alg\`ebre quasi-h\'er\'editaire gradu\'ee comme au~\S{\rm \ref{ss:qhered}}
(pour un ensemble partiellement ordonné $(I,\preceq)$), et supposons que $A$ est quasi-héréditairement de Koszul. Alors le dual de Koszul $A^!$ est 
quasi-héréditaire gradu\'e (par rapport à l'ensemble partiellement ordonné $(I, \preceq^\op)$) et quasi-héréditairement de Koszul. Il existe une équivalence de catégories triangul\'ees
\[
\kappa: \Db(A\lgmod) \simto \Db(A^!\lgmod)
\]
qui vérifie
\[
\kappa(M\la 1\ra) \cong \kappa(M)\la -1\ra[1]
\]
pour tout $M$ dans $\Db(A\lgmod)$, et telle que
\[
\kappa(L_i) \cong P_i^!, \qquad \kappa(\nabla_i) \cong \Delta_i^!, \qquad \kappa(J_i) \cong L_i^!
\]
pour tout $i \in I$.
\end{thm}

Pour d\'emontrer ce r\'esultat, nous aurons besoin du lemme suivant. Ici, pour $\mathscr{D}$ une cat\'egorie triangul\'ee et $(X_\alpha : \alpha \in A)$ une collection d'objets de $\mathscr{D}$, nous noterons $\langle X_\alpha : \alpha \in A \rangle_\ext$ la sous-cat\'egorie de $\mathscr{D}$ engendr\'ee par extensions par les objets $(X_\alpha : \alpha \in A)$, c'est-\`a-dire la plus petite sous-cat\'egorie pleine $\mathscr{D}'$ de $\mathscr{D}$ contenant ces objets et telle que si $M_1,M_3 \in \mathscr{D}'$ et si $M_1 \to M_2 \to M_3 \xrightarrow{[1]}$ est un triangle distingu\'e, alors $M_2$ appartient \`a $\mathscr{D}'$.

\begin{lem}
\label{lem:!-pure}
Soit $A$ une $\bk$-algèbre 
quasi-héréditaire gradu\'ee comme au~\S{\rm \ref{ss:qhered}} (pour un ensemble partiellement ordonné $(I,\preceq)$).
Pour $M$ dans $\Db(A\lgmod)$, les deux conditions suivantes 
sont équivalentes:
\begin{enumerate}
\item 
\label{it:lem-!-pure-1}
pour tout $i \in I$, on a $\Hom_{\Db(A\lgmod)}(\Delta_i\la n\ra[m], M) = 0$ si $n \ne -m$ ;
\item 
\label{it:lem-!-pure-2}
$M$ appartient à $\langle \nabla_j\la k\ra[-k] : j \in I, \, k \in \Z \rangle_\ext$.
\end{enumerate}
\end{lem}

\begin{proof}
Pour d\'emontrer que~\eqref{it:lem-!-pure-2} implique~\eqref{it:lem-!-pure-1}, on peut supposer que $M = \nabla_j\la k\ra[-k]$. Dans ce cas, le r\'esultat voulu est une cons\'equence de~\eqref{eqn:qh-Hom-D-N}.

Supposons maintenant que $M$ v\'erifie la condition~\eqref{it:lem-!-pure-1}, et
soit $T \subset I$ le plus petit idéal tel que $M$ appartient à la sous-catégorie triangul\'ee de $\Db(A\lgmod)$ engendrée par les objets $\{ L_i\la n\ra : i \in T, \, n \in \Z \}$. Nous procéderons par récurrence sur le cardinal de $T$. Si $T$ est vide, alors $M = 0$, et il n'y a rien à prouver.  Sinon, soit $i$ un élément maximal de $T$.  Le formalisme de recollement (tel que d\'evelopp\'e dans~\cite[Lemma~4]{bez} ; voir aussi~\cite[\S 7.3]{riche-hab}) et notre condition sur $M$ permettent de construire un morphisme canonique
\[
\phi: M \to \bigoplus_{n \in \Z} \Hom(\Delta_i \la n\ra[-n],M) \otimes_\bk \nabla_i\la n\ra[-n]
\]
dont le coc\^one $M'$ appartient \`a la sous-cat\'egorie triangul\'ee engendr\'ee par les objets $\{ L_j\la n\ra : j \in T \smallsetminus \{i\}, \, n \in \Z \}$. En utilisant~\eqref{eqn:qh-Hom-D-N}, ll n'est pas difficile de voir que $M'$ v\'erifie encore la condition~\eqref{it:lem-!-pure-1}. Par r\'ecurrence il v\'erifie \'egalement~\eqref{it:lem-!-pure-2}, et on en d\'eduit la m\^eme condition pour $M$.
\end{proof}

\begin{proof}[D\'emonstration du Th\'eor\`eme~{\rm \ref{thm:koszul2}}]
Définissons $A^!$ et $\kappa$ comme dans le Théorème~\ref{thm:koszul1}. Pour tout $i \in I$, notons $A^!\lgmod_{\{\preceq^\op i\}}$ la sous-catégorie pleine de $A^!\lgmod$ dont les objets sont les modules dont tous les facteurs de composition sont de la forme $L_j^!\la n\ra$ avec $j \preceq^\op i$. 

Pour commencer,
nous allons montrer que $A^!$ vérifie les axiomes~\eqref{it:qh-1}--\eqref{it:qh-4} d\'efinissant les anneaux quasi-hérédi\-taires, et que $\kappa$ envoie les $A$-modules standards sur les $A^!$-modules costandards.  

Pour d\'emontrer~\eqref{it:qh-1}, on remarque que
puisque $\kappa(J_i) \cong L_i^!$, pour tout $n \in \Z$ on a
\[
\Hom_{A^!\lgmod}(L_i^!, L_i^! \langle n \rangle) \cong \Hom_{\Db(A\lgmod)}(J_i, J_i \langle -n \rangle [n]). 
\]
Bien s\^ur, ce groupe est nul si $n \neq 0$. Et pour $n=0$, on remarque que
\[
\dim(\Hom_{A^!\lgmod}(J_i,J_i))=[J_i : L_i].
\]
Mais, puisque $L_i$ est le socle de $J_i$, en utilisant le fait que $A^!$ est concentr\'e en degr\'es positifs ou nuls et que $A^!_0$ est semi-simple, on voit que le conoyau de l'injection $L_i \hookrightarrow J_i$ est concentr\'e en degr\'es strictement n\'egatifs;
il ne peut donc pas poss\'eder $L_i$ comme facteur de composition, ce qui implique que $[J_i : L_i]=1$.

Ensuite, définissons $\Delta_i^!$ par
\[
\Delta_i^! := \kappa(\nabla_i).
\]
A priori, il s'agit d'un objet de $\Db(A^!\lgmod)$. Pour montrer qu'il appartient en fait \`a $A^!\lgmod_{\{\preceq^\op i\}}$, il suffit de montrer que
\[
\Hom_{\Db(A^!\lgmod)}(P^!_j\la n\ra, \Delta_i^![k]) = 0 \qquad\text{sauf si $k = 0$ et $j \preceq^\op i$.}
\]
En appliquant $\kappa^{-1}$, on transforme cette condition en la version équivalente suivante :
\[
\Hom_{\Db(A\lgmod)}(L_j, \nabla_i\la n\ra[k-n]) = 0 \qquad\text{sauf si $k = 0$ et $j \succeq i$.}
\]
Le fait que cet espace est nul si $k \ne 0$ découle du fait que $A$ est quasi-héréditairement de Koszul.  Si $k = 0$ et $j \not\succeq i$, alors $L_j$ appartient à $A\lgmod_T$, o\`u $T$ est l'id\'eal $I \smallsetminus \{k \in I \mid k \succ i\}$, dans lequel $i$ est maximal.
L'annulation voulue découle alors du fait que (par hypoth\`ese) $\nabla_i$ est l'enveloppe injective de $L_i$ dans $A\lgmod_T$, puisque le foncteur naturel
\[
\Db(A\lgmod_T) \to \Db(A\lgmod)
\]
est pleinement fid\`ele (voir~\cite[Proposition~7.9(1)]{riche-hab}).

Nous avons établi que $\Delta_i^! \in A^!\lgmod_{\{\preceq^\op i\}}$. Pour montrer que, pour tout id\'eal $T$ (pour $\preceq^\op$) dans lequel $i$ est maximal, $\Delta_i^!$ est la couverture projective de $L^!_i$ dans $A^!\lgmod_T$, il suffit de montrer que pour tout $j \in I$ on a
\[
\Hom_{A^!\lgmod}(\Delta_i^!, L^!_j\la n\ra) \cong
\begin{cases}
\bk & \text{si $j = i$ et $n = 0$,} \\
0 & \text{sinon,}
\end{cases}
\]
et que si $j \not\preceq i$ on a
\[
\Hom_{\Db(A^!\lgmod)}(\Delta_i^!, L^!_j\la n\ra[1]) = 0.
\]
En appliquant $\kappa^{-1}$, on se ram\`ene \`a des arguments similaires \`a ceux utilis\'es ci-dessus.

Soit $K^!_i$ le noyau de la surjection $\Delta^!_i \to L^!_i$. On veut montrer que $K_i^!$ appartient \`a $A^!\lgmod_{\{\prec^\op i\}}$, c'est-\`a-dire que $\Hom(P^!_j, K^!_i\la n\ra) = 0$ si $i \not\prec j$.  L'objet $\kappa^{-1}(K^!_i)$ s'identifie au décalage par $[-1]$ du conoyau de l'injection $\nabla_i \hookrightarrow J_i$. Ce dernier objet admet une filtration costandard, dans laquelle les objets $\nabla_k\la m\ra$ qui interviennent vérifient $k \succ i$.  On peut en déduire que $\Hom(L_j, \kappa^{-1}(K^!_i)\la n\ra[k]) = 0$ si $j \not\succ i$, et ainsi que $K^!_i \in A^!\lgmod_{\{\prec^\op i\}}$, comme souhait\'e. On a donc v\'erifi\'e les axiomes~\eqref{it:qh-2} et~\eqref{it:qh-3} pour les objets $\Delta_i^!$.

Construisons maintenant les objets $\nabla_i^!$. Puisque $A^!$ est de dimension finie, tout module simple $L^!_i$ admet une enveloppe injective $J^!_i$. Si $T \subset I$ est un id\'eal (pour $\preceq^\op$) dans lequel $i$ est maximal, notons $\nabla^!_T$ le plus grand sous-objet de $J_i^!$ qui appartient \`a $A^!\lgmod_T$. Alors $\nabla^!_T$ est l'enveloppe injective de $L_i^!$ dans $A^!\lgmod_T$. Nous affirmons que cet objet ne d\'epend par de $T$, c'est-\`a-dire que si $\{\preceq^\op i\} := \{k \in I \mid k \preceq^\op i\}$ alors l'inclusion $\nabla^!_{i,\{\preceq^\op i\}} \hookrightarrow \nabla^!_{i,T}$ est une \'egalit\'e. Nous proc\'ederons par r\'ecurrence descendante sur le cardinal de $T$. Pour cela choisissons $j \in T$ maximal et diff\'erent de $i$, et posons $T':=T \smallsetminus \{j\}$. Si $\nabla^!_{i,T}$ admet $L_j \langle n \rangle$ comme facteur de composition pour un $n \in \Z$, d'apr\`es les propri\'et\'es des objets standards (d\'emontr\'ees ci-dessus) on doit avoir
\[
\Hom_{A\lgmod}(\Delta^!_j \langle n \rangle, \nabla^!_{i,T}) \neq 0.
\]
Par d\'efinition de $\nabla^!_{i,T}$, cet espace s'identifie \`a $\Hom_{A\lgmod}(\Delta^!_j \langle n \rangle, J^!_{i})$, et donc $L_i^!$ est un facteur de composition de $\Delta^!_j \langle n \rangle$. Ceci implique que $i \preceq^\op j$, ce qui est absurde puisque $i$ est maximal dans $T$.

Au vu de cette propri\'et\'e, on peut donc poser $\nabla_i^! := \nabla^!_{i,\{\preceq^\op i\}}$, et cet objet v\'erifie l'axiome~\eqref{it:qh-2}. Il v\'erifie \'egalement l'axiome~\eqref{it:qh-3}, puisque pour $n \in \Z$ on a
\[
[\nabla^!_i : L_i^! \langle n \rangle] =\dim \Hom_{A^!\lgmod}(\Delta^!_i \langle n \rangle, \nabla^!_i) = [\Delta^!_i \langle n \rangle : L_i^!] = \delta_{n,0}.
\]

On consid\`ere finalement l'axiome~\eqref{it:qh-4}.  L'axiome~\eqref{it:qh-3} implique déjà que
\[
\uExt^1(\Delta^!_i, \nabla^!_j) = 0
\]
pour tous $i,j \in I$ (voir par exemple la preuve de~\cite[Lemma~7.4]{riche-hab}). On en d\'eduit plus généralement que si $M$ est un $A^!$-module gradu\'e admettant une filtration par des objets $\Delta^!_j \langle n \rangle$, alors $\uExt^1(M,\nabla^!_j) = 0$. Soit $M^!_i$ le noyau de l'unique morphisme non nul $P^!_i \to \Delta^!_i$ (qui est surjectif puisque $L_i^!$ est la t\^ete de $\Delta^!_i$).  Pour montrer que $\uExt^2(\Delta^!_i, \nabla^!_j) = 0$, il suffit de montrer que $\uExt^1(M^!, \nabla^!_j) = 0$, et on se ramène donc à montrer que $M^!$ (ou, de fa{\c c}on \'equivalente, $P^!_i$) possède une filtration par des objets $\Delta^!_j \langle n \rangle$, c'est-\`a-dire appartient \`a $\langle \Delta^!_j \langle n \rangle : j \in I, \, n \in \Z \rangle_\ext$.  Via $\kappa^{-1}$, cette condition équivaut à dire que $L_i$ appartient à $\langle \nabla_j\la k\ra[-k] : j \in I, \, k \in \Z \rangle_\ext$.  Cette derni\`ere propri\'et\'e découle du Lemme~\ref{lem:!-pure} et de notre hypoth\`ese sur $A$.

Pour conclure, il reste maintenant \`a montrer que $A^!$ est quasi-h\'er\'editairement de Koszul, c'est-\`a-dire que
\begin{equation*}
\Hom_{\Db(A^!\lgmod)}(\Delta_i^!,L_j^!\la n\ra[k]) = 0 = \Hom_{\Db(A^!\lgmod)}(L_i^!, \nabla_j^!\la n\ra[k])
\end{equation*}
si $n \ne -k$. En appliquant $\kappa^{-1}$, on voit que la premi\`ere \'egalit\'e revient \`a dire que $\Hom_{\Db(A\lgmod)}(\nabla_i,J_j\la -n\ra[k+n]) = 0$ si $n \ne -k$, ce qui est clair. Pour la deuxi\`eme condition, en appliquant l'\'enonc\'e dual du Lemme~\ref{lem:!-pure}, on voit qu'on doit d\'emontrer que $L_i^!$ appartient \`a $\langle \Delta_j^! \langle -m \rangle [m] : j \in I, \, m \in \Z \rangle_\ext$, ce qui se traduit (en appliquant $\kappa^{-1}$) par la condition que $J_i$ appartient \`a $\langle \nabla_j \langle m \rangle : j \in I, \, m \in \Z \rangle_\ext$, c'est-\`a-dire que $J_i$ admet une filtration costandard. Cette derni\`ere propri\'et\'e est connue, comme rappel\'e ci-dessus.
\end{proof}

\begin{ex}\label{ex:carquois2}
Soit $R$
comme  au~\S\ref{ss:carquois}.
Considérons son quotient $R/(\alpha\beta)$, et posons
\begin{equation*}
\begin{gathered}
\Delta_1 = \nabla_1 = L_1 = \begin{tikzcd}
\bk \ar[r, shift left=0.5ex, "0"] &
0 \ar[l, shift left=0.5ex, "0"]
\end{tikzcd}
\\
\Delta_2 = P_2 = \begin{tikzcd}
\bk\la -1\ra \ar[r, shift left=0.5ex, "0"] &
\bk \ar[l, shift left=0.5ex, "\id"]
\end{tikzcd}
\qquad
\nabla_2 = 
J_2 = \begin{tikzcd}
\bk\la 1\ra \ar[r, shift left=0.5ex, "\id"] &
\bk \ar[l, shift left=0.5ex, "0"]
\end{tikzcd}
\end{gathered}
\end{equation*}
On peut montrer que ces objets vérifient les axiomes de la d\'efinition d'un anneau quasi-héréditaire (par rapport à l'ordre usuel $1 \prec 2$). Le seul point d\'elicat de cette preuve est la démonstration que $\uExt^2(\Delta_1,\nabla_1) = 0$ ; pour cela on peut utiliser la premi\`ere r\'esolution projective de~\eqref{eqn:res-proj}.

Les anneaux de Koszul autres que $R/(\alpha\beta)$ et $R^*/(\alpha^*\beta^*)$ considérés au~\S\ref{ss:carquois} 
ne sont quasi-h\'er\'editaires pour aucun ordre sur $\{1,2\}$.
\end{ex}

\subsection{Un exemple fondamental : la dualit\'e de Koszul de Be{\u\i}linson--Ginz\-burg--Soergel}
\label{ss:bgs}

Soit $\frg$ une algèbre de Lie complexe semi-simple, dont on fixe une sous-alg\`ebre de Borel et une sous-alg\`ebre de Cartan, et soit $\cO$ la catégorie de $\frg$-modules définie dans~\cite{bgg}. Soit $\cO_0 \subset \cO$ son \emph{bloc principal}, c'est-\`a-dire la sous-catégorie de Serre engendrée par les modules irréductibles dont le plus haut poids est de la forme $w\rho-\rho$ pour $w$ dans le groupe de Weyl $W$, où $\rho$ est la demi-somme des racines positives.  Par d\'efinition, les objets simples dans $\cO_0$ sont paramétrés par $W$. Il d\'ecoule par ailleurs des r\'esultats de~\cite{bgg} que tout objet de $\cO_0$ est de longueur finie, et que cette catégorie a assez d'objets projectifs.  Soit $L_w$, resp.~$P_w$, le module simple correspondant à $w \in W$, resp.~sa couverture projective.  Des r\'esultats g\'en\'eraux d'alg\`ebre homologique impliquent que
\begin{equation}
\label{eqn:o-alg}
\cO_0 \cong A\lmod, \quad \text{où} \quad A = \End \left( \bigoplus_{w \in W} P_w \right)^\op.
\end{equation}
Ici, l'anneau $A$ est une $\C$-algèbre de dimension finie.

Cet anneau est en fait quasi-héréditaire : sous l'équivalence~\eqref{eqn:o-alg}, les $\frg$-modules de Verma deviennent les $A$-modules standards, et leurs duaux deviennent les $A$-mo\-dules costandards. Par contre, $A$ n'a pas de graduation évidente.

Les résultats principaux de~\cite{bgs} affirment que cet anneau poss\`ede les propriétés remarquables suivantes :
\begin{enumerate}
\item L'anneau $A$ admet une graduation positive $A = \bigoplus_{n \ge 0} A_n$, où $A_0$ est un anneau semi-simple. Ses modules (co)standards admettent des graduations; ainsi, $A$ est un anneau quasi-héréditaire gradué.
\item L'anneau gradu\'e quasi-h\'er\'editaire $A$ est quasi-héréditairement de Koszul, et donc de Koszul.
\item L'anneau gradu\'e $A$ est isomorphe \`a son dual de Koszul.
\end{enumerate}
La catégorie $\cO_0^\gr=A\lgmod$ des modules gradués de type fini sur $A$ peut être considérée comme une ``version graduée'' de $\cO_0$.  
Le Théorème~\ref{thm:koszul2} affirme donc l'existence d'une équivalence de catégories triangul\'ees
\[
\kappa: \Db(\cO_0^\gr) \simto \Db(\cO_0^\gr)
\]
qui vérifie $\kappa(M\la 1\ra) \cong \kappa(M)\la -1\ra[1]$.  Les objets simples dans $\cO_0^\gr$ sont paramétrés par $(W, \preceq)$, où $\preceq$ désigne l'opposé de l'ordre de Bruhat. Le Théorème~\ref{thm:koszul2} donne un paramétrage des objets simples du côté dual par $(W, \preceq^\op)$ ; mais dans ce cas il est plus pratique d'utiliser le paramétrage ``intrinsèque'' des objets dans $\cO_0$.  D'après~\cite[Theorem~1.1.3]{bgs}, le foncteur $\kappa$ v\'erifie
\[
\kappa(L_w^\gr) \cong P_{w^{-1}w_0}^\gr, \quad \kappa(\nabla_w^\gr) \cong \Delta_{w^{-1} w_0}^\gr, \quad \kappa(J_w^\gr) \cong L_{w^{-1} w_0}^\gr
\]
pour tout $w \in W$,
où $w_0$ est l'élément le plus long de $W$. (Ici, pour tout $y \in W$, $L_y^\gr$ d\'esigne le module simple gradu\'e concentr\'e en degr\'e $0$ correspondant \`a $L_y$, $P_y^\gr$ sa couverture projective, $J_y^\gr$ son enveloppe injective, $\Delta_y^\gr$ l'unique ``relev\'e'' de $\Delta_y$ admettant une surjection vers $L_y^\gr$, et $\nabla_y^\gr$ l'unique ``relev\'e'' de $\nabla_y$ admettant $L_y^\gr$ comme sous-objet.)

\begin{ex}
Dans le cas où $\frg = \mathfrak{sl}_2$, l'anneau $A$ est isomorphe à l'anneau $R/(\alpha\beta)$ 
\'etudi\'e au~\S\ref{ss:carquois} et dans l'Exemple~\ref{ex:carquois2}. En particulier, on a vu 
explicitement
que cet anneau est auto-dual au sens de Koszul.
\end{ex}

\subsection{Catégories gradu\'ees de plus haut poids}
\label{ss:cat-qhered}

Afin de pouvoir adapter les notions du~\S\ref{ss:anneau-qh-koszul} dans un cadre plus général, nous allons maintenant introduire une version ``catégorique'' de quasi-hérédité.  Soit $\sA$ une catégorie abélienne $\bk$-linéaire qui est de longueur finie (c'est-\`a-dire telle que tout objet est de longueur finie). Supposons que $\sA$ est munie d'un automorphisme $\la 1\ra: \sA \to \sA$, qu'on appellera \emph{décalage interne}. On peut alors définir $\uHom$ et $\uExt$ comme dans~\eqref{eqn:gr-ext}. 

Soit $\Irr_0(\sA)$ un ensemble de représentants des objets simples à isomorphisme et à décalage interne près. Soit $I$ un ensemble 
muni d'une bijection fix\'ee
\[
I \simto \Irr_0(\sA),
\]
qu'on notera $i \mapsto L_i$,
ainsi que d'un ordre partiel $\preceq$ tel que pour tout $i \in I$, l'ensemble $\{ j \in I : j \preceq i \}$ est fini. Pour tout idéal $T \subset I$, notons $\sA_T$ la sous-catégorie pleine de $\sA$ dont les objets sont ceux dont tous les facteurs de composition sont de la forme $L_i\la n\ra$ avec $i \in T$ et $n \in \Z$. 

En s'inspirant de~\cite[\S 3.2]{bgs}, on dira que la catégorie $\sA$ (munie des structures ci-dessus) est \emph{gradu\'ee de plus haut poids} 
s'il existe des objets $\Delta_i, \nabla_i \in \sA_{\{\preceq i\}}$ et des morphismes $\Delta_i \to L_i \to \nabla_i$ qui vérifient les conditions suivantes :
\begin{enumerate}
\item 
\label{it:hw-1}
Pour tout $i \in I$, on a $\uEnd(L_i) \cong \bk$.
\item 
\label{it:hw-2}
Pour tout $i \in I$, et tout id\'eal $T \subset I$ dans lequel $i$ est maximal, $\Delta_i \to L_i$ est la couverture projective de $L_i$ dans $\sA_T$, et $L_i \to \nabla_i$ est son enveloppe injective. 
\item
\label{it:hw-3}
Le noyau du morphisme $\Delta_i \to L_i$ et le conoyau du morphisme $L_i \to \nabla_i$ appartiennent à $\sA_{\{\prec i\}}$. 
\item
\label{it:hw-4}
Pour tous $i,j \in I$, on a $\uExt^2(\Delta_i, \nabla_j) = 0$.
\end{enumerate}
Comme au~\S\ref{ss:anneau-qh-koszul}, les objets de la forme $\Delta_i\la n\ra$, resp.~$\nabla_i\la n\ra$, seront dits \emph{standards}, resp.~\emph{costandards}. Ils v\'erifient automatiquement
\begin{equation}
\label{eqn:hw-Ext-vanishing}
\Hom_{\Db(\sA)}(\Delta_i, \nabla_j \langle n \rangle [k]) = \begin{cases}
\bk & \text{si $i=j$ et $n=k=0$;}\\
0 & \text{sinon.}
\end{cases}
\end{equation}

La catégorie $\sA$ sera dite \emph{de type fini} si de plus elle vérifie la condition suivante:
\begin{enumerate}
\setcounter{enumi}{4}
\item 
\label{it:hw-5}
L'ensemble $I \cong \Irr_0(\sA)$ est fini.
\end{enumerate}
D'apr\`es~\cite[Theorem~3.2.1]{bgs}, si $\sA$ est une cat\'egorie gradu\'ee de plus haut poids et de type fini, alors il y a assez de projectifs et d'injectifs dans $\sA$. De plus, tout objet projectif (resp.~injectif) admet une filtration dont les sous-quotients sont des objets standards (resp.~costandards).  Dans ce cas, on note
\[
P_i \twoheadrightarrow L_i \qquad\text{et}\qquad L_i \hookrightarrow J_i
\]
la couverture projective et l'enveloppe injective de $L_i$, respectivement.

\begin{rmq}
La d\'efinition d'une cat\'egorie de plus haut poids (non gradu\'ee) est similaire, en omettant le foncteur $\langle 1 \rangle$ et en rempla{\c c}ant $\uEnd$ par $\End$ et $\uExt$ par $\Ext$. Les m\^emes propri\'et\'es restent valables dans ce cadre, avec les adaptations \'evidentes.
\end{rmq}

La catégorie $\sA$ sera dite \emph{positivement graduée}, resp.~\emph{strictement positivement graduée}, si elle vérifie
\begin{equation}\label{eqn:positive}
\Ext^1(L_i, L_j\la n\ra) = 0
\qquad\text{pour tout $n > 0$, resp.~pour tout $n \ge 0$.}
\end{equation}
Si $A = \bigoplus_{n \ge 0} A_n$ est un anneau positivement gradué avec $A_0$ de dimension finie, et si on choisit pour $\Irr_0(\sA)$ des repr\'esentants concentr\'es en degr\'e $0$, alors $A \lgmod$ est positivement graduée au sens de~\eqref{eqn:positive}. Elle est strictement positivement graduée si et seulement si $A_0$ est un anneau semi-simple.

\begin{rmq}
Il est à noter que la propriété d'être positivement graduée pour une catégorie abélienne $\sA$ comme ci-dessus dépend du choix des objets $L_i$ : si on remplace un $L_i$ par un décalage interne, cela peut changer les propriétés de positivité de la catégorie $\sA$. Ce phénomène correspond au fait qu'un anneau positivement gradué peut être Morita-équivalent à un anneau qui ne l'est pas.
\end{rmq}

\begin{ex}\label{ex:var-cplx}
Soit $X$ une variété algébrique complexe, munie d'une stratification algébrique finie $\mathscr{S}$ (au sens de~\cite[Definition~3.2.23]{cg}) dont toute strate est un espace affine. Comme expliqué dans~\cite[\S3.3]{bgs}, la catégorie $\Perv_{\mathscr{S}}(X,\bk)$ des faisceaux pervers sur $X$ à coefficients dans $\bk$ et constructibles par rapport à $\mathscr{S}$ est une catégorie de plus haut poids (non graduée) et de type fini. L'ensemble partiellement ordonné $(I,\preceq)$ qui gouverne cette catégorie est l'ensemble des strates, ordonnées par les inclusions de leurs adhérences. L'objet standard, resp.~costandard, associé à une strate $S$ est l'extension par zéro $j_{S!}\underline{\bk}[\dim S]$, resp.~l'image directe $j_{S*}\underline{\bk}[\dim S]$ du faisceau pervers constant (o\`u $j_S: S \hookrightarrow X$ est l'inclusion).

Si $\bk=\overline{\mathbb{Q}_\ell}$, et pour certains choix de $X$ (par exemple une vari\'et\'e de drapeaux, \'eventuellement parabolique, d'un groupe alg\'ebrique r\'eductif connexe),
les constructions de~\cite[\S4]{bgs} (qui exploitent la théorie des faisceaux étales $\ell$-adiques de Deligne) permettent de définir une catégorie $\Perv^\mix_{\mathscr{S}}(X,\bk)$, qui est graduée de plus haut poids et strictement positivement graduée, et qui peut \^etre consid\'er\'ee comme une ``version gradu\'ee'' de la catégorie $\Perv_{\mathscr{S}}(X,\bk)$.

Dans~\cite{ar:kdsf}, les auteurs ont proposé une nouvelle démarche qui mène à la même catégorie que dans~\cite[\S4]{bgs}. Cette nouvelle approche a présagé une gén\'eralisation au cas o\`u $\bk$ est de caractéristique positive, qui a été développée dans~\cite{modrap2}. Pour $\bk$ quelconque, la catégorie $\Perv_{\mathscr{S}}^\mix(X,\bk)$ comme définie dans~\cite{modrap2} est graduée de plus haut poids, mais pas n\'ecessairement positivement graduée.
\end{ex}

\subsection{Cat\'egories quasi-h\'er\'editairement de Koszul}

Une catégorie abélienne graduée de plus haut poids
est dite \emph{quasi-héréditairement de Koszul} si elle vérifie la condition suivante (cf.~\eqref{eqn:qh-koszul-defn}):
\[
\Ext^k(\Delta_i,L_j\la n\ra) = \Ext^k(L_i, \nabla_j\la n\ra) = 0
\qquad\text{si $n \ne -k$.}
\]
D'apr\`es~\cite[Proposition~2.13]{modrap3}, si $\sA$ est quasi-héréditairement de Koszul, alors elle est de Koszul au sens o\`u elle v\'erifie
\[
\Ext^k_{\sA}(L_i,L_j\la n\ra) =0
\qquad\text{si $n \ne -k$.}
\]
(En particulier, $\sA$ est alors strictement positivement graduée.)

Soient $\sA$ et $\sB$ deux catégories abéliennes graduées de plus haut poids, de type fini, et quasi-héréditairement de Koszul. Soient $(I_{\sA},\preceq_{\sA})$ et $(I_{\sB},\preceq_{\sB})$ les ensembles partiellement ordonnés qui paramètrent leurs objets simples respectifs. Une \emph{dualité de Koszul ``classique''} entre $\sA$ et $\sB$ est une \'equivalence de cat\'egories triangulées
\[
\kappa: \Db(\sA) \to \Db(\sB)
\]
qui vérifie
\[
\kappa(X\la 1\ra) \cong \kappa(X)\la -1\ra[1]
\]
pour tout $X$ dans $\Db(\sA)$,
et telle qu'il existe un isomorphisme d'ensembles partiellement ordonnés $\varphi: (I_{\sA},\preceq_{\sA}) \simto (I_{\sB}, \preceq_{\sB}^\op)$ tel que
\[
\kappa(L^\sA_i) \cong P^\sB_{\varphi(i)}, \qquad
\kappa(\nabla^\sA_i) \cong \Delta^\sB_{\varphi(i)}, \qquad
\kappa(I^\sA_i) \cong L^\sB_{\varphi(i)}
\]
pour tout $i \in I_{\sA}$.

Un analogue du Théorème~\ref{thm:koszul2} affirme que si $\sA$ est graduée de plus haut poids, de type fini, 
et quasi-héréditairement de Koszul, alors il existe une catégorie $\sB$ poss\'edant les m\^emes propri\'et\'es et une dualité de Koszul $\Db(\sA) \to \Db(\sB)$. En fait, on a alors $\sA \cong A\lgmod$ pour une certaine $\bk$-alg\`ebre $A$ comme dans le Théorème~\ref{thm:koszul2}, et on peut prendre $\sB = A^!\lgmod$, où $A^! = \uExt^\bullet(L,L)^\op$.

\begin{ex}[Beilinson--Ginzburg--Soergel]\label{ex:bgs2}
Soient $G$ un groupe réductif complexe et $B \subset G$ un sous-groupe de Borel. Notons $\Flag := G/B$ sa variété de drapeaux. Notons $\Perv_{\mathrm{Br}}(\Flag,\C)$ la catégorie des faisceaux pervers sur $\Flag$, à coefficients dans $\C$, et constructible par rapport à la stratification dont les strates sont les orbites de $B$ sur $\Flag$. Ces derniers étant des espaces affines, on peut faire appel à l'Exemple~\ref{ex:var-cplx} et ainsi conclure que la catégorie $\Perv_{\mathrm{Br}}(\Flag,\C)$ est de plus haut poids.

Dans ce cadre, Beilinson--Ginzburg--Soergel ont établi dans~\cite[\S4]{bgs} une version géométrique de l'exemple du~\S\ref{ss:bgs}. Soit $W$ le groupe de Weyl de $G$, et pour $w \in W$, notons $\mathscr{L}^\mix_w \in \Perv_{\mathrm{Br}}^\mix(\Flag,\C)$ le faisceau pervers simple de poids $0$ dont le support est la variété de Schubert correspondant à $w$.  Soiet $\mathscr{P}_w^\mix$, resp.~$\mathscr{I}^\mix_w$, sa couverture projective, resp.~son enveloppe injective, dans $\Perv_{\mathrm{Br}}^\mix(\Flag,\C)$. Les résultats de~\cite{bgs} impliquent qu'il existe une équivalence de catégories
\[
\kappa: \Db\Perv^\mix_{\mathrm{Br}}(\Flag,\C) \to \Db\Perv^\mix_{\mathrm{Br}}(\Flag,\C)
\]
qui vérifie
\[
\kappa(\mathscr{F}\la 1\ra) \cong \kappa(\mathscr{F})\la -1\ra[1]
\]
pour tout $\mathscr{F}$ dans $\Db\Perv^\mix_{\mathrm{Br}}(\Flag,\C)$.
Par rapport au paramétrage ``intrinsèque'' des deux côtés, ce foncteur v\'erifie
\[
\kappa(\mathscr{L}^\mix_w) \cong \mathscr{P}^\mix_{w^{-1}w_0}, \qquad
\kappa(\nabla^\mix_w) \cong \Delta^\mix_{w^{-1}w_0}, \qquad
\kappa(\mathscr{I}^\mix_w) \cong \mathscr{L}^\mix_{w^{-1}w_0}
\]
pour tout $w \in W$,
où $w_0$ est l'élément le plus long de $W$. 
\end{ex}

\section{Objets basculants et dualité de Ringel}
\label{sec:dualite-Ringel}

\subsection{Objets basculants}
\label{ss:basculant}

Soit $\sA$ une catégorie de plus haut poids (pas n\'ecessairement de type fini \`a ce stade). Un objet $X \in \sA$ est dit \emph{basculant} s'il existe une suite de sous-objets $0 = X_0 \subset X_1 \subset \cdots \subset X_n = X$ telle que tout sous-quotient $X_i/X_{i-1}$ est standard, ainsi qu'une autre suite $0 = X'_0 \subset X'_1 \subset \cdots \subset X'_m = X$ telle que tout sous-quotient $X'_j/X'_{j-1}$ est costandard.

Le r\'esultat suivant est d\^u \`a Ringel~\cite{ringel}. Pour une preuve dans le langage de cet article, on pourra consulter~\cite[\S 7]{riche-hab}.

\begin{thm}[Classification des objets basculants]
\label{thm:tilt-classif}
Soit $\sA$ une catégorie de plus haut poids.  Pour tout $i \in I$, il existe un objet basculant indécomposable $T_i$, unique à isomorphisme près, tel que $T_i \in \sA_{\{\preceq i\}}$ et $[T_i : L_i] = 1$. De plus, tout objet basculant indécomposable dans $\sA$ est isomorphe à $T_i$ pour un unique $i \in I$.
\end{thm}

La classification dans le cas gradué est très similaire. 
En utilisant~\eqref{eqn:hw-Ext-vanishing}, on voit
que tous les $\Ext$ supérieurs entre objets basculants s'annulent. De ce point de vue, le comportement des objets basculants ressemble à celui des objets projectifs (ou injectifs).

La \emph{dualité de Ringel} transforme cette remarque en une équivalence de catégories.  Supposons désormais que $\sA$ est de type fini. Il y a donc assez de projectifs (et d'injectifs) dans $\sA$. Soit $\sB$ une autre catégorie de plus haut poids et de type fini. Soient $(I_{\sA},\preceq_{\sA})$ et $(I_{\sB},\preceq_{\sB})$ les ensembles partiellement ordonnés qui paramètrent leurs objets simples respectifs. Une \emph{dualité de Ringel} entre $\sA$ et $\sB$ est un foncteur triangulé
\[
\rho: \Db(\sA) \to \Db(\sB)
\]
qui vérifie
\[
\rho(T^\sA_i) \cong P^\sB_{\varphi(i)}, \qquad
\rho(\Delta^\sA_i) \cong \nabla^\sB_{\varphi(i)}, \qquad
\rho(I^\sA_i) \cong T^\sB_{\varphi(i)}
\]
pour tout $i \in I_\sA$,
pour un isomorphisme d'ensembles partiellement ordonn\'es
\[
\varphi : (I_{\sA},\preceq_{\sA}) \simto (I_{\sB}, \preceq_{\sB}^\op).
\]
Dans le cas gradué, on impose de plus que $\rho$ commute au décalage interne.

\'Etant donn\'ee $\sA$, une telle dualit\'e existe toujours, comme affirm\'e dans l'\'enonc\'e suivant.

\begin{thm}[{Ringel~\cite{ringel}}]
\label{thm:ringel}
Soit $\sA$ une catégorie de plus haut poids et de type fini. Il existe une catégorie de plus haut poids et de type fini $\sB$ et une dualité de Ringel $\rho: \Db(\sA) \to \Db(\sB)$.
\end{thm}

La preuve dans~\cite{ringel} est \'ecrite dans le langage des anneaux quasi-héréditaires. Concrètement, si $\sA \cong A\lmod$ pour un anneau quasi-h\'er\'editaire $A$, on peut prendre $\sB = A^\sharp\lmod$, où
\begin{equation}\label{eqn:asharp-defn}
A^\sharp = \End\Big(\bigoplus_{i \in I} T_i\Big)^\op,
\end{equation}
puis poser $\rho := R\Hom(\bigoplus_i T_i, {-})$. Le contenu principal du Théorème~\ref{thm:ringel} est que l'anneau $A^\sharp$ est quasi-héréditaire.

Pour la version graduée du Théorème~\ref{thm:ringel}, on définit $A^\sharp$ et $\rho$ en termes de $\uEnd$ et $\uHom$. Notons que dans ce cas, $A^\sharp$ n'est pas n\'ecessairement positivement graduée.

\begin{ex}\label{ex:carquois3}
Nous avons vu dans l'Exemple~\ref{ex:carquois2} que l'anneau $R/(\alpha\beta)$ qui a été introduit au~\S\ref{ss:carquois}
est quasi-héréditaire. Les modules basculants indécomposables pour cet anneau sont:
\[
T_1 = L_1 = \begin{tikzcd}
\bk \ar[r, shift left=0.5ex, "0"] &
0 \ar[l, shift left=0.5ex, "0"]
\end{tikzcd},
\qquad
T_2 = P_1\la 1\ra = I_1\la -1\ra = 
\begin{tikzcd}[ampersand replacement=\&]
(\bk\la -1\ra \oplus \bk\la 1\ra)  \ar[r, shift left=0.5ex, "{\left[\begin{smallmatrix} 0 & 1\end{smallmatrix}\right]}"]\&
\bk \ar[l, shift left=0.5ex, "{\left[\begin{smallmatrix} 1 \\ 0\end{smallmatrix}\right]}"]
\end{tikzcd}.
\]
Il existe des morphismes non nuls $\alpha^\sharp: T_1 \to T_2\la 1\ra$ et $\beta^\sharp: T_2 \to T_1\la 1\ra$ qui vérifient $\beta^\sharp \circ \alpha^\sharp = 0$. On peut en déduire un isomorphisme $(R/(\alpha\beta))^\sharp \cong R/(\alpha\beta)$.  Autrement dit, l'anneau $R/(\alpha\beta)$ est auto-dual au sens de Ringel.
\end{ex}

\begin{ex}[Be{\u\i}linson--Bezrukavnikov--Mirkovi\'c]\label{ex:bbm}
Soient $G$, $B$, et $\Flag$ comme dans l'Exemple~\ref{ex:bgs2}. On a déjà vu que la catégorie $\Perv_{\mathrm{Br}}(\Flag,\bk)$ est de plus haut poids et de type fini, et on peut se demander quel est son dual au sens de Ringel. Un résultat remarquable de Be{\u\i}linson--Bezrukavnikov--Mirkovi\'c~\cite{bbm} affirme que cette cat\'egorie est auto-duale au sens de Ringel. (Dans le cas particulier de $G = \mathrm{SL}_2$, on peut déduire ce résultat de l'Exemple~\ref{ex:carquois3}.) En fait, les auteurs donnent une construction ``faisceautique'' d'une equivalence de catégories
\[
\rho: \Db\Perv_{\mathrm{Br}}(\Flag,\bk) \simto  \Db\Perv_{\mathrm{Br}}(\Flag,\bk)
\]
qui v\'erifie
\[
\rho(\mathscr{T}_w) \cong \mathscr{P}_{ww_0}, \qquad
\rho(\nabla_w) \cong \Delta_{ww_0}, \qquad
\rho(\mathscr{I}_w) \cong \mathscr{T}_{ww_0}
\]
pour tout $w \in W$.
Ici nous avons suivi la convention du~\S\ref{ss:bgs} et utilisé le paramétrage intrinsèque des deux côtés.
\end{ex}

\subsection{Composition des dualités de Koszul et de Ringel}
\label{ss:koszul-ringel}

Soit $\sA$ une catégorie graduée de plus haut poids, de type fini, 
et quasi-héréditairement de Koszul.  Il existe alors une autre catégorie quasi-héréditairement de Koszul $\sB$ ainsi qu'une dualité de Koszul classique
\[
\kappa: \Db(\sA) \to \Db(\sB).
\]
Appliquons la version graduée et inversée du Théorème~\ref{thm:ringel} à $\sB$: il existe une catégorie de plus haut poids $\sC$ et une équivalence de catégories
\[
\rho: \Db(\sC) \to \Db(\sB)
\]
qui envoie les objets basculants (resp.~injectifs) sur des objets projectifs (resp.~basculants).  Le foncteur composé
\[
\rho^{-1} \circ \kappa: \Db(\sA) \to \Db(\sC)
\]
a la propriété qu'il envoie
$L^\sA_i$ sur $T^\sC_{\varphi(i)}$
pour une bijection convenable $\varphi: I_\sA \simto I_\sC$.

En général, il n'est pas possible de prendre la composition des Théorèmes~\ref{thm:koszul2} et~\ref{thm:ringel} (ou des foncteurs $\kappa$ et $\rho^{-1}$) dans l'autre sens: même si $\sA$ est quasi-héréditaire\-ment de Koszul, son dual au sens de Ringel n'est pas n\'ecessairement de Koszul. Si celui-ci est effectivement quasi-héréditairement de Koszul, on dit (d'après~\cite{mazorchuk}) que $\sA$ est \emph{équilibrée}. Le résultat principal de~\cite{mazorchuk} affirme que si $\sA$ est équilibrée, les dualités de Koszul et de Ringel commutent. 

L'énoncé suivant est un résumé des remarques précédentes. 

\begin{thm}[Dualité de Koszul, 3ème version]
\label{thm:koszul3}
Soit $\sA$ une catégorie de plus haut poids, de type fini, et quasi-héréditairement de Koszul, dont les objets simples sont paramétrés par $(I_\sA, \preceq_\sA)$.

Alors il existe une catégorie graduée de plus haut poids $\sB$ (par rapport à un ensemble $(I_\sB, \preceq_\sB)$), un isomorphisme $\varphi: (I_\sA, \preceq_\sA) \simto (I_\sB, \preceq_\sB)$, et une équivalence de catégories triangul\'ees
\[
\kappa: \Db(\sA) \simto \Db(\sB)
\]
qui vérifie
\[
\kappa(M\la 1\ra) \cong \kappa(M)\la -1\ra[1]
\]
pour tout $M$ dans $\Db(\sA)$, et telle que
$\kappa(L^\sA_i) \cong T^\sB_{\varphi(i)}$ pour tout $i \in I_\sA$.

Si $\sA$ est équilibrée, alors $\kappa$ v\'erifie \'egalement
$\kappa(T^\sA_i) \cong L^\sB_{\varphi(i)}$ pour tout $i \in I_\sA$.
\end{thm}

\subsection{Travaux de Bezrukavnikov--Yun}
\label{ss:bez-yun}

Soient $G$, $B$, et $\Flag$ comme dans les Ex\-emples~\ref{ex:bgs2} et~\ref{ex:bbm}. Les constructions de l'Exemple~\ref{ex:bbm} se g\'en\'eralisent au cadre des faisceaux pervers mixtes (voir~\cite{ar:kdsf}), de sorte que le dual de Ringel de la catégorie $\Perv_{\mathrm{Br}}^\mix(\Flag,\C)$ est bien quasi-héréditairement de Koszul. Cette catégorie est donc équilibrée au sens de~\cite{mazorchuk}, et le Théorème~\ref{thm:koszul3} nous permet d'affirmer qu'il existe une équivalence de catégories triangul\'ees
\[
\kappa: \Db\Perv^\mix_{\mathrm{Br}}(\Flag,\C) \simto \Db\Perv^\mix_{\mathrm{Br}}(\Flag,\C)
\]
qui vérifie
\[
\kappa(\mathscr{F}\la 1\ra) \cong \kappa(\mathscr{F})\la -1\ra[1]
\]
pour tout $\mathscr{F}$ dans $\Db\Perv^\mix_{\mathrm{Br}}(\Flag,\C)$, et telle que
\[
\kappa(\mathscr{L}^\mix_w) \cong \mathscr{T}^\mix_{w^{-1}}, \qquad
\kappa(\mathscr{T}^\mix_w) \cong \mathscr{L}^\mix_{w^{-1}}
\]
pour tout $w \in W$.
Il est à noter que ni les objets projectifs, ni les objets injectifs ne figurent dans cet énoncé. Prenant cette observation (sugg\'er\'ee plus t\^ot par Be{\u\i}linson--Ginzburg dans~\cite{bg}) comme point de départ, Bezrukavnikov--Yun ont établi une vaste généralisation de cet énoncé. 

Ils travaillent dans le contexte d'un groupe de Kac--Moody $\scG$. (Il existe plusieurs versions des ``groupes de Kac--Moody'' dans la litt\'erature. On renvoie \`a~\cite[\S 9]{rw} pour une discussion de la version que nous utilisons. En particulier, pour nous un groupe de Kac--Moody est d\'etermin\'e par une ``donn\'ee radicielle de Kac--Moody''; cette donn\'ee est sous-entendue quand on consid\`ere $\mathscr{G}$. La dualit\'e de Langlands pour les groupes de Kac--Moody est d\'efinie par l'involution \'evidente sur l'ensemble des donn\'ees radicielles de Kac--Moody.) Par d\'efinition ce groupe poss\`ede un sous-groupe de Borel canonique $\mathscr{B}$, et on notera
$\Flag = \Flag_{\mathscr{G}}:=\mathscr{G}/\mathscr{B}$ sa vari\'et\'e de drapeaux. Alors la d\'ecomposition de Bruhat fournit une stratification naturelle
\begin{equation}
\label{eqn:stratif-Bruhat}
\Flag_\mathscr{G} = \bigsqcup_{w \in \WKM} \Flag_{\mathscr{G},w}
\end{equation}
o\`u $\WKM$ est le groupe de Weyl du groupe de Kac--Moody correspondant (muni de sa structure naturelle de groupe de Coxeter), et o\`u $\Flag_{\mathscr{G},w} \cong \mathbb{C}^{\ell(w)}$ pour tout $w \in \WKM$.

Si $\scG$ est de dimension infinie, alors $\Flag$ l'est aussi, et il y a un nombre infini de strates. Cependant, l'adhérence de toute strate est la réunion d'un nombre fini de strates, ce qui permet de définir la catégorie des complexes constructibles $\Db_{\mathrm{Br}}(\Flag,\C)$ ainsi que la sous-catégorie des faisceaux pervers $\Perv_{\mathrm{Br}}(\Flag,\C)$. Cette derni\`ere cat\'egorie est encore une catégorie de plus haut poids : on peut donc parler des objets standards, costandards, ou basculants dans cette cat\'egorie.

Par contre, quand $\scG$ est de dimension infinie, la catégorie $\Perv_{\mathrm{Br}}(\Flag_\scG,\C)$ n'a en g\'en\'eral pas assez d'objets projectifs ni d'objets injectifs. Les Théorèmes~\ref{thm:koszul2} et~\ref{thm:ringel} n'ont donc pas de sens pour $\Perv_{\mathrm{Br}}(\Flag_\scG,\C)$, mais leur ``composée'', le Théorème~\ref{thm:koszul3}, admet une généralisation.

\begin{thm}[Bezrukavnikov--Yun]\label{thm:by}
Soit $\scG$ un groupe de Kac--Moody, et soit $\scG^\vee$ son dual de Langlands. 
Il existe une équivalence de catégories triangul\'ees
\[
\kappa: \Db\Perv^\mix_{\mathrm{Br}}(\Flag_{\scG},\C) \simto \Db\Perv^\mix_{\mathrm{Br}}(\Flag_{\scG^\vee},\C)
\]
qui vérifie
\[
\kappa(\mathscr{F}\la 1\ra) \cong \kappa(\mathscr{F})\la -1\ra[1]
\]
pour tout $\mathscr{F}$ dans $\Db\Perv^\mix_{\mathrm{Br}}(\Flag_{\scG},\C)$, et telle que
\[
\kappa(\mathscr{L}^\mix_{\scG,w}) \cong \mathscr{T}^\mix_{\scG^\vee,w^{-1}}, \qquad
\kappa(\mathscr{T}^\mix_{\scG,w}) \cong \mathscr{L}^\mix_{\scG^\vee,w^{-1}}
\]
pour tout $w \in \WKM$,
o\`u on identifie les groupes de Weyl de $\scG$ et $\scG^\vee$ de la fa{\c c}on canonique.
\end{thm}

Si on remplace $\Flag_{\scG^\vee}$ par la variété de drapeaux ``opposée'' $\Flag_{\scG^\vee}^\op := \scB^\vee\backslash \scG^\vee$, on obtient un foncteur $\kappa'$ v\'erifiant
\[
\kappa'(\mathscr{L}^\mix_{\scG,w}) \cong \mathscr{T}^\mix_{\scG^\vee,\op,w}, \qquad
\kappa'(\mathscr{T}^\mix_{\scG,w}) \cong \mathscr{L}^\mix_{\scG^\vee,\op,w}
\]
pour tout $w \in \WKM$.

\begin{rmq}
\begin{enumerate}
\item
Dans~\cite{by} les auteurs ne consid\`erent pas des cat\'egories du type $\Db\Perv^\mix_{\mathrm{Br}}(\Flag_{\scG},\C)$, mais plut\^ot certaines cat\'egories dont la d\'efinition fait intervenir \emph{tous} les faisceaux pervers mixtes au sens de Deligne. La version \'enonc\'ee au Th\'eor\`eme~\ref{thm:by} peut \^etre obtenue par des m\'ethodes similaires (ou \^etre vue comme un cas particulier du Th\'eor\`eme~\ref{thm:mkd} ci-dessous).
\item
Sous l'hypoth\`ese o\`u $\mathscr{G}$ est sym\'etrisable, le Th\'eor\`eme~\ref{thm:by} reste vrai si on remplace $\mathscr{G}^\vee$ par $\mathscr{G}$. Ceci explique pourquoi la dualit\'e de Langlands n'apparaissait pas dans l'\'enonc\'e consid\'er\'e au d\'ebut de ce paragraphe.
\end{enumerate}
\end{rmq}

Pour les faisceaux pervers $\Perv_{\mathrm{Br}}(\Flag_\scG,\bk)$ à coefficients dans un corps g\'en\'eral $\bk$, le Théorème~\ref{thm:by} tel qu'\'enonc\'e ici n'est pas vrai: si $\bk$ est de caractéristique positive, les faisceaux pervers mixtes standards et simples ne vérifient pas~\eqref{eqn:qh-koszul-defn} en général.

Cependant, il est possible d'\'enoncer une version du Théorème~\ref{thm:by} qui est valable en caractéristique positive si on remplace les faisceaux pervers simples par de nouveau objets : les \emph{faisceaux à parité} introduits dans~\cite{jmw}. Dans le~\S\ref{sec:objets-parite} nous consid\'erons une variante de cette construction, \`a savoir le concept d'\emph{objets à parité} dans une cat\'egorie triangul\'ee munie d'une suite quasi-exceptionnelle.

\part{Dualité de Koszul formelle} 
\label{pt:Koszul-formelle}

\section{Objets à parité}
\label{sec:objets-parite}

Dans cette partie nous d\'eveloppons une th\'eorie d'\emph{objets \`a parit\'e} dans le cadre des cat\'egories triangul\'ees munies d'une collection quasi-exceptionnelle satisfaisant certaines conditions techniques. Cette th\'eorie s'inspire de la th\'eorie des \emph{complexes \`a parit\'e} consid\'er\'es par Juteau--Mautner--Williamson~\cite{jmw}, et les preuves des propri\'et\'es de nos objets suivent \'etroitement celles de ces auteurs. Cependant, le cadre consid\'er\'e est un petit peu diff\'erent (voir le~\S\ref{ss:comparaison-JMW} pour une comparaison pr\'ecise).

\subsection{Collections quasi-exceptionnelles}

Consid\'erons une cat\'egorie triangul\'ee $\mathscr{D}$. Si $X$ et $Y$ sont deux objets de $\mathscr{D}$, on notera
\[
\Hom^\bullet_{\mathscr{D}}(X,Y) := \bigoplus_{n \in \Z} \Hom_{\mathscr{D}}(X,Y[n]),
\]
et on consid\'erera ce groupe ab\'elien comme gradu\'e avec la graduation naturelle.

Soit $(I,\preceq)$ un ensemble partiellement ordonn\'e.
 Suivant Bezrukavnikov~\cite{bez}, une \emph{collection quasi-exception\-nelle} param\'etr\'ee par $(I,\preceq)$ est une famille $(\nabla_i : i \in I)$ d'objets de $\mathscr{D}$ param\'etr\'ee par $I$ et v\'erifiant les conditions suivantes :
\begin{enumerate}
\item
si $i,j \in I$ et
$\Hom_{\mathscr{D}}^\bullet(\nabla_i, \nabla_j) \neq 0$, alors $i \succeq j$;
\item
si $i \in I$ alors $\Hom_{\mathscr{D}}(\nabla_i, \nabla_i[n])=0$ pour $n<0$, et $\Hom_{\mathscr{D}}(\nabla_i, \nabla_i)$ est une alg\`ebre \`a division.
\end{enumerate}
\'Etant donn\'ee une telle collection exceptionnelle, une famille $(\Delta_i : i \in I)$ d'objets de $\mathscr{D}$ sera dite \emph{duale} de $(\nabla_i : i \in I)$ si elle v\'erifie
\[
\Hom^\bullet_{\mathscr{D}}(\Delta_i, \nabla_j)=0 \quad \text{pour tous $i \succ j$}
\]
et si de plus les images de $\Delta_i$ et $\nabla_i$ dans le quotient de Verdier $\mathscr{D} / \mathscr{D}_{\{\prec i\}}$ (o\`u $\mathscr{D}_{\{\prec i\}}$ est la sous-cat\'egorie triangul\'ee pleine de $\mathscr{D}$ engendr\'ee par les objets $\nabla_j$ avec $j \prec i$) sont isomorphes. D'apr\`es~\cite[Lemma~2]{bez}, une collection duale est unique (\`a isomorphismes pr\`es) si elle existe, elle forme une collection quasi-exceptionnelle param\'etr\'ee par $(I,\preceq^\op)$ o\`u $\preceq^\op$ est l'ordre oppos\'e \`a $\preceq$, et enfin on a :
\begin{enumerate}
\item
$\Hom^\bullet_{\mathscr{D}}(\Delta_i, \nabla_j)=0$ pour tous $i,j \in I$ avec $i \neq j$ ;
\item
un isomorphisme d'anneaux gradu\'es $\Hom^\bullet_{\mathscr{D}}(\nabla_i, \nabla_i) \cong \Hom^\bullet_{\mathscr{D}}(\Delta_i, \Delta_i)$ pour tout $i \in I$.
\end{enumerate}

\begin{ex}
\label{ex:collections-qe}
Les exemples principaux de collections quasi-exceptionnelles que nous rencontrerons seront les suivants.
\begin{enumerate}
\item
\label{it:exceptionnelle}
Supposons que $\mathscr{D}$ est lin\'eaire sur un corps $\bk$, et que pour tous $X,Y$ dans $\mathscr{D}$ le $\bk$-espace vectoriel $\bigoplus_{n \in \Z} \Hom_{\mathscr{D}}(X,Y[n])$ est de dimension finie. Suivant~\cite{bez:ctm}, une collection quasi-exceptionnelle $(\nabla_i : i \in I)$ est dite \emph{exceptionnelle} si elle v\'erifie $\Hom_{\mathscr{D}}(\nabla_i, \nabla_i[n])=0$ pour $n>0$, et $\Hom_{\mathscr{D}}(\nabla_i, \nabla_i)=\bk$. Sous ces conditions, si $I$ est une r\'eunion disjointe d'ensembles finis et de copies de $\Z_{\geq 0}$ (avec son ordre naturel), la collection duale existe toujours d'apr\`es~\cite[Proposition~1]{bez:ctm}.
\item
\label{it:qe-faisceaux}
Soit $\bk$ un corps, soit $X$ une vari\'et\'e alg\'ebrique complexe, munie d'une action d'un groupe alg\'ebrique complexe connexe $H$, et soit $\mathscr{S}$ une stratification alg\'ebrique finie de $X$ dont toute strate est connexe et stable par l'action de $H$. On notera $j_s : X_s \to X$ l'inclusion. Soit $\mathscr{D}=\Db_{H,\mathscr{S}}(X,\bk)$ la cat\'egorie d\'eriv\'ee $H$-\'equivariante et $\mathscr{S}$-constructible des faisceaux de $\bk$-espaces vectoriels sur $X$. On supposera que pour tout $s \in \mathscr{S}$ et pour tous syst\`emes locaux simples $H$-\'equivariants non isomorphes $\mathscr{L}$ et $\mathscr{L}'$ on a
\[
\Hom^\bullet_{\Db_H(X_s,\bk)}(\mathscr{L}, \mathscr{L}')=0.
\]
(Cette condition est bien s\^ur automatique si le syst\`eme local constant est le seul syst\`eme local $H$-\'equivariant simple sur $X_s$.)
Soit alors $I$ l'ensemble des paires $(s, \mathscr{L})$ o\`u $s \in \mathscr{S}$ et $\mathscr{L}$ est un(e classe d'isomorphisme de) syst\`eme local $H$-\'equivariant simple sur $X_s$, muni de l'ordre induit par l'inclusion des adh\'erences de strates. Il est facile de v\'erifier que la collection d\'efinie par
\[
\nabla_{s, \mathscr{L}} := (j_s)_* \mathscr{L}[\dim X_s]
\]
est quasi-exceptionnelle, que la collection duale existe, et qu'elle est donn\'ee par
\[
\Delta_{s, \mathscr{L}} := (j_s)_! \mathscr{L}[\dim X_s].
\]
\end{enumerate}
\end{ex}

Dans la suite, on supposera toujours que $(I,\preceq)$ est bien ordonn\'e (c'est-\`a-dire que l'ordre est total, et que tout sous-ensemble non vide de $I$ admet un \'el\'ement minimal), et on notera $\mathscr{D}_{\{\preceq i\}}$, resp.~$\mathscr{D}_{\{\prec i\}}$, la sous-cat\'egorie triangul\'ee de $\mathscr{D}$ engendr\'ee par les objets $\nabla_j$ avec $j \preceq i$, resp.~$j \prec i$. D'apr\`es~\cite[Lemma~2(f)]{bez}, notre hypoth\`ese implique que $\mathscr{D}_{\{\preceq i\}}$ est \'egalement engendr\'ee par les objets $\Delta_j$ avec $j \preceq i$, et de m\^eme pour $\mathscr{D}_{\{\prec i\}}$.

Sous cette hypoth\`ese, et si $\mathscr{D}$ est engendr\'ee comme cat\'egorie triangul\'ee par les objets $(\nabla_i : i \in I)$, il existe une (unique) t-structure sur $\mathscr{D}$ telle que $\mathscr{D}^{\geq 0}$ est la sous-cat\'egorie pleine engendr\'ee par extension par les objets de la forme $\nabla_i[n]$ avec $i \in I$ et $n \in \Z_{\leq 0}$, et $\mathscr{D}^{\leq 0}$ est la sous-cat\'egorie pleine engendr\'ee par extension par les objets de la forme $\Delta_i[n]$ avec $i \in I$ et $n \in \Z_{\geq 0}$; voir~\cite[Proposition~1]{bez}. Cette t-structure se restreint en une t-structure sur chacune des sous-cat\'egories $\mathscr{D}_{\{\preceq i\}}$ et $\mathscr{D}_{\{\prec i\}}$ (plus pr\'ecis\'ement, en la t-structure associ\'ee aux collections quasi-exceptionnelles $(\nabla_j : j \preceq i)$ et $(\nabla_j : j \prec i)$ respectivement). 

\begin{ex}
Dans le cas particulier de l'Exemple~\ref{ex:collections-qe}\eqref{it:qe-faisceaux}, la t-structure ainsi cons\-truite est la t-structure perverse.
\end{ex}

\subsection{Objets \`a parit\'e : d\'efinition}
\label{ss:parite-def}

Soit $\mathscr{D}$ une cat\'egorie triangul\'ee qui est de Krull--Schmidt (c'est-\`a-dire que tout objet s'\'ecrit comme somme directe d'objets ind\'ecomposables dont l'anneau des endomorphismes est local), et soit $(I,\preceq)$ un ensemble bien ordonn\'e.  

Si $(\nabla_i : i \in I)$ est une collection quasi-exceptionnelle qui admet une collection duale $(\Delta_i : i \in I)$, on dira que $(\nabla_i : i \in I)$ est \emph{\`a parit\'e} si elle v\'erifie
\begin{equation}
\label{eqn:Hom-nabla-pair}
\Hom_{\mathscr{D}}(\nabla_i, \nabla_i[n]) = 0 \quad \text{pour tout $i \in I$ et tout $n$ impair.}
\end{equation}
Dans la suite, on fixe une telle collection, et on suppose de plus que la famille
$(\nabla_i : i \in I)$ (ou, de fa{\c c}on \'equivalente, la famille $(\Delta_i : i \in I)$) engendre $\mathscr{D}$ comme cat\'egorie triangul\'ee.

La d\'efinition suivante est inspir\'ee par~\cite{jmw}. (Voir le~\S\ref{ss:comparaison-JMW} ci-dessous pour une comparaison d\'etaill\'ee.)

\begin{defn}
\label{def:parite}
Consid\'erons une fonction $\dag : I \to \Z/2\Z$.
On dira qu'un object $X$ de $\mathscr{D}$ est:
\begin{enumerate}
\item
$(\dag,*)$-pair si pour tout $i \in I$ le $\Hom^\bullet_{\mathscr{D}}(\nabla_i, \nabla_i)$-module \`a gauche $\Z$-gradu\'e
\[
\Hom^\bullet_{\mathscr{D}}(X, \nabla_i)
\]
est libre de rang fini et concentr\'e en degr\'es dont la classe modulo $2$ est $\dag(i)$;
\item
$(\dag,!)$-pair si pour tout $i \in I$ le $\Hom^\bullet_{\mathscr{D}}(\Delta_i, \Delta_i)$-module \`a droite $\Z$-gradu\'e
\[
\Hom_{\mathscr{D}}^\bullet(\Delta_i,X)
\]
est libre de rang fini et concentr\'e en degr\'es dont la classe modulo $2$ est $\dag(i)$;
\item
$\dag$-pair si $X$ est $(\dag,*)$-pair et $(\dag,!)$-pair;
\item
$(\dag,*)$-impair si $X[1]$ est $(\dag,*)$-pair;
\item
$(\dag,!)$-impair si $X[1]$ est $(\dag,!)$-pair;
\item
$\dag$-impair si $X[1]$ est $\dag$-pair;
\item
\`a $\dag$-parit\'e si $X$ est isomorphe \`a une somme $X' \oplus X''$ avec $X'$ $\dag$-pair et $X''$ $\dag$-impair.
\end{enumerate}
\end{defn}

Dans la suite, la fonction $\dag$ sera souvent sous-entendue, et omise des notations.

\begin{rmq}\phantomsection
\label{rmq:def-parite}
\begin{enumerate}
\item
Dans la situation de l'Exemple~\ref{ex:collections-qe}\eqref{it:exceptionnelle},
la condition d'\^etre libre de rang fini est bien s\^ur automatique.
\item
La sous-cat\'egorie pleine de $\mathscr{D}$ dont les objets sont les objets \`a $\dag$-parit\'e est stable par le d\'ecalage cohomologique $[1]$.
\item
\label{it:parite-facteur-direct}
Par le lemme de Nakayama gradu\'e, tout facteur direct d'un $\Hom^\bullet_{\mathscr{D}}(\nabla_i, \nabla_i)$-module gradu\'e libre de rang fini est \'egalement libre de rang fini; on en d\'eduit qu'un facteur direct d'un objet $(\dag,*)$-pair est $(\dag,*)$-pair, et qu'un facteur direct d'un objet $(\dag,!)$-pair est $(\dag,!)$-pair.
\item
\label{it:decalage}
Il est clair que ces d\'efinitions d\'ependent du choix de la collection quasi-exceptionnelle. Par exemple, si $(\nabla_i : i \in I)$ est une collection quasi-exception\-nelle et si $\ddag : I \to \Z$ est une fonction, alors $(\nabla_i[\ddag(i)] : i \in I)$ est \'egalement une collection quasi-exceptionnelle (pour le m\^eme ordre sur $I$), pour laquelle les notions de parit\'e seront diff\'erentes en g\'en\'eral. Par exemple, un objet $X$ sera $(\dag,*)$-pair par rapport \`a $(\nabla_i : i \in I)$ ssi il est $(\dag + (\ddag \bmod 2),*)$-pair pour $(\nabla_i[\ddag(i)] : i \in I)$. Cette id\'ee sugg\`ere de ne pas mentionner de fonction $\dag$ dans la d\'efinition, puisqu'elle peut \^etre modifi\'ee en ``renormalisant'' la collection quasi-exceptionnelle. Cependant, les collections quasi-exceptionnelles qu'on rencontrera ci-dessous auront une ``normalisation'' naturelle, et changer celle-ci risque d'amener \`a des confusions.
\end{enumerate}
\end{rmq}

\subsection{Le cas des \'el\'ements minimaux}

Avant d'\'etudier les objets \`a parit\'e en g\'en\'eral, nous devons consid\'erer le cas de la cat\'egorie $\mathscr{D}_{\{\preceq i\}}$ o\`u $i$ est l'\'el\'ement minimal de $I$.

\begin{lem}
\label{lem:parite-minimal}
Si $i$ est minimal dans $I$ (de sorte que $\Delta_i=\nabla_i$), alors pour un objet $X$ de $\mathscr{D}_{\{\preceq i\}}$ les conditions suivantes sont \'equivalentes :
\begin{enumerate}
\item
\label{it:parite-minimal-1}
$X$ est $*$-pair;
\item
\label{it:parite-minimal-2}
$X$ est $!$-pair;
\item
\label{it:parite-minimal-3}
$X$ est isomorphe \`a une somme directe d'objets de la forme $\nabla_i[n] = \Delta_i[n]$ o\`u $n \bmod 2 = \dag(i)$.
\end{enumerate}
\end{lem}

\begin{proof}
On va montrer que~\eqref{it:parite-minimal-1} implique~\eqref{it:parite-minimal-3}; l'implication oppos\'ee est claire par d\'efinition (et d'apr\`es notre hypoth\`ese~\eqref{eqn:Hom-nabla-pair}), et l'\'equivalence entre~\eqref{it:parite-minimal-2} et~\eqref{it:parite-minimal-3} s'obtient par des arguments duaux. 

On raisonne par r\'ecurrence sur le rang de $\Hom^\bullet_{\mathscr{D}}(X, \nabla_i)$. Si ce rang est nul alors $X=0$ (car $\nabla_i$ engendre $\mathscr{D}_{\{\preceq i\}}$) et il n'y a rien \`a d\'emontrer. Sinon, soit $n$ l'entier minimal tel que $\Hom_{\mathscr{D}}(X,\nabla_i[n])\neq 0$. (En particulier, $n \bmod 2 =\dag(i)$.) On consid\`ere la t-structure sur $\mathscr{D}_{\{\preceq i\}}$ associ\'ee \`a la collection quasi-exceptionnelle $(\nabla_i)$. Alors $\nabla_i$ est l'unique objet simple dans le coeur de cette t-structure (voir par exemple~\cite[Proposition~2]{bez}) ; notre hypoth\`ese~\eqref{eqn:Hom-nabla-pair}
 implique donc que tout objet dans le coeur de cette t-structure est une somme directe de copies de $\nabla_i$. On en d\'eduit que $n$ est l'entier minimal 
tel que $\mathcal{H}^{-n}(X) \neq 0$ (pour cette t-structure), et que $\mathcal{H}^{-n}(X)$ est une somme directe de copies de $\nabla_i$.

Consid\'erons le triangle distingu\'e de troncation
\[
\tau_{<-n}(X) \to X \to \mathcal{H}^{-n}(X)[n] \xrightarrow{[1]}.
\]
Alors le morphisme 
\[
\Hom^\bullet_{\mathscr{D}}(\mathcal{H}^{-n}(X)[n], \nabla_i) \to \Hom^\bullet_{\mathscr{D}}(X,\nabla_i)
\]
induit par la deuxi\`eme fl\`eche de ce triangle
est l'inclusion du facteur direct engendr\'e par $\Hom_{\mathscr{D}}(X,\nabla_i[n])$, de sorte que $\tau_{<-n}(X)$ est $*$-pair \'egalement. Par r\'ecurrence c'est une somme directe d'objets de la forme $\nabla_i[m]$ avec $m \bmod 2 = \dag(i)$.
En utilisant encore~\eqref{eqn:Hom-nabla-pair} on voit que le triangle consid\'er\'e ci-dessus est scind\'e, et le r\'esultat suit.
\end{proof}

\subsection{Objets \`a parit\'e : classification}
\label{ss:parite-classification}

Les hypothèses du~\S\ref{ss:parite-def} restent en vigueur. Puisque $\mathscr{D}$ est suppos\'ee de Krull--Schmidt, et au vu de la Remarque~\ref{rmq:def-parite}\eqref{it:parite-facteur-direct}, tout objet \`a parit\'e dans $\mathscr{D}$ est une somme directe d'objets \`a parit\'e ind\'ecomposables.
Le r\'esultat suivant (une variante de~\cite[Theorem~2.12]{jmw}, dont la preuve est essentiellement identique) classifie ces derniers objets \`a isomorphisme pr\`es.

\begin{thm}[Juteau--Mautner--Williamson]
\label{thm:JMW}
Pour tout $i \in I$, il existe au plus un objet \`a parit\'e ind\'ecomposable $E_i$ qui appartient \`a $\mathscr{D}_{\{\preceq i\}}$ et dont l'image dans $\mathscr{D}_{\{\preceq i\}} / \mathscr{D}_{\{\prec i\}}$ est isomorphe \`a celle de $\nabla_i$. De plus, tout objet \`a parit\'e ind\'ecomposable est isomorphe \`a $E_i[n]$ pour un unique couple $(i,n) \in I \times \Z$ tel que $E_i$ existe.
\end{thm}

Lorsqu'il existe, l'objet $E_i$ sera appel\'e l'objet \`a parit\'e ind\'ecomposable \emph{normalis\'e} associ\'e \`a $i$.

\begin{rmq}\phantomsection
\label{rmq:classification}
\begin{enumerate}
\item
Notons que le Th\'eor\`eme~\ref{thm:JMW} n'affirme rien quant \`a l'existence des objets $E_i$. En fait, dans~\cite{jmw} les auteurs donnent des exemples pour lesquels de tels objets n'existent pas.
\item
Consid\'erons le cas particulier o\`u $\mathscr{D}$ est la cat\'egorie d\'eriv\'ee d'une cat\'egorie de plus haut poids $\mathscr{A}$ avec objets standards $(\Delta_i : i \in I)$ et objets costandards $(\nabla_i : i \in I)$ param\'etr\'es par un ensemble bien ordonn\'e $I$. Alors $(\nabla_i : i \in I)$ est une collection exceptionnelle dans $\mathscr{D}$, et $(\Delta_i : i \in I)$ est la collection duale. Les objets basculants de $\mathscr{A}$ sont alors pairs au sens de la D\'efinition~\ref{def:parite} pour la fonction $\dag$ constante \'egale \`a $0$; d'apr\`es le Th\'eor\`eme~\ref{thm:JMW} on en d\'eduit que dans ce cas les objets \`a parit\'e sont les sommes directes de d\'ecal\'es cohomologiques d'objets basculants de $\mathscr{A}$.
\item
\label{it:parite-Ei}
Par d\'efinition,
un objet \`a parit\'e ind\'ecomposable est soit pair, soit impair. Avec nos conventions, l'objet $E_i$ est pair si $\dag(i)=\overline{0}$, et impair si $\dag(i)=\overline{1}$.
\end{enumerate}
\end{rmq}

L'\'etape-cl\'e de la d\'emonstration du th\'eor\`eme sera le lemme suivant.

\begin{lem}
\label{lem:annulation-Hom-parite}
Si $X$ est $*$-pair et $Y$ est $!$-pair, alors on a $\Hom_{\mathscr{D}}(X,Y[n])=0$ pour tout $n$ impair.
\end{lem}

\begin{proof}
On raisonne par r\'ecurrence (transfinie) sur un \'el\'ement $i \in I$ tel que $X$ est dans $\mathscr{D}_{\{\preceq i\}}$. Si $i$ est minimal, le r\'esultat d\'ecoule du Lemme~\ref{lem:parite-minimal} et de la d\'efinition des objets $!$-pairs.

Soit maintenant $i \in I$, et supposons que le r\'esultat est vrai si $X$ appartient \`a $\mathscr{D}_{\{\prec i\}}$. Rappelons que d'apr\`es~\cite[Lemma~4]{bez} le foncteur de projection
\[
\Pi : \mathscr{D}_{\{\preceq i\}} \to \mathscr{D}_{\{\preceq i\}}/\mathscr{D}_{\{\prec i\}}
\]
admet un adjoint \`a gauche $\Pi^l$ v\'erifiant $\Pi^l \Pi(\Delta_i) \cong \Delta_i$ et un adjoint \`a droite $\Pi^r$ v\'erifiant $\Pi^r \Pi(\nabla_i) \cong \nabla_i$, et que l'inclusion
\[
\iota : \mathscr{D}_{\{\prec i\}} \to \mathscr{D}_{\{\preceq i\}}
\]
admet \'egalement un adjoint \`a gauche $\iota^l$ et un adjoint \`a droite $\iota^r$. De plus, ces donn\'ees v\'erifient les axiomes de ``recollement'' de~\cite[\S 1.4]{bbd}; en particulier on a un triangle distingu\'e fonctoriel
\begin{equation}
\label{eqn:triangle-recollement}
\Pi^l \Pi(Z) \to Z \to \iota \iota^l(Z) \xrightarrow{[1]}
\end{equation}
pour tout $Z$ dans $\mathscr{D}_{\{\preceq i\}}$.

Il n'est pas difficile de v\'erifier que l'objet $\iota \iota^l(X)$ est $*$-pair, de sorte qu'on a $\Hom_{\mathscr{D}}(\iota \iota^l(X),Y[n])=0$ pour $n$ impair par r\'ecurrence. D'un autre c\^ot\'e, puisque $\Pi^r \Pi(\nabla_i) \cong \nabla_i$ l'objet $\Pi(X)$ est $*$-pair dans $\mathscr{D}_{\{\preceq i\}} / \mathscr{D}_{\{\prec i\}}$ (pour la suite quasi-exceptionnelle $(\Pi(\nabla_i))$); d'apr\`es le Lemme~\ref{lem:parite-minimal} ceci implique que $\Pi(X)$ est une somme directe d'objets $\Pi(\Delta_i)[n]$ avec $n \bmod 2 = \dag(i)$, et donc que $\Pi^l \Pi(X)$ est une somme directe d'objets $\Delta_i[n]$ avec $n \bmod 2 = \dag(i)$. Le r\'esultat voulu suit en consid\'erant la suite exacte longue obtenue en appliquant le foncteur $\Hom_{\mathscr{D}}(-,Y)$ au triangle~\eqref{eqn:triangle-recollement} (appliqu\'e \`a $Z=X$).
\end{proof}

\begin{proof}[Démonstration du Théorème~{\rm \ref{thm:JMW}}]
Supposons qu'il existe deux objets $E_i$ et $F_i$ satisfaisant les conditions de l'\'enonc\'e. Soient $\Pi$ et $\iota$ comme dans la preuve du Lemme~\ref{lem:annulation-Hom-parite}, et fixons des isomorphismes $\Pi(E_i) \cong \Pi(\nabla_i)$ et $\Pi(F_i) \cong \Pi(\nabla_i)$. Consid\'erons le triangle distingu\'e~\eqref{eqn:triangle-recollement} pour $Z=E_i$ :
\[
\Pi^l \Pi(E_i) \to E_i \to \iota \iota^l(E_i) \xrightarrow{[1]}.
\]
Comme dans la preuve du Lemme~\ref{lem:annulation-Hom-parite} l'objet $\iota \iota^l(E_i)$ est $*$-pair, de sorte que le Lemme~\ref{lem:annulation-Hom-parite} implique que le morphisme naturel
\begin{multline*}
\Hom(E_i, F_i) \to \Hom(\Pi^l \Pi(E_i), F_i) \cong \Hom(\Pi(E_i),\Pi(F_i)) \\
= \Hom(\Pi(\nabla_i),\Pi(\nabla_i)) = \Hom(\nabla_i,\nabla_i)
\end{multline*}
est surjectif. Choisissons un morphisme $f : E_i \to F_i$ dont l'image est $\id_{\nabla_i}$. De m\^eme, il existe un morphisme  $g : F_i \to E_i$ tel que $\Pi(g)=\Pi(\id_{\nabla_i})$. Alors $g \circ f$ est un élément de l'anneau local $\End(E_i)$ dont l'image dans son quotient $\End(\nabla_i)$ est inversible, et donc $g \circ f$ est un isomorphisme. De m\^eme, $f \circ g$ est un isomorphisme, 
donc $f$ et $g$ \'egalement.

Pour un objet \`a parit\'e indécomposable $E$ g\'en\'eral, on choisit $i$ minimal tel que $E$ appartient \`a $\mathscr{D}_{\{\preceq i\}}$. Alors l'image de $E$ dans $\mathscr{D}_{\{\preceq i\}} / \mathscr{D}_{\{\prec i\}}$ est non nulle, \`a parit\'e et ind\'ecomposable par le m\^eme argument que ci-dessus ; il est donc isomorphe \`a l'image de $\nabla_i[n]$ pour un $n \in \Z$ d'apr\`es le Lemme~\ref{lem:parite-minimal}. On en d\'eduit que $E\cong E_i[n]$.
\end{proof}

\subsection{L'exemple des faisceaux constructibles}
\label{ss:comparaison-JMW}

Notre d\'efinition des objets \`a parit\'e est inspir\'ee de celle consid\'er\'ee dans~\cite{jmw}. Le contexte que ces auteurs consid\`erent est celui de l'Exemple~\ref{ex:collections-qe}\eqref{it:qe-faisceaux}. Pour que la collection quasi-exceptionnelle consid\'er\'ee dans cet exemple soit \`a parit\'e, il faut supposer que pour tout $s \in \mathscr{S}$ et tout syst\`eme local $H$-\'equivariant simple $\mathscr{L}$ sur $X_s$ on a
\[
\Hom_{\Db_H(X_s,\bk)}(\mathscr{L},\mathscr{L}[n])=0 \quad \text{si $n$ est impair.}
\]

Les d\'efinitions des objets pairs et impairs dans la D\'efinition~\ref{def:parite} (dans ce cas particulier) sont formellement diff\'erentes de celles donn\'ees dans~\cite[Definition~2.4]{jmw}, mais les deux d\'efinitions sont \'equivalentes. (Ce fait peut se d\'emontrer par r\'ecurrence sur un \'el\'ement $i$ tel que $X$ appartient \`a $\mathscr{D}_{\{\preceq i\}}$, en utilisant le Lemme~\ref{lem:parite-minimal}, le triangle distingu\'e~\eqref{eqn:triangle-recollement}, et le triangle ``dual'' faisant intervenir les adjoints \`a droite.)

Les normalisations que nous utilisons sont \'egalement diff\'erentes de celles utilis\'ees dans~\cite{jmw} : par exemple, un objet $X$ de $\Db_{H,\mathscr{S}}(X,\bk)$ est $(\dag,*)$-pair au sens du pr\'esent article ssi il est $(\dag',*)$-pair au sens de~\cite{jmw}, o\`u $\dag'(s)=\dag(s)+(\dim(X_s) \bmod 2)$. Dans la pratique, le cas 
qui intervient dans toutes les applications connues de cette th\'eorie
est celui o\`u on prend $\dag(s)=\dim(X_s) \bmod 2$ (correspondant au cas de la ``pariversit\'e constante'' $\natural$ dans les conventions de~\cite{jmw}).

\begin{rmq}
Dans ce contexte particulier,
les auteurs de~\cite{jmw} consid\`erent une situation un petit peu plus g\'en\'erale : ils autorisent des anneaux de coefficients plus g\'en\'eraux, et font l'hypoth\`ese plus faible que
\[
\Hom_{\Db_H(X_s,\bk)}(\mathscr{L},\mathscr{L}'[n])=0 \quad \text{si $n$ est impair}
\]
pour tous $s \in \mathscr{S}$ et $\mathscr{L}, \mathscr{L}'$ syst\`emes locaux $H$-\'equivariants simples sur $X_s$. Sans nos hypoth\`eses suppl\'ementaires, ce cadre ne rentre pas dans le formalisme des collections quasi-exceptionnelles (mais on dispose tout de m\^eme d'un formalisme de recollement qui permet de consid\'erer des constructions similaires).
\end{rmq}

\subsection{Version gradu\'ee}
\label{ss:graduee}

En plus du contexte \'etudi\'e aux~\S\S\ref{ss:parite-def}--\ref{ss:parite-classification}, il est utile de consid\'erer une situation ``gradu\'ee'' dans laquelle la cat\'egorie triangul\'ee $\mathscr{D}$ admet un automorphisme triangul\'e $\langle 1 \rangle$. Pour tout $n \in \Z$, on notera alors $\langle n \rangle$ la puissance $n$-i\`eme de $\langle 1 \rangle$.
On supposera de nouveau que $\mathscr{D}$ est de Krull--Schmidt. Si $X$ et $Y$ sont dans $\mathscr{D}$ on notera
\[
\uHom^{\bullet}_{\mathscr{D}}(X,Y) := \bigoplus_{n,m \in \Z} \Hom_{\mathscr{D}}(X,Y\langle n \rangle [m]),
\]
avec la $\Z^2$-graduation \'evidente.
On consid\'erera \'egalement l'automorphisme
\[
\{1\}:=\langle -1 \rangle [1]
\]
de $\mathscr{D}$.

Si $(I,\preceq)$ est un ensemble bien ordonn\'e, une \emph{collection quasi-exceptionnelle gradu\'ee} est une collection $(\nabla_i : i \in I)$ telle que
\begin{enumerate}
\item
si $i,j \in I$ et
$\uHom_{\mathscr{D}}^\bullet(\nabla_i, \nabla_j) \neq 0$, alors $i \succeq j$;
\item
si $i \in I$ alors $\Hom_{\mathscr{D}}(\nabla_i, \nabla_i\langle m \rangle[n])=0$ pour tous $m,n \in \Z$ tels que $n<0$ ou $n=0$ et $m \neq 0$, et $\Hom_{\mathscr{D}}(\nabla_i, \nabla_i)$ est une alg\`ebre \`a division.
\end{enumerate}
On d\'efinit alors $\mathscr{D}_{\{\preceq i\}}$ et $\mathscr{D}_{\{\prec i\}}$ comme les sous-cat\'egories triangul\'ees engendr\'ees par les objets de la forme $\nabla_j \langle m \rangle$ avec $j \preceq i$ et $m \in \Z$ et avec $j \prec i$ et $m \in \Z$ respectivement. On peut d\'efinir de fa{\c c}on \'evidente la collection duale, et la th\'eorie d\'evelopp\'ee dans~\cite{bez} s'\'etend ais\'ement \`a ce cadre.

Si $(\nabla_i : i \in I)$ est une collection quasi-exceptionnelle gradu\'ee admettant une collection duale $(\Delta_i : i \in I)$, on dira que $(\nabla_i : i \in I)$ est \emph{\`a parit\'e} si
\[
\Hom_{\mathscr{D}}(\nabla_i, \nabla_i \{ m \}[n])=0 \quad \text{si $n \neq 0$ ou si $m$ est impair.}
\]
Sous ces conditions, l'anneau $\Z^2$-gradu\'e $\uHom^\bullet_{\mathscr{D}}(\nabla_i,\nabla_i) = \uHom^\bullet_{\mathscr{D}}(\Delta_i,\Delta_i)$ est con\-cen\-tr\'e sur $\{(n,-n) : n \in 2\Z_{\geq 0}\} \subset \Z^2$. 
Dans la suite, on supposera de plus que la famille
$(\nabla_i \langle m \rangle : i \in I, \, m \in \Z)$ engendre $\mathscr{D}$ comme cat\'egorie triangul\'ee.

Si on se fixe une fonction $\dag : I \to \Z/2\Z$,
on dira alors qu'un objet $X$ de $\mathscr{D}$ est
\begin{enumerate}
\item
$(\dag,*)$-pair si pour tout $i \in I$ le $\uHom^{\bullet}_{\mathscr{D}}(\nabla_i, \nabla_i)$-module \`a gauche $\Z^2$-gradu\'e
\[
\uHom^\bullet_{\mathscr{D}}(X, \nabla_i)
\]
est libre de rang fini et concentr\'e sur $\{(n,-n) : n \in \Z, \, n \bmod 2 = \dag(i)\} \subset \Z^2$;
\item
$(\dag,!)$-pair si pour tout $i \in I$ le $\uHom^\bullet_{\mathscr{D}}(\Delta_i, \Delta_i)$-module \`a droite $\Z^2$-gradu\'e
\[
\uHom_{\mathscr{D}}^\bullet(\Delta_i,X)
\]
est libre de rang fini et concentr\'e sur $\{(n,-n) : n \in \Z, \, n \bmod 2 = \dag(i)\} \subset \Z^2$;
\item
$\dag$-pair si $X$ est $(\dag,*)$-pair et $(\dag,!)$-pair;
\item
$(\dag,*)$-impair si $X \{1\}$ est $(\dag,*)$-pair;
\item
$(\dag,!)$-impair si $X \{1\}$ est $(\dag,!)$-pair;
\item
$\dag$-impair si $X \{1\}$ est $\dag$-pair;
\item
\`a $\dag$-parit\'e si $X$ est isomorphe \`a une somme $X' \oplus X''$ avec $X'$ $\dag$-pair et $X''$ $\dag$-impair.
\end{enumerate}

Comme dans le cas non gradu\'e, la fonction $\dag$ sera souvent omise des notations.

\begin{rmq}
\begin{enumerate}
\item
Cette notion a d\'ej\`a \'et\'e consid\'er\'ee, dans le cas particulier o\`u $\mathscr{D}$ est la cat\'egorie d\'eriv\'ee d'une cat\'egorie de plus haut poids gradu\'ee $\mathscr{A}$ (pour la suite quasi-exceptionnelle gradu\'ee form\'ee par les objets costandards de $\mathscr{A}$), dans~\cite{ahr}.
\item
La sous-cat\'egorie pleine de $\mathscr{D}$ dont les objets sont les objets \`a $\dag$-parit\'e est stable par le d\'ecalage $\{ 1 \}$.
\end{enumerate}
\end{rmq}

Dans ce contexte le Th\'eor\`eme~\ref{thm:JMW} admet la variante suivante, dont la preuve est essentiellement identique.

\begin{thm}
\label{thm:JMW-gradue}
Pour tout $i \in I$, il existe au plus un objet \`a parit\'e ind\'ecomposable $E_i$ qui appartient \`a $\mathscr{D}_{\{\preceq i\}}$ et dont l'image dans $\mathscr{D}_{\{\preceq i\}} / \mathscr{D}_{\{\prec i\}}$ est isomorphe \`a celle de $\nabla_i$. De plus, tout objet \`a parit\'e ind\'ecomposable est isomorphe \`a $E_i\{ n \}$ pour un unique couple $(i,n) \in I \times \Z$ tel que $E_i$ existe.
\end{thm}

\begin{rmq}
Comme dans la Remarque~\ref{rmq:classification}\eqref{it:parite-Ei}, l'objet $E_i$ est pair si $\dag(i)=\overline{0}$, et impair si $\dag(i)=\overline{1}$.
\end{rmq}

\begin{ex}
\label{ex:parite-basculants}
Soit $\mathscr{A}$ une cat\'egorie de plus haut poids gradu\'ee, d'objets costandards et standards associ\'es $(\nabla_i : i \in I)$ et $(\Delta_i : i \in I)$ respectivement. Posons $\mathscr{D}:=\Db (\mathscr{A})$ et consid\'erons la collection quasi-exceptionnelle gradu\'ee form\'ee par les $(\nabla_i : i \in I)$. Supposons qu'il existe une fonction $\dag : I \to \Z/2\Z$ telle que les objets basculants ind\'ecomposables normalis\'es $(\mathsf{T}_i : i \in I)$ v\'erifient
\[
\Hom_{\mathscr{A}}(\Delta_i, \mathsf{T}_j \langle n \rangle) \neq 0 \quad \Rightarrow \quad n \bmod 2 = \dag(i) + \dag(j).
\]
Plut\^ot que de consid\'erer l'automorphisme de $\mathscr{D}$ induit par l'automorphisme $\langle 1 \rangle$ de $\mathscr{A}$ (qu'on notera similairement), consid\'erons l'automorphisme
$\langle\hspace{-1pt}\langle 1 \rangle\hspace{-1pt}\rangle := \langle -1 \rangle [1]$. Alors puisque
\[
\Hom_{\mathscr{D}}(\Delta_i, \mathsf{T}_j \langle\hspace{-1pt}\langle n \rangle\hspace{-1pt}\rangle[m]) = \Hom_{\mathscr{D}}(\Delta_i, \mathsf{T}_j \langle -n \rangle [n+m]),
\]
si le terme de gauche est non nul on doit avoir $n+m=0$ et $(n \bmod 2)+\dag(j)=\dag(i)$. Donc $\mathsf{T}_j$ est pair si $\dag(j)=\overline{0}$, et impair si $\dag(j)=\overline{1}$. On en d\'eduit que, dans ce contexte, pour tout $i \in I$ l'objet $E_i$ du Th\'eor\`eme~\ref{thm:JMW-gradue} existe, et est isomorphe \`a $\mathsf{T}_i$.
\end{ex}

\subsection{Foncteurs de d\'egraduation}
\label{ss:degrad}

Soit $\mathscr{C}$ une catégorie triangulée munie d'un automorphisme $\la 1\ra$ comme au~\S\ref{ss:graduee}, et soit $\overline{\mathscr{D}}$ une autre catégorie triangulée. Un \emph{foncteur de dégraduation} (par rapport à l'automorphisme $\la 1\ra$) est un foncteur triangulé $\Phi: \mathscr{C} \to \overline{\mathscr{D}}$ muni d'un isomorphisme naturel
\[
\varepsilon: \Phi \circ \la 1\ra \simto \Phi
\]
qui vérifie les deux conditions suivantes:
\begin{enumerate}
\item 
\label{it:degr-1}
les objets dans l'image essentielle de $\Phi$ engendrent $\overline{\mathscr{D}}$ en tant que catégorie triangulée;
\item 
\label{it:degr-2}
si $X$ et $Y$ sont dans $\mathscr{C}$, alors l'application
\[
\uHom_{\mathscr{C}}(X,Y) = \bigoplus_{n \in \Z} \Hom_{\mathscr{C}}(X, Y\la n\ra) \longrightarrow \Hom_{\overline{\mathscr{D}}}(\Phi(X),\Phi(Y))
\]
induite par $\Phi$ et $\varepsilon$ est un isomorphisme.
\end{enumerate}
Moralement, un foncteur de dégraduation est ``aussi proche que possible d'\^etre une équivalence de catégories'' entre $\mathscr{C}$ et $\overline{\mathscr{D}}$, compte tenu du fait que ces cat\'egories ont des structures diff\'erentes : la condition~\eqref{it:degr-1} ci-dessus remplace la condition d'être essentiellement surjectif, et la condition~\eqref{it:degr-1} remplace la condition d'être pleinement fidèle. 

\begin{ex}
Soit $A$ un anneau gradué Noetherien. Si $A\lgmod$ d\'esigne la cat\'egorie des $A$-modules gradu\'es de type fini et $A\lmod$ celle des $A$-modules de type fini, alors le foncteur d'oubli $\mathrm{For} : \Db(A\lgmod) \to \Db(A\lmod)$ vérifie toujours la propriété~\eqref{it:degr-2} ci-dessus. Si $A$ est Artinien (comme anneau non gradu\'e), alors la cat\'egorie triangul\'ee $\Db(A\lmod)$ est engendr\'ee par les $A$-modules simples, et ceux-ci appartiennent \`a l'image essentielle de $\mathrm{For}$ (voir~\cite[Proposition~3.5]{gg}). Celui-ci est donc un foncteur de d\'egraduation.
\end{ex}

Consid\'erons encore $\mathscr{C}$, $\overline{\mathscr{D}}$ et $\Phi$ comme ci-dessus. Il est clair que si $(\nabla_i: i \in I)$ est une collection quasi-exceptionelle graduée dans $\mathscr{C}$, alors $(\Phi(\nabla_i) : i \in I)$ est une collection quasi-exceptionelle dans $\overline{\mathscr{D}}$. Il est clair \'egalement que, pour ces collections, l'image par $\Phi$ d'un objet $(\dag,*)$-pair, resp.~$(\dag,!)$-pair, resp.~pair, est $(\dag,*)$-pair, resp.~$(\dag,!)$-pair, resp.~pair, et de m\^eme pour les objets impairs ou \`a parit\'e.

\section{Dualit\'e de Koszul formelle}
\label{sec:Koszul-formelle}

\subsection{Contexte}
\label{ss:Koszul-contexte}

On consid\'erera \`a partir de maintenant le contexte suivant: $\bk$ d\'esignera un corps, $\mathscr{D}$ sera une cat\'egorie triangul\'ee $\bk$-lin\'eaire munie d'un automorphisme $\langle 1 \rangle$. On supposera donn\'es un ensemble bien ordonn\'e $(I,\preceq)$ isomorphe \`a $(\Z_{\ge 0}, \leq)$ et deux familles d'objets $(\Delta_i : i \in I)$ et $(\nabla_i : i \in I)$ v\'erifiant les conditions suivantes :
\begin{enumerate}
\item
la famille $(\nabla_i \langle n \rangle : i \in I, \, n \in \Z)$ engendre $\mathscr{D}$ comme cat\'egorie triangul\'ee;
\item
si $\uHom^\bullet_{\mathscr{D}}(\nabla_i, \nabla_j) \neq 0$ alors $i \succeq j$;
\item
pour tout $i \in I$ et tous $n,m \in \Z$ on a $\Hom_{\mathscr{D}}(\nabla_i, \nabla_i \langle n \rangle [m])=0$ si $(n,m) \neq (0,0)$, et de plus $\End_{\mathscr{D}}(\nabla_i)=\bk$;
\item
pour tous $i \succ j$ on a $\uHom^\bullet_{\mathscr{D}}(\Delta_i, \nabla_j)=0$, et $\Delta_i$ et $\nabla_i$ ont des images isomorphes dans $\mathscr{D}_{\{\preceq i\}} / \mathscr{D}_{\{\prec i\}}$;
\item
\label{it:condition-coeur}
pour tous $i,j \in I$ et tous $n,m \in \Z$, si $\Hom(\nabla_i, \nabla_j \langle m \rangle [n]) \neq 0$ alors $n\geq0$, et de m\^eme si $\Hom(\Delta_i, \Delta_j \langle m \rangle [n]) \neq 0$ alors $n \geq 0$.
\end{enumerate}

En particulier, ces conditions impliquent que pour tous $X,Y$ dans $\mathscr{D}$ le $\bk$-espace vectoriel
\[
\bigoplus_{n,m \in \Z} \Hom_{\mathscr{D}}(X,Y \langle n \rangle [m])
\]
est de dimension finie, et donc que $\mathscr{D}$ est de Krull--Schmidt (voir notamment~\cite[Corollary~A.2]{cyz}). De plus
 $(\nabla_i : i \in I)$ est une collection exceptionnelle gradu\'ee dans $\mathscr{D}$ (voir l'Exemple~\ref{ex:collections-qe}\eqref{it:exceptionnelle}), de collection duale $(\Delta_i : i \in I)$. Notre condition~\eqref{it:condition-coeur} implique que les objets $\Delta_i$ et $\nabla_i$ appartiennent au coeur $\mathscr{A}$ de la t-structure associ\'ee \`a cette collection exceptionnelle, et des arguments standards montrent alors que $\mathscr{A}$ est une cat\'egorie gradu\'ee de plus haut poids au sens du~\S\ref{ss:cat-qhered}, pour la restriction de $\langle 1 \rangle$ \`a $\mathscr{A}$, et avec objets costandards les $(\nabla_i : i \in I)$ et objets standards les $(\Delta_i : i \in I)$. En particulier, ceci permet de consid\'erer les objets basculants ind\'ecomposables $(\mathsf{T}_i : i \in I)$ dans cette cat\'egorie. 
 
Ci-dessous nous consid\'ererons les objets \`a parit\'e dans cette cat\'egorie au sens du~\S\ref{ss:graduee}, principalement par rapport au choix de l'automorphisme $\langle 1 \rangle$ et de la collection quasi-exceptionnelle $(\nabla_i : i \in I)$. Mais on peut \'egalement consid\'erer ces objets pour le choix d'automorphisme $\lla 1 \rra := \{1\} = \langle -1 \rangle[1]$ plut\^ot que $\langle 1 \rangle$ (et la m\^eme collection quasi-exceptionnelle gradu\'ee). L'\'enonc\'e suivant est une sorte de r\'eciproque de la situation consid\'er\'ee dans l'Exemple~\ref{ex:parite-basculants}.

\begin{lem}
\label{lem:parite-lla}
Soit $\dag : I \to \Z/2\Z$ une fonction. Tout objet \`a parit\'e $E$ pour la fonction $\dag$ et l'automorphisme $\lla 1 \rra$ appartient \`a $\mathscr{A}$, et est basculant dans cette cat\'egorie.
En particulier, s'il existe un objet ind\'ecomposable $E_i$ comme dans le Th\'eor\`eme~{\rm \ref{thm:JMW-gradue}} pour la fonction $\dag$ et l'automorphisme $\lla 1 \rra$, alors $E_i \cong \mathsf{T}_i$.
\end{lem}

\begin{proof}
Soit $E$ comme dans l'\'enonc\'e.
Par d\'efinition, pour tout $j \in I$ l'espace
\[
\Hom_{\mathscr{D}}(\Delta_j, E \la k \ra [l]) = \Hom_{\mathscr{D}}(\Delta_j, E \lla -k \rra [l+k])
\]
s'annule sauf si $l=0$. De m\^eme, on a
\[
\Hom_{\mathscr{D}}(E,\nabla_j \la k \ra [l])=0
\]
sauf si $l=0$. D'apr\`es un argument classique (voir~\cite[Lemma~4]{bez:ctm}), ceci implique que $E$ appartient \`a $\mathscr{A}$ et est basculant. Si $E=E_i$ pour un $i$, il est facile de voir que cet objet v\'erifie les conditions qui caract\'erisent $\mathsf{T}_i$.
\end{proof}

\begin{rmq}
Le Lemme~\ref{lem:parite-lla} montre que l'existence d'objets \`a parit\'e impose de fortes contraintes sur la fonction $\dag$. Dans la pratique, cette fonction sera toujours impos\'ee par le contexte \'etudi\'e.
\end{rmq}
 
\subsection{D\'efinition}

On consid\`ere maintenant deux cat\'egories $\mathscr{C}$ et $\mathscr{D}$ munies des structures consid\'er\'ees au~\S\ref{ss:Koszul-contexte}. Par souci de clart\'e, les donn\'ees associ\'ees seront not\'ees avec un indice $\mathscr{C}$ ou $\mathscr{D}$.

\begin{defn}
\label{def:dualite-Koszul}
Une \emph{dualit\'e de Koszul formelle} est un foncteur triangul\'e $\Phi : \mathscr{C} \to \mathscr{D}$ qui v\'erifie les conditions suivantes :
\begin{enumerate}
\item
$\Phi \circ \langle 1 \rangle_{\mathscr{C}} \cong \{1\}_{\mathscr{D}} \circ \Phi$;
\item
\label{it:def-Koszul-dn}
pour tout $i \in I_\mathscr{C}$, il existe un indice $\varphi(i) \in I_{\mathscr{D}}$ tel que 
\[
\Phi(\Delta_i^{\mathscr{C}}) \cong \Delta_{\varphi(i)}^{\mathscr{D}} \quad \text{ et } \quad \Phi(\nabla_i^{\mathscr{C}}) \cong \nabla_{\varphi(i)}^{\mathscr{D}};
\]
\item
\label{it:def-Koszul-I}
la fonction $\varphi$ d\'etermin\'ee par la condition~\eqref{it:def-Koszul-dn} est une bijection $I_{\mathscr{C}} \simto I_{\mathscr{D}}$ qui identifie les ordres $\preceq_{\mathscr{C}}$ et $\preceq_{\mathscr{D}}$.
\end{enumerate}
\end{defn}

\subsection{Propri\'et\'es}

\begin{lem}
\label{lem:Koszul-equiv}
Si $\Phi$ est une dualit\'e de Koszul formelle, alors c'est une \'equivalence de cat\'egories, et $\Phi^{-1}$ est \'egalement une dualit\'e de Koszul formelle.
\end{lem}

\begin{proof}
Les conditions impos\'ees dans la D\'efinition~\ref{def:dualite-Koszul} montrent que pour tout $i \in I_{\mathscr{C}}$ on a $\Phi(\mathscr{C}_{\{\prec_{\mathscr{C}} i\}}) \subset \mathscr{D}_{\{\prec_{\mathscr{D}} \varphi(i)\}}$. Comme le c\^one de tout morphisme non nul $\Delta_i^{\mathscr{C}} \to \nabla_i^\mathscr{C}$ appartient \`a $\mathscr{C}_{\{\prec_{\mathscr{C}} i\}}$, ceci montre que le c\^one de son image par $\Phi$ appartient \`a $\mathscr{D}_{\{\prec_{\mathscr{D}} \varphi(i)\}}$; en particulier, ce morphisme est non nul. Donc $\Phi$ induit un isomorphisme
\[
\Hom_{\mathscr{C}}(\Delta^{\mathscr{C}}_i, \nabla^{\mathscr{C}}_j \la m \ra_{\mathscr{C}} [n]) \simto \Hom_{\mathscr{D}}(\Phi(\Delta^{\mathscr{C}}_i), \Phi(\nabla^{\mathscr{C}}_j \la m \ra_{\mathscr{C}}) [n])
\]
pour tous $i,j \in I_{\mathscr{C}}$ et $n,m \in \Z$. Comme les objets $(\nabla^{\mathscr{C}}_i \langle m \rangle_{\mathscr{C}} : i \in I, \, m \in \Z)$ et les objets $(\Delta^{\mathscr{C}}_i \langle m \rangle_{\mathscr{C}} : i \in I, \, m \in \Z)$ engendrent chacun $\mathscr{C}$ comme cat\'egorie triangul\'ee, par un argument classique on en d\'eduit que $\Phi$ est pleinement fid\`ele. Enfin, puisque les objets $(\nabla^{\mathscr{D}}_i \langle m \rangle_{\mathscr{D}} : i \in I, \, m \in \Z)$ engendrent $\mathscr{D}$ comme cat\'egorie triangul\'ee, $\Phi$ est \'egalement essentiellement surjective.

L'assertion concernant $\Phi^{-1}$ est \'evidente.
\end{proof}

\begin{prop}
\label{prop:Koszul-parite-basculant}
Soit $\dag : I_{\mathscr{C}} \to \Z/2\Z$ une fonction, et soit
$\Phi : \mathscr{C} \to \mathscr{D}$ une dualit\'e de Koszul formelle. Alors pour tout objet $E$ \`a $\dag$-parit\'e dans $\mathscr{C}$, $\Phi(E)$ appartient \`a $\mathscr{A}_{\mathscr{D}}$ et est basculant. En particulier, si $i \in I$ et s'il existe un objet ind\'ecomposable $E_i^{\mathscr{C}}$ comme dans le Th\'eor\`eme~{\rm \ref{thm:JMW-gradue}}, alors $\Phi(E_i^{\mathscr{C}}) \cong \mathsf{T}_{\varphi(i)}^{\mathscr{D}}$.
\end{prop}

\begin{proof}
Puisque $\Phi \circ \la 1 \ra_{\mathscr{C}} \cong \{1\}_{\mathscr{D}} \circ \Phi$, si $E$ est \`a $\dag$-parit\'e dans $\mathscr{C}$ alors $\Phi(E)$ est \`a $\dag$-parit\'e dans $\mathscr{D}$, pour le choix de foncteur de d\'ecalage $\lla 1 \rra := \{1\}_{\mathscr{D}}$. L'\'enonc\'e d\'ecoule alors du Lemme~\ref{lem:parite-lla}.
\end{proof}

\subsection{Version non gradu\'ee}
\label{ss:Koszul-mod-nongrad}

On consid\`ere maintenant une cat\'egorie triangul\'ee $\mathscr{C}$ comme au~\S\ref{ss:Koszul-contexte}, et une cat\'egorie triangul\'ee $\overline{\mathscr{D}}$ munie d'une structure similaire mais \emph{non gradu\'ee}, c'est-\`a-dire de collections $(\overline{\nabla}_i : i \in \overline{I})$ et $(\overline{\Delta}_i : i \in \overline{I})$, telles que
\begin{enumerate}
\item
la famille $(\overline{\nabla}_i : i \in I)$ engendre $\overline{\mathscr{D}}$ comme cat\'egorie triangul\'ee;
\item
si $\Hom^{\bullet}_{\overline{\mathscr{D}}}(\overline{\nabla}_i, \overline{\nabla}_j) \neq 0$ alors $i \succeq j$;
\item
pour tout $i \in I$ on a $\Hom_{\overline{\mathscr{D}}}(\overline{\nabla}_i, \overline{\nabla}_i[m])=0$ si $m \neq 0$, et de plus $\End_{\overline{\mathscr{D}}}(\overline{\nabla}_i)=\bk$;
\item
pour tous $i \succ j$ on a $\Hom^{\bullet}_{\overline{\mathscr{D}}}(\overline{\Delta}_i, \overline{\nabla}_j)=0$, et $\overline{\Delta}_i$ et $\overline{\nabla}_i$ ont des images isomorphes dans $\overline{\mathscr{D}}_{\{\preceq i\}} / \overline{\mathscr{D}}_{\{\prec i\}}$;
\item
\label{it:conditiondg-coeur}
pour tous $i,j \in I$, si $\Hom(\overline{\nabla}_i, \overline{\nabla}_j [n]) \neq 0$ alors $n\geq0$, et de m\^eme si $\Hom(\overline{\Delta}_i, \overline{\Delta}_j [n]) \neq 0$ alors $n \geq 0$.
\end{enumerate}

Comme au~\S\ref{ss:Koszul-contexte}, ces conditions impliquent que $(\overline{\nabla}_i : i \in I)$ est une collection quasi-exceptionnelle, avec collection duale $(\overline{\Delta}_i : i \in I)$ et que le coeur $\mathscr{A}_{\overline{\mathscr{D}}}$ de la t-structure associ\'ee a une structure naturelle de cat\'egorie de plus haut poids. Les objets basculants ind\'ecomposables dans cette cat\'egorie seront not\'es $(\overline{\mathsf{T}}_i : i \in I)$.

\begin{defn}
\label{def:dualite-Koszul-degrad}
Une \emph{dualit\'e de Koszul formelle d\'egraduante} est un foncteur triangul\'e $\Phi : \mathscr{C} \to \overline{\mathscr{D}}$ qui v\'erifie les conditions suivantes :
\begin{enumerate}
\item
$\Phi \circ \langle 1 \rangle \cong [1] \circ \Phi$;
\item
\label{it:def-Koszuldg-dn}
pour tout $i \in I$, il existe un indice $\varphi(i) \in \overline{I}$ tel que 
\[
\Phi(\Delta_i) \cong \overline{\Delta}_{\varphi(i)} \quad \text{ et } \quad \Phi(\nabla_i) \cong \overline{\nabla}_{\varphi(i)};
\]
\item
\label{it:def-Koszuldg-I}
la fonction $\varphi$ d\'etermin\'ee par la condition~\eqref{it:def-Koszuldg-dn} est une bijection $I \simto \overline{I}$ qui identifie les ordres correspondant.
\end{enumerate}
\end{defn}

Moralement, on consid\'erera une dualit\'e de Koszul d\'egraduante comme ``la compos\'ee d'une dualit\'e de Koszul formelle et d'un foncteur de d\'egraduation'' (comme d\'efini au~\S\ref{ss:degrad}).

Les \'enonc\'es suivants sont des analogues du Lemme~\ref{lem:Koszul-equiv} et de la Proposition~\ref{prop:Koszul-parite-basculant}, et d\'ecoulent d'arguments similaires.

\begin{lem}
\label{lem:koszul-degrad}
Si $\Phi$ est une dualit\'e de Koszul formelle d\'egraduante, alors c'est un foncteur de d\'egraduation pour l'automorphisme $\langle -1 \rangle [1]$.
\end{lem}

\begin{prop}
\label{prop:Koszul-parite-basculant-degrad}
Soit $\dag : I \to \Z/2\Z$ une fonction, et soit
$\Phi : \mathscr{C} \to \overline{\mathscr{D}}$ une dualit\'e de Koszul formelle dégraduante. Alors pour tout objet $E$ \`a $\dag$-parit\'e dans $\mathscr{C}$, $\Phi(E)$ appartient \`a $\mathscr{A}_{\overline{\mathscr{D}}}$ et est basculant. En particulier, si $i \in I$ et s'il existe un objet ind\'ecomposable $E_i$ comme dans le Th\'eor\`eme~{\rm \ref{thm:JMW-gradue}}, alors $\Phi(E_i) \cong \overline{\mathsf{T}}_{\varphi(i)}$.
\end{prop}

\part{Exemples et applications} 
\label{pt:exemples}

\section{Exemples}
\label{sec:exemples}

\subsection{Cat\'egories d\'eriv\'ees mixtes pour les vari\'et\'es de drapeaux des groupes de Kac--Moody}
\label{ss:Dmix}

Soit $\mathscr{G}$ un groupe de Kac--Moody, et soit $\WKM$ son groupe de Weyl. (Voir le~\S\ref{ss:bez-yun} pour des rappels et le d\'etail des notations que nous utiliserons ci-dessous).
Choisissons un ordre $\preceq$ sur $\WKM$ \'etendant l'ordre de Bruhat et tel que $(\WKM,\preceq)$ est isomorphe \`a $(\Z_{\ge 0},\leq)$.

Soit $\Flag_\scG$ la variété de drapeaux de $\mathscr{G}$, et notons $\Db_{\mathrm{Br}}(\Flag_\mathscr{G},\bk)$ la catégorie des complexes de $\bk$-faisceaux sur $\Flag_{\mathscr{G}}$, constructibles par rapport \`a la stratification~\eqref{eqn:stratif-Bruhat}. Cette catégorie poss\`ede une collection exceptionnelle naturelle, donn\'ee par
\[
\nabla_{\mathscr{G},w} := i_{w*} \underline{\bk}_{\Flag_{\mathscr{G},w}}[\ell(w)], 
\]
o\`u $i_w : \Flag_{\mathscr{G},w} \to \Flag_\mathscr{G}$ est l'inclusion (voir l'Exemple~\ref{ex:collections-qe}\eqref{it:qe-faisceaux}). La collection duale est donn\'ee par $\Delta_{\mathscr{G},w} := i_{w!} \underline{\bk}_{\Flag_{\mathscr{G},w}}[\ell(w)]$, et la t-structure associ\'ee est la t-structure perverse.

Consid\'erons la fonction $\dag : \WKM \to \Z/2\Z$ associant \`a $w \in \WKM$ la classe de $\ell(w)$. Alors les r\'esultats de~\cite[\S 4.1]{jmw} (voir \'egalement le~\S\ref{ss:comparaison-JMW} pour une comparaison avec nos conventions) montrent que pour tout $w \in \WKM$ il existe un objet \`a $\dag$-parit\'e $\mathscr{E}_{\mathscr{G},w}$ dans $\Db_{\mathrm{Br}}(\Flag_{\mathscr{G}},\bk)$ comme dans le Th\'eor\`eme~\ref{thm:JMW}. Nous noterons $\Par_{\mathrm{Br}}(\Flag_\mathscr{G},\bk)$ la sous-cat\'egorie pleine de $\Db_{\mathrm{Br}}(\Flag_\mathscr{G},\bk)$ dont les objets sont les objets \`a $\dag$-parit\'e. Suivant~\cite{modrap2} on pose alors
\[
\Dmix_{\mathrm{Br}}(\Flag_\mathscr{G},\bk) := \Kb \Par_{\mathrm{Br}}(\Flag_\mathscr{G},\bk).
\]
Le d\'ecalage cohomologique dans $\Db_{\mathrm{Br}}(\Flag_\mathscr{G},\bk)$ induit un automorphisme de la cat\'egorie $\Par_{\mathrm{Br}}(\Flag_\mathscr{G},\bk)$, et donc de $\Dmix_{\mathrm{Br}}(\Flag_\mathscr{G},\bk)$; cet automorphisme sera not\'e $\{1\}$. On d\'efinit alors le ``d\'ecalage de Tate'' par $\la 1 \ra := \{-1\}[1]$.

La th\'eorie d\'evelopp\'ee dans~\cite{modrap2} fournit des objets $\dmix_{\mathscr{G},w}$ et $\nmix_{\mathscr{G},w}$ dans la cat\'egorie $\Dmix_{\mathrm{Br}}(\Flag_\mathscr{G},\bk)$. Pour $w \in \WKM$, nous noterons $\mathscr{E}_{\mathscr{G},w}^{\mix}$ le complexe concentr\'e en degr\'e $0$, et form\'e par l'objet $\mathscr{E}_{\mathscr{G},w}$ en ce degr\'e.

\begin{prop}
\label{prop:Dmix-Bruhat}
La cat\'egorie $\Dmix_{\mathrm{Br}}(\Flag_{\mathscr{G}},\bk)$, munie des structures d\'efinies ci-dessus, v\'erifie les conditions du~{\rm \S\ref{ss:Koszul-contexte}}. De plus, si $\dag : \WKM \to \Z/2\Z$ est la fonction associant \`a $w \in \WKM$ la classe de $\dim(\Flag_{\mathscr{G},w})$, alors pour tout $w \in \WKM$ l'objet $\mathscr{E}_{\mathscr{G},w}^{\mix}$ est un objet \`a $\dag$-parit\'e ind\'ecomposable satisfaisant les conditions du Th\'eor\`eme~{\rm \ref{thm:JMW-gradue}} (dans $\Dmix_{\mathrm{Br}}(\Flag_\mathscr{G},\bk)$ munie de l'automorphisme $\la 1 \ra$ et de la suite exceptionnelle gradu\'ee ci-dessus).
\end{prop}

\begin{proof}
Notons $\Dmix_{\mathrm{Br}}(\Flag_{\mathscr{G},w},\bk)$ la cat\'egorie d\'eriv\'ee mixte construite com\-me ci-dessus pour la vari\'et\'e $\Flag_{\mathscr{G},w}$ au lieu de $\Flag_{\mathscr{G}}$. Dans~\cite{modrap2} on construit un formalisme de ``recollement'' (au sens de~\cite[\S 1.4]{bbd}), qui fournit notamment des foncteurs $i_{w*}, i_{w!} : \Dmix_{\mathrm{Br}}(\Flag_{\mathscr{G},w},\bk) \to \Dmix_{\mathrm{Br}}(\Flag_\mathscr{G},\bk)$ et $i_w^*, i_w^! : \Dmix_{\mathrm{Br}}(\Flag_\mathscr{G},\bk) \to \Dmix_{\mathrm{Br}}(\Flag_{\mathscr{G},w},\bk)$ qui v\'erifient les adjonctions habituelles. Par d\'efinition on a alors
\[
\dmix_{\mathscr{G},w} := i_{w!} \underline{\bk}_{\Flag_{\mathscr{G},w}}\{\ell(w)\}, \qquad \nmix_{\mathscr{G},w} := i_{w*} \underline{\bk}_{\Flag_{\mathscr{G},w}}\{\ell(w)\}.
\]
Le fait que la collection $(\nmix_{\mathscr{G},w} : w \in \WKM)$ est une collection exceptionnelle gradu\'ee d\'ecoule du formalisme de recollement. Le fait qu'elle v\'erifie la condition~\eqref{it:condition-coeur} du~\S\ref{ss:Koszul-contexte} est le contenu de~\cite[Theorem~4.7]{modrap2}.

Pour $y,w \in \WKM$ et $m,n \in \Z$, par adjonction on a
\begin{multline*}
\Hom_{\Dmix_{\mathrm{Br}}(\Flag_\mathscr{G},\bk)}(\dmix_{\mathscr{G},y}, \mathscr{E}_{\mathscr{G},w}^\mix \la m \ra [n]) \cong \\
\Hom_{\Dmix_{\mathrm{Br}}(\Flag_{\mathscr{G},y},\bk)}(\underline{\bk}_{\Flag_{\mathscr{G},y}}\{\ell(y)\}, i_y^! \mathscr{E}_{\mathscr{G},w}^\mix \{-m\}[m+n]).
\end{multline*}
D'apr\`es~\cite[Remark~2.7]{modrap2} le complexe $i_y^! \mathscr{E}_{\mathscr{G},w}^\mix$ est concentr\'e en degr\'e $0$, avec en ce degr\'e l'objet $i_y^! \mathscr{E}_{\mathscr{G},w}$ (o\`u ici $i_y^!$ est le foncteur d'image inverse exceptionnelle usuel). L'espace ci-dessus s'annule donc si $n+m \neq 0$, et est isomorphe \`a
\[
\Hom_{\Db_{\mathrm{Br}}(\Flag_{\mathscr{G}},\bk)}(\Delta_{\mathscr{G},y}, \mathscr{E}_{\mathscr{G},w} \{-m\})
\]
si $m+n=0$. Puisque $\mathscr{E}_{\mathscr{G},w}$ a la m\^eme parit\'e que $\ell(w)$, par d\'efinition cet espace s'annule si $m$ n'a pas la m\^eme parit\'e que $\ell(y)+\ell(w)$. 

On v\'erifie par des arguments similaires que
\[
\Hom_{\Dmix_{\mathrm{Br}}(\Flag_\mathscr{G},\bk)}(\mathscr{E}_{\mathscr{G},w}^\mix, \nmix_{\mathscr{G},y} \la m \ra [n]) 
\]
s'annule sauf si $m+n=0$ et $m$ a la m\^eme parit\'e que $\ell(y)+\ell(w)$, ce qui entra\^ine que $\mathscr{E}_{\mathscr{G},w}^\mix$ est \`a parit\'e (plus pr\'ecis\'ement, pair si $\ell(w)$ est pair et impair sinon). Cet objet est clairement ind\'ecomposable ; il co\"incide donc avec l'objet caract\'eris\'e dans le Th\'eor\`eme~\ref{thm:JMW-gradue}.
\end{proof}

Suivant la terminologie de~\cite{modrap2}, nous appellerons ``t-structure perverse'' la t-structure associ\'ee \`a la collection exceptionnelle $(\nmix_{\mathscr{G},w} : w \in \WKM)$. Son coeur sera not\'e $\Perv^\mix_{\mathrm{Br}}(\Flag_{\mathscr{G}},\bk)$. La Proposition~\ref{prop:Dmix-Bruhat} assure que cette cat\'egorie poss\`ede une structure naturelle de cat\'egorie de plus haut poids gradu\'ee. En particulier, on dispose d'objets basculants ind\'ecomposables ``normalis\'es'' $\mathscr{T}^\mix_{\mathscr{G},w}$ pour $w \in \WKM$.

\begin{ex}
Si $s$ est une r\'eflexion simple, l'objet $\mathscr{T}^\mix_{\mathscr{G},s}$ peut se d\'ecrire explicitement comme le complexe
\[
\cdots \to 0 \to \mathscr{E}_e \{-1\} \to \mathscr{E}_s \to \mathscr{E}_e \{1\} \to 0 \to \cdots
\]
o\`u $\mathscr{E}_s$ est en degr\'e $0$, et les diff\'erentielles sont les applications naturelles.
\end{ex}

\begin{rmq}
\begin{enumerate}
\item
Notre notation $\{1\}$ est bien s\^ur compatible avec celles utilis\'ees dans la Partie~\ref{sec:Koszul-formelle} (si l'automorphisme de d\'ecalage interne est $\la 1 \ra$).
\item
Une partie de la th\'eorie d\'evelopp\'ee dans~\cite{modrap2} s'applique \`a des vari\'et\'es tr\`es g\'en\'erales (la seule hypoth\`ese importante \'etant que les strates de la stratification consid\'er\'ee sont des espaces affines). Cependant, nous ne savons pas si la condition~\eqref{it:condition-coeur} du~\S\ref{ss:Koszul-contexte} est vraie en g\'en\'erale. Dans~\cite{modrap2}, cette condition n'est d\'emontr\'ee que dans le cas des vari\'et\'es de drapeaux (\'eventuellement paraboliques) des groupes de Kac--Moody. (La preuve s'applique \'egalement aux ``variantes'' consid\'er\'ees au~\S\ref{ss:drap-affine} ci-dessous.)
\item
Nous consid\'ererons la cat\'egorie $\Perv^\mix_{\mathrm{Br}}(\Flag_{\mathscr{G}},\bk)$ comme une ``version gradu\'ee'' de la cat\'egorie des faisceaux pervers ordinaires $\Perv_{\mathrm{Br}}(\Flag_{\mathscr{G}},\bk)$. Notons cependant qu'il n'existe aucun foncteur naturel ``d'oubli'' reliant ces cat\'egories.
\end{enumerate}
\end{rmq}

Ci-dessous on consid\'erera \'egalement la vari\'et\'e de drapeaux ``oppos\'ee'' $\Flag_{\mathscr{G}}^\op = \mathscr{B} \backslash \mathscr{G}$. Toutes les constructions ci-dessus s'appliquent de la m\^eme fa{\c c}on \`a cette vari\'et\'e ; on utilisera des notations similaires, en rempla{\c c}ant les indices ``$\mathscr{G},w$'' par ``$\mathscr{G},\op,w$''. Bien s\^ur, l'application $g \cdot \mathscr{B} \mapsto \mathscr{B} \cdot g^{-1}$ induit des \'equivalences de cat\'egories
\[
\Db_{\mathrm{Br}}(\Flag_{\mathscr{G}},\bk) \simto \Db_{\mathrm{Br}}(\Flag^\op_{\mathscr{G}},\bk) \quad \text{et} \quad
\Dmix_{\mathrm{Br}}(\Flag_{\mathscr{G}},\bk) \simto \Dmix_{\mathrm{Br}}(\Flag^\op_{\mathscr{G}},\bk)
\]
envoyant $\mathscr{E}_{\mathscr{G},w}$ sur $\mathscr{E}_{\mathscr{G},\op,w^{-1}}$, $\dmix_{\mathscr{G},w}$ sur $\dmix_{\mathscr{G},\op,w^{-1}}$, etc. ; mais comme on le verra il est parfois utile de distinguer ces deux versions.

\subsection{Variante : les vari\'et\'es de drapeaux affines}
\label{ss:drap-affine}

Soit $G$ un groupe alg\'ebrique complexe r\'eductif connexe, $B \subset G$ un sous-groupe de Borel, et $T \subset B$ un tore maximal. Posons $\mathscr{K}:=\C( \hspace{-1pt} (z) \hspace{-1pt} )$, $\mathscr{O}:=\C[ \hspace{-1pt} [z] \hspace{-1pt} ]$, et notons $G_{\mathscr{K}}$, resp.~$G_{\mathscr{O}}$ le ind-sch\'ema en groupe, resp.~sch\'ema en groupe, sur $\C$ qui repr\'esente le foncteur
\[
R \mapsto G \bigl( R( \hspace{-1pt} (z) \hspace{-1pt} ) \bigr), \quad \text{resp.} \quad R \mapsto G \bigl( R[ \hspace{-1pt} [z] \hspace{-1pt} ] \bigr).
\]
On a un morphisme d'\'evaluation $G_{\mathscr{O}} \to G$ (correspondant \`a $z \mapsto 0$), et on note $I$ l'image inverse de $B$ par ce morphisme. La \emph{vari\'et\'e de drapeaux affine} de $G$ est la ind-vari\'et\'e complexe
\[
\Fl_G := G_{\mathscr{K}} / I.
\]

Si $G$ est quasi-simple et simplement connexe,
cette ind-vari\'et\'e est la vari\'et\'e de drapeaux d'un groupe de Kac--Moody (affine, non tordu), et on peut donc lui appliquer les constructions du~\S\ref{ss:Dmix}. Dans le cas g\'en\'eral, $\Fl_G$ n'est pas la vari\'et\'e de drapeaux d'un groupe de Kac--Moody ; cependant elle poss\`ede une structure tr\`es similaire. En particulier on a une ``d\'ecomposition de Bruhat'' d\'ecrivant les $I$-orbites sur $\Fl_G$, qui sont isomorphes \`a des espaces affines et param\'etr\'ees par le ``groupe de Weyl affine \'etendu'' $\Wext:=W \ltimes X_*(T)$ (voir le~\S\ref{ss:Koszul-affine} ci-dessous pour plus de d\'etails sur ce groupe). On peut donc consid\'erer les complexes \`a parit\'e sur cette vari\'et\'e (pour la fonction $\dag$ donn\'ee par la parit\'e de la dimension de l'orbite), la cat\'egorie d\'eriv\'ee mixte associ\'ee $\Dmix_{\mathrm{Br}}(\Fl_G,\bk)$, et sa t-structure perverse. Les objets \`a parit\'e ind\'ecomposables dans $\Db_{\mathrm{Br}}(\Fl_G,\bk)$ sont param\'etr\'es par $\Wext \times \Z$, et on notera $\mathscr{E}_{\Fl_G,w}$ l'objet associ\'e \`a $(w,0)$ (pour $w \in \Wext$) et $\mathscr{E}_{\Fl_G,w}^\mix$ son image dans $\Dmix_{\mathrm{Br}}(\Fl_G,\bk)$. 
Les objets standards et costandards associ\'es \`a $w \in \Wext$ seront not\'es $\dmix_{\Fl_G,w}$ et $\nmix_{\Fl_G,w}$ respectivement; alors l'analogue de la
Proposition~\ref{prop:Dmix-Bruhat} est vrai, et se prouve de fa{\c c}on similaire.
Les faisceaux pervers basculants ind\'ecomposables dans $\Dmix_{\mathrm{Br}}(\Fl_G,\bk)$ sont \'egalement param\'etr\'es par $\Wext \times \Z$, et on notera $\mathscr{T}^\mix_{\Fl_G,w}$ l'objet associ\'e \`a $(w,0)$.

De m\^eme on consid\'erera la vari\'et\'e de drapeaux affine ``oppos\'ee''
\[
\Fl_G^\op = I \backslash G_{\mathscr{K}},
\]
et la cat\'egorie d\'eriv\'ee mixte associ\'ee $\Dmix_{\mathrm{Br}}(\Fl_G^\op,\bk)$. Les objets \`a parit\'e ind\'ecomposables normalis\'es, leurs images dans $\Dmix_{\mathrm{Br}}(\Fl_G^\op,\bk)$, et les faisceaux pervers basculants normalis\'es seront not\'es respectivement $\mathscr{E}_{\Fl_G^\op,w}$, $\mathscr{E}^\mix_{\Fl_G^\op,w}$ et $\mathscr{T}^\mix_{\Fl_G^\op,w}$. Les objets standards et costandards seront not\'es $\dmix_{\Fl^\op_G,w}$ et $\nmix_{\Fl^\op_G,w}$ respectivement.

Ces vari\'et\'es ont \'egalement des analogues paraboliques (souvent appel\'ees plut\^ot ``parahoriques'' dans ce contexte). Le cas qui nous int\'eressera sera celui de la ``Grassmannienne affine'' et sa version oppos\'ee, d\'efinies respectivement par
\[
\Gr_G:=G_\mathscr{K} / G_\mathscr{O}, \qquad \Gr_G^\op := G_\mathscr{O} \backslash G_\mathscr{K}.
\]
Dans ce cas on peut consid\'erer les orbites de l'action naturelle de $I$ sur $\Gr_G$ ou $\Gr_G^\op$; la combinatoire associ\'ee sera rappel\'ee au~\S\ref{ss:koszul-parabolique-aff}. Mais on peut \'egalement consid\'erer l'action du sous-groupe d'Iwahori ``oppos\'e'' $\Iw$, d\'efini comme l'image inverse par le morphisme d'\'evaluation $G_{\mathscr{O}} \to G$ du sous-groupe de Borel $B^+ \subset G$ oppos\'e \`a $B$ (par rapport \`a $T$). Dans ce cadre-l\`a on param\`etrera les orbites par $X_*(T)$, en envoyant $\lambda$ sur l'orbite de la classe de l'image de $z \in \mathscr{K}^\times$ par le morphisme $\mathscr{K}^\times \to T_{\mathscr{K}}$ induit par $\lambda$. On obtient ainsi une stratification alg\'ebrique
\[
\Gr_G = \bigsqcup_{\lambda \in X_*(T)} \Gr_{G,\lambda},
\]
o\`u chaque $\Gr_{G,\lambda}$ est isomorphe \`a un espace affine. On notera $\Db_{(\Iw)}(\Gr_G,\bk)$ la cat\'egorie d\'eriv\'ee des $\bk$-faisceaux constructibles par rapport \`a cette stratification, et $\Dmix_{(\Iw)}(\Gr_G,\bk)$ la cat\'egorie d\'eriv\'ee mixte associ\'ee. Cette cat\'egorie poss\`ede une t-structure perverse, dont on notera $\Perv^{\mix}_{(\Iw)}(\Gr_G,\bk)$ le coeur. Les objets standards et costandards associ\'es \`a $\lambda \in X_*(T)$ seront not\'es $\dmix_{\Gr_G,\lambda}$ et $\nmix_{\Gr_G,\lambda}$ respectivement, et les complexes \`a parit\'e ind\'ecomposables normalis\'es et leurs images dans $\Dmix_{(\Iw)}(\Gr_G,\bk)$ seront not\'es $\mathscr{E}_{\Gr_G,\lambda}$ et $\mathscr{E}^\mix_{\Gr_G,\lambda}$ respectivement. Encore une fois, l'analogue de la
Proposition~\ref{prop:Dmix-Bruhat} est vrai, et se prouve de fa{\c c}on similaire. Les faisceaux pervers ind\'ecomposables normalis\'es seront not\'es $\mathscr{T}^\mix_{\Gr_G,\lambda}$.

Nous utiliserons des notations similaires pour $\Gr_G^\op$ \`a la place de $\Gr_G$.

\subsection{Dualit\'e de Koszul pour les faisceaux constructibles}
\label{ss:Koszul-constructible}

On revient aux notations du~\S\ref{ss:Dmix} : on suppose donc que $\Flag_{\mathscr{G}}$ est la vari\'et\'e de drapeaux d'un groupe de Kac--Moody $\mathscr{G}$. On consid\`ere \'egalement la vari\'et\'e de drapeaux oppos\'ee $\Flag_{\mathscr{G}^\vee}^{\mathrm{op}}$ du groupe de Kac--Moody $\mathscr{G}^\vee$ dual de $\mathscr{G}$ au sens de Langlands (c'est-\`a-dire dont la donn\'ee radicielle de Kac--Moody est duale de celle de $\mathscr{G}$). On peut alors consid\'erer les cat\'egories $\Dmix_{\mathrm{Br}}(\Flag,\bk)$ et $\Dmix_{\mathrm{Br}}(\Flag^\op_{\mathscr{G}^{\vee}},\bk)$. 
Notons que les groupes de Weyl associ\'es \`a $\scG$ et $\scG^\vee$ sont canoniquement isomorphes ; ils seront tous deux not\'es $\WKM$.

Le th\'eor\`eme suivant est le r\'esultat principal de~\cite{amrw2}.

\begin{thm}
\label{thm:mkd}
Supposons que $\mathrm{char}(\bk)$ est 
diff\'erent de $2$.
Il existe une dualit\'e de Koszul formelle
\[
\kappa : \Dmix_{\mathrm{Br}}(\Flag_{\mathscr{G}},\bk) \simto \Dmix_{\mathrm{Br}}(\Flag^\op_{\mathscr{G}^\vee},\bk),
\]
dont la bijection associ\'ee sur les ensembles de param\`etres est $\id_\WKM$.
\end{thm}

Notons que la Proposition~\ref{prop:Koszul-parite-basculant} nous assure que l'\'equivalence du Th\'eor\`eme~\ref{thm:mkd} v\'erifie
\begin{equation}
\label{eqn:kappa-E-T}
\kappa(\mathscr{E}_{\mathscr{G},w}^\mix) \cong \mathscr{T}^{\mix}_{\mathscr{G}^\vee,\op,w} \quad \text{et} \quad \kappa(\mathscr{T}_{\mathscr{G},w}^\mix) \cong \mathscr{E}_{\mathscr{G}^\vee,\op,w}^{\mix}
\end{equation}
pour tout $w \in \WKM$.

\begin{rmq}
\begin{enumerate}
\item
Dans~\cite{amrw2} nous travaillons en fait sous une hypoth\`ese plus faible (mais plus technique \`a \'enoncer) sur $\mathrm{char}(\bk)$. Puisque cette hypoth\`ese est v\'erifi\'ee d\`es que $\mathrm{char}(\bk)$ est diff\'erent de $2$, ici nous avons simplifi\'e l'\'enonc\'e en nous restreignant \`a ce cas.
\item
Il devrait \^etre possible de g\'en\'eraliser les constructions du Th\'eor\`eme~\ref{thm:mkd} au contexte du~\S\ref{ss:drap-affine} (pour $G$ r\'eductif quelconque). Cependant les m\'ethodes utilis\'ees dans~\cite{amrw2} reposent sur les constructions de~\cite{ew}, qui n'ont \'et\'e d\'evelopp\'ees que pour les groupes de Coxeter (et non pour leurs ``cousins'' que sont les groupes de Weyl affines \'etendus) ; cette g\'en\'eralisation demanderait donc comme pr\'ealable de g\'en\'eraliser ces outils, ce que nous ne ferons pas ici. Cependant la ``trace combinatoire'' de cette \'equivalence se g\'en\'eralise facilement au cadre du~\S\ref{ss:drap-affine}; voir le~\S\ref{ss:Koszul-affine} ci-dessous. Ceci est suffisant pour toutes les applications consid\'er\'ees jusque-l\`a.
\end{enumerate}
\end{rmq}

\subsection{Faisceaux coh\'erents sur la r\'esolution de Springer et t-structure exotique}
\label{ss:ets}

\newcommand{\GCoh}{\mathfrak{G}}
\newcommand{\BCoh}{\mathfrak{B}}
\newcommand{\TCoh}{\mathfrak{T}}

On suppose maintenant que $\bk$ est alg\'ebriquement clos, et on note $\GCoh$ un groupe alg\'ebrique connexe r\'eductif sur $\bk$ dont le sous-groupe d\'eriv\'e est simplement connexe. On choisit un sous-groupe de Borel $\BCoh \subset \GCoh$ et un tore maximal $\TCoh \subset \BCoh$, et on pose $\bX := X^*(\TCoh)$.

Soit $\tcN_{\GCoh}$ l'espace cotangent \`a la vari\'et\'e de drapeaux $\GCoh/\BCoh$. Cette vari\'et\'e poss\`ede une action naturelle de $\GCoh \times \Gm$, o\`u $\GCoh$ agit via l'action induite par l'action naturelle sur $\GCoh/\BCoh$, et $z \in \Gm$ agit par multiplication par $z^{-2}$ le long des fibres de la projection $\tcN_{\GCoh} \to \GCoh/\BCoh$. On consid\`ere alors la cat\'egorie d\'eriv\'ee born\'ee
\[
\Db \Coh^{\GCoh \times \Gm}(\tcN_{\GCoh})
\]
de la cat\'egorie des faisceaux coh\'erents $\GCoh \times \Gm$-\'equivariants sur $\tcN_{\GCoh}$. On note\footnote{La convention adopt\'ee ici est oppos\'ee \`a celle utilis\'ee dans~\cite{mr,prinblock,ahr}.} $\la 1 \ra$ l'automorphisme donn\'e par le produit tensoriel avec le $\Gm$-module de dimension $1$ et de poids $-1$.

Dans~\cite{bez:ctm} (voir aussi~\cite{mr:ets}), Bezrukavnikov construit une collection exceptionnelle gradu\'ee $(\nabla^\Coh_\lambda : \lambda \in \bX)$ dans $\Db \Coh^{\GCoh \times \Gm}(\tcN_\GCoh)$, pour un certain ordre $\preceq$ sur $\bX$. (Cet ordre n'est pas uniquement d\'etermin\'e ; dans la suite on le fixe, en supposant en particulier que $(\bX,\preceq)$ est isomorphe \`a $(\Z,\leq)$. Ce choix n'affecte aucune des constructions consid\'er\'ees.) On notera $(\Delta^\Coh_\lambda : \lambda \in \bX)$ la collection duale, et $\dag : \bX \to \Z/2\Z$ la fonction envoyant $\lambda$ sur la classe de
\begin{equation}
\label{eqn:long-wlambda}
\sum_{\alpha \in R^+} |\langle \lambda,\alpha^\vee \rangle| - \ell(v_\lambda),
\end{equation}
o\`u $R$ est le syst\`eme de racines de $(\GCoh,\TCoh)$, $R^+ \subset R$ est le syst\`eme de racines positives tel que $\BCoh$ est le sous-groupe de Borel n\'egatif, $\alpha^\vee$ est la coracine associ\'ee \`a $\alpha$, et enfin $v_\lambda$ est l'\'el\'ement du groupe de Weyl $W$ de $(\bG,\bT)$ de longueur minimale tel que $v_\lambda(\lambda)$ est dominant (pour la structure de groupe de Coxeter sur $W$ d\'etermin\'ee par $R^+$). (Voir la Remarque~\ref{rmq:longueur-wlambda} ci-dessous pour une interpr\'etation plus naturelle de cette fonction.)

\begin{prop}
\label{prop:ets}
La cat\'egorie $\Db \Coh^{\GCoh \times \Gm}(\tcN_\GCoh)$, munie des structures ci-dessus, v\'erifie les conditions du~{\rm \S\ref{ss:Koszul-contexte}}. Soit $\dag : \bX \to \Z/2\Z$ la fonction consid\'er\'ee ci-dessus. Alors, de plus, pour tout $\lambda \in \bX$ il existe un objet \`a $\dag$-parit\'e ind\'ecomposable $\mathcal{E}_\lambda$ dans $\Db \Coh^{\GCoh \times \Gm}(\tcN_\GCoh)$ v\'erifiant les conditions du Th\'eor\`eme~{\rm \ref{thm:JMW-gradue}}.
\end{prop}

\begin{proof}
Comme expliqu\'e ci-dessus, la collection $(\nabla^\Coh_\lambda : \lambda \in \bX)$ est exceptionnelle gradu\'ee par construction. Le fait qu'elle v\'erifie la condition~\eqref{it:condition-coeur} du~\S\ref{ss:Koszul-contexte} est le contenu de~\cite[Corollary~3.11]{mr:ets} (voir \'egalement~\cite[Proposition~8.5]{arider}). Enfin, l'existence des objets \`a parit\'e est d\'emontr\'ee dans~\cite[\S 3.4]{ahr}.
\end{proof}

La t-structure associ\'ee \`a la collection exceptionnelle gradu\'ee $(\nabla^\Coh_\lambda : \lambda \in \bX)$ est appel\'ee la \emph{t-structure exotique}. La Proposition~\ref{prop:ets} implique que le coeur de cette t-structure est une cat\'egorie gradu\'ee de plus haut poids ; on peut donc consid\'erer ses objets basculants ind\'ecomposables. Ceux-ci sont param\'etr\'es par $\bX \times \Z$, et on notera $\mathcal{T}_\lambda$ l'objet associ\'e \`a $(\lambda,0)$ (pour $\lambda \in \bX$).

\subsection{Dualit\'e de Koszul constructibles-coh\'erents}
\label{ss:koszul-coh-const}

On continue avec les notations du~\S\ref{ss:ets}, et on suppose de plus que la caract\'eristique de $\bk$ est bonne pour $\GCoh$, et que l'alg\`ebre de Lie de $\GCoh$ poss\`ede une forme bilin\'eaire $\GCoh$-invariante non-d\'eg\'en\'er\'ee. 
On notera $G^\vee$ le groupe r\'eductif connexe complexe qui est dual de $\GCoh$ au sens de Langlands, avec le tore maximal et le sous-groupe de Borel $T^\vee \subset B^\vee \subset G^\vee$ qui sont duaux de $\TCoh \subset \BCoh \subset \GCoh$, et $\Gr_{G^\vee}$ la Grassmannienne affine de $G^\vee$. On a alors, par d\'efinition $\bX=X^*(\TCoh) = X_*(T^\vee)$. Comme au~\S\ref{ss:drap-affine} on consid\`ere l'action sur $\Gr_{G^\vee}$ du sous-groupe d'Iwahori $\Iw$ d\'efini par le sous-groupe de Borel oppos\'e \`a $B^\vee$ (par rapport \`a $T^\vee$), et la cat\'egorie d\'eriv\'ee mixte associ\'ee $\Dmix_{(\Iw)}(\Gr_{G^\vee},\bk)$.

L'\'enonc\'e suivant est le r\'esultat principal de~\cite{arider} (voir \'egalement~\cite{mr} pour une preuve alternative, sous des hypoth\`ese l\'eg\`erement plus fortes).

\begin{thm}
\label{thm:coh-const}
Il existe une dualit\'e de Koszul formelle
\[
\Psi : \Db \Coh^{\GCoh \times \Gm}(\tcN_{\GCoh}) \simto \Dmix_{(\Iw)}(\Gr_{G^\vee},\bk)
\]
dont la bijection associ\'ee sur les ensembles de param\`etres est $\id_\bX$.
\end{thm}

Notons que la Proposition~\ref{prop:Koszul-parite-basculant} nous assure que l'\'equivalence du Th\'eor\`eme~\ref{thm:coh-const} v\'erifie
\[
\Psi(\mathcal{E}_{\lambda}) \cong \mathscr{T}^{\mix}_{\Gr_{G^\vee},\lambda} \quad \text{et} \quad \Psi(\mathcal{T}_\lambda) \cong \mathscr{E}_{\Gr_{G^\vee},\lambda}^{\mix}
\]
pour tout $\lambda \in \bX$.

\subsection{Des faisceaux coh\'erents aux repr\'esentations}
\label{ss:coh-rep}

Comme aux~\S\S\ref{ss:ets}--\ref{ss:koszul-coh-const} on suppose que $\bk$ est un corps alg\'ebriquement clos, mais on suppose maintenant que sa caract\'eristique $p$ est positive.
Pour tout $\bk$-groupe alg\'ebrique affine $\mathbf{H}$, on notera $\dot{\mathbf{H}}$ le d\'ecal\'e de Frobenius de $\mathbf{H}$ (voir par exemple~\cite[\S I.9.2]{jantzen}).

Soit $\bG$ un groupe alg\'ebrique connexe r\'eductif sur $\bk$, \`a sous-groupe d\'eriv\'e simplement connexe, et soient $\bT \subset \bB \subset \bG$ un tore maximal et un sous-groupe de Borel.
Dans ce paragraphe on supposera que $p>h$, o\`u $h$ est le nombre de Coxeter de $\bG$. 
Alors le groupe $\GCoh:=\dot{\bG}$ v\'erifie les hypoth\`eses du~\S\ref{ss:koszul-coh-const}. On appliquera les constructions de ce paragraphe pour le tore maximal $\TCoh = \dot{\bT}$ et le sous-groupe de Borel $\BCoh=\dot{\bB}$. On identifiera le groupe des caract\`eres de $\bT$ \`a $\bX=X^*(\dot{\bT})$ de sorte que la composition avec le morphisme de Frobenius $\bT \to \dot{\bT}$ correspond \`a la multiplication par $p$ sur $\bX$.

\begin{rmq}
Il est plus courant en th\'eorie des repr\'esentations d'imposer l'hypoth\`ese $p \geq h$ plut\^ot que $p>h$. On peut v\'erifier (en consultant la table des nombres de Coxeter par exemple) que cette derni\`ere condition est \'equivalente \`a demander que :
\begin{enumerate}
\item
$p \geq h$ ;
\item
$p$ est tr\`es bon pour $\bG$.
\end{enumerate}
\end{rmq}

On notera
$\Rep(\bG)$ la cat\'egorie des repr\'esentations alg\'ebriques de dimension finie de $\bG$. Si $\bX^+ \subset \bX$ est le sous-ensemble des poids dominants (pour le choix de racines positives tel que $\bB$ est le sous-groupe de Borel \emph{n\'egatif}), alors pour $\lambda \in \bX^+$ on notera $\mathsf{N}(\lambda)$ le module induit de plus haut poids $\lambda$ (not\'e $H^0(\lambda)$ dans~\cite{jantzen}) et $\mathsf{M}(\lambda)$ le module de Weyl de plus haut poids $\lambda$. Si $w_0$ est l'\'el\'ement de plus grande longueur dans le groupe de Weyl $W$ de $(\bG,\bT)$, alors on a $\mathsf{M}(\lambda) = \mathsf{N}(-w_0\lambda)^*$. Il existe (\`a scalaire pr\`es) un unique morphisme de $\bG$-modules non nul $\mathsf{M}(\lambda) \to \mathsf{N}(\lambda)$; son image est l'unique sous-module simple de $\mathsf{N}(\lambda)$, qu'on notera $\mathsf{L}(\lambda)$. 

D'apr\`es~\cite[\S II.4.13, Remark~2]{jantzen}, les objets $(\mathsf{N}(\lambda) : \lambda \in \bX^+)$ forment une collection exceptionnelle dans la cat\'egorie triangul\'ee $\Db \Rep(\bG)$ pour l'ordre naturel sur $\bX^+$ (voir par exemple~\cite[\S II.1.5]{jantzen}). La famille duale est la famille $(\mathsf{M}(\lambda) : \lambda \in \bX^+)$ (voir~\cite[Proposition~II.4.13]{jantzen}), et la t-structure associ\'ee est la t-structure tautologique. Puisque les objets $\mathsf{N}(\lambda)$ et $\mathsf{M}(\lambda)$ sont concentr\'es en degr\'e $0$, la cat\'egorie $\Rep(\bG)$ poss\`ede une structure naturelle de cat\'egorie de plus haut poids, et on peut donc consid\'erer ses objets basculants ; les objets basculants ind\'ecomposables sont classifi\'es (\`a isomorphisme pr\`es) par $\bX^+$, et on notera $\mathsf{T}(\lambda)$ l'objet associ\'e \`a $\lambda$ (voir \'egalement~\cite[Appendice~E]{jantzen}).

Consid\'erons maintenant le groupe de Weyl affine \'etendu
\[
\Wext := W \ltimes \bX
\]
et son sous-groupe $\Waff := W \ltimes \Z R$ (o\`u $R$ est le syst\`eme de racines de $(\bG,\bT)$, et $\Z R \subset \bX$ est le r\'eseau radiciel). On notera $t_\lambda$ l'image dans $\Wext$ d'un \'el\'ement $\lambda \in \bX$. Alors $\Waff$ poss\`ede une structure naturelle de groupe de Coxeter (voir~\cite[Chap.~II.6]{jantzen}). Sa fonction de longueur v\'erifie
\[
\ell(w \cdot t_\lambda)=\sum_{\substack{\alpha \in R^+ \\  w(\alpha) \in R^+}} |\langle \lambda,
\alpha^\vee \rangle | + \sum_{\substack{\alpha \in
     R^+ \\ w(\alpha) \in - R^+}}
  |1 + \langle \lambda, \alpha^\vee \rangle |
\]
pour $w \in W$ et $\lambda \in \Z R$. Cette formule a un sens plus g\'en\'eralement pour $\lambda \in \bX$, ce qui permet d'\'etendre $\ell$ en une fonction sur $\Wext$ (qu'on notera de la m\^eme fa{\c c}on). En utilisant cette fonction on peut \'egalement \'etendre de fa{\c c}on naturelle l'ordre de Bruhat sur $\Waff$ en un ordre partiel sur $\Wext$.

On consid\`ere l'action de $\Wext$ sur $\bX$ d\'efinie par
\[
(w t_\lambda) \cdot_p \mu = w(\mu + p\lambda + \rho) - \rho
\]
pour $\lambda,\mu \in \bX$ et $w \in W$, o\`u $\rho$ est la demi-somme des racines positives de $\bG$. Alors il est bien connu que, sous nos hypoth\`eses, pour $w \in \Wext$ le poids $w \cdot_p 0$ est dominant si et seulement si $w$ appartient au sous-ensemble $\Wextmin \subset \Wext$ form\'e des \'el\'ements $y$ de longueur minimale dans $W \cdot y$. De plus, on a des bijections naturelles
\[
\bX \simto W \backslash \Wext \simto \Wextmin ; 
\]
leur compos\'ee sera not\'ee $\lambda \mapsto w_\lambda$. On utilise cette bijection pour transf\'erer l'ordre $\preceq$ sur $\bX$ en un ordre sur $\Wextmin$, qu'on notera par le m\^eme symbole. (Par construction, cet ordre \'etend la restriction \`a $\Wextmin$ de l'ordre de Bruhat sur $\Wext$.)

\begin{rmq}
\label{rmq:longueur-wlambda}
Une fois ces notations introduites, on peut interpr\'eter la quantit\'e~\eqref{eqn:long-wlambda} comme $\ell(w_\lambda)$ ; voir par exemple~\cite[Lemma~2.4]{mr:ets}.
\end{rmq}

Pour $w \in \Wextmin$, on pose maintenant
\[
\mathsf{L}_w := \mathsf{L}(w \cdot_p 0), \quad \mathsf{N}_w := \mathsf{N}(w \cdot_p 0), \quad \mathsf{M}_w := \mathsf{M}(w \cdot_p 0), \quad \mathsf{T}_w := \mathsf{T}(w \cdot_p 0).
\]
Le ``linkage principle'' (voir~\cite[Corollary~II.6.17]{jantzen}) assure que la sous-cat\'egorie de Serre $\Rep_\varnothing(\bG)$ de $\Rep(\bG)$ engendr\'ee par les modules simple $(\mathsf{L}_w : w \in \Wextmin)$ est un facteur direct de $\Rep(\bG)$, qui contient les repr\'esentations $\mathsf{N}_w$, $\mathsf{M}_w$ et $\mathsf{T}_w$ pour $w \in \Wextmin$. De plus, la cat\'egorie $\Rep_\varnothing(\bG)$, munie des objets $(\mathsf{N}_w : w \in \Wextmin)$ et $(\mathsf{M}_w : w \in \Wextmin)$, est encore une cat\'egorie de plus haut poids pour l'ordre $\preceq$ sur $\Wextmin$ (voir notamment~\cite[Proposition~II.6.16]{jantzen} et~\cite[Lemma~10.1]{prinblock}).

L'\'enonc\'e suivant est le r\'esultat principal de~\cite{prinblock}.

\begin{thm}
\label{thm:prinblock}
Il existe une dualit\'e de Koszul formelle d\'egraduante
\[
\Xi : \Db \Coh^{\dot \bG \times \Gm}(\tcN_{\dot \bG}) \to \Db \Rep_{\varnothing}(\bG)
\]
dont la bijection associ\'ee sur les espaces de param\`etres est $\lambda \mapsto w_\lambda$.
\end{thm}

La Proposition~\ref{prop:Koszul-parite-basculant-degrad} nous assure que l'\'equivalence $\Xi$ du Th\'eor\`eme~\ref{thm:prinblock} v\'erifie
\[
\Xi(\mathcal{E}_\lambda) \cong \mathsf{T}_{w_\lambda}
\]
pour tout $\lambda \in \bX$.

\section{Traces combinatoires}
\label{sec:traces-comb}

Dans cette partie on interpr\`ete les resultats du~\S\ref{sec:exemples} en termes de la ``base $p$-canonique'' de Williamson.

\subsection{Base $p$-canonique}
\label{ss:base-p-can}

Consid\'erons le cadre du~\S\ref{ss:Dmix}: $\mathscr{G}$ est un groupe de Kac--Moody, $\Flag_{\mathscr{G}}$ la vari\'et\'e de drapeaux correspondante, 
et $\WKM$ son groupe de Weyl,
qu'on munit de sa structure naturelle de groupe de Coxeter. On choisit un corps $\bk$ de caract\'eristique $p \neq 2$.

\`A $\WKM$ on peut associer son alg\`ebre de Hecke $\mathcal{H}_\WKM$. (On suivra les conventions de notation de~\cite{soergel-comb-tilting}; en particulier la base standard de $\mathcal{H}_\WKM$ sera not\'ee $(H_w : w \in \WKM)$, et sa base de Kazhdan--Lusztig sera not\'ee $(\underline{H}_w : w \in \WKM)$.) Si on note $[\Dmix_{\mathrm{Br}}(\Flag_{\mathscr{G}},\bk)]$ le groupe de Grothendieck de la cat\'egorie triangul\'ee $\Dmix_{\mathrm{Br}}(\Flag_{\mathscr{G}},\bk)$, il n'est pas difficile de v\'erifier qu'il existe un unique isomorphisme de $\Z[v,v^{-1}]$-modules
\[
\varphi_{\mathscr{G}}^\bk : [\Dmix_{\mathrm{Br}}(\Flag_{\mathscr{G}},\bk)] \simto \mathcal{H}_\WKM
\]
envoyant $[\dmix_{\mathscr{G},w}\{n\}]$ sur $v^n H_w$. Sous cet isomorphisme, pour toute r\'eflexion simple $s$ l'action de $\underline{H}_s$ sur $\mathcal{H}_\WKM$ correspond au foncteur induit par la convolution \`a droite avec le complexe \`a parit\'e $\underline{\bk}_{\Flag_{\mathscr{G},s}}\{1\}$ (un complexe $\mathscr{B}$-\'equivariant).

De fa{\c c}on plus concr\`ete, si $\mathscr{E}$ est un complexe \`a parit\'e sur $\Flag_{\mathscr{G}}$, qu'on identifie \`a l'objet de $\Dmix_{\mathrm{Br}}(\Flag_{\mathscr{G}},\bk)$ donn\'e par le complexe concentr\'e en degr\'e $0$, avec $\mathscr{E}$ en degr\'e $0$, alors on a
\begin{equation}
\label{eqn:phi-parity}
\varphi_{\mathscr{G}}^\bk([\mathscr{E}])= \sum_{\substack{n \in \Z \\ w \in \WKM}} \dim_\bk \mathsf{H}^{-\ell(w)-n}(\Flag_{\mathscr{G},w}, \mathscr{E}_{|\Flag_{\mathscr{G},w}}) \cdot v^n H_w.
\end{equation}

\begin{defn}
Pour $w \in \WKM$, on pose
\[
\puH^{\mathscr{G}}_w := \varphi_{\mathscr{G}}^\bk(\mathscr{E}^\mix_{\mathscr{G},w}).
\]
\end{defn}

Il est facile de v\'erifier que :
\begin{enumerate}
\item
les \'el\'ements $\puH^{\mathscr{G}}_w$ 
ne d\'ependent pas de $\bk$ lui-m\^eme, mais uniquement de sa caract\'eristique (ce qui justifie la notation) ;
\item
la famille $(\puH^{\mathscr{G}}_w : w \in \WKM)$
est une base de $\mathcal{H}_\WKM$.
\end{enumerate}

La famille $(\puH^{\mathscr{G}}_w : w \in \WKM)$ est la \emph{base $p$-canonique} de Williamson (voir~\cite{jw} ; le lien entre la d\'efinition adopt\'ee dans~\cite{jw} et celle consid\'er\'ee ici---qui est sugg\'er\'ee dans~\cite{jmw}---est \'etabli dans~\cite[Partie III]{rw}). Dans le cas o\`u $p=0$, des r\'esultats de Kazhdan--Lusztig~\cite{kl} (voir \'egalement~\cite{springer}) montrent que $\mathscr{E}_{\mathscr{G},w}$ est le complexe de cohomologie d'intersection de $\Flag_{\mathscr{G,w}}$, et que ${}^0 \hspace{-1pt} \underline{H}_w^{\mathscr{G}} = \underline{H}_w$ pour tout $w \in \WKM$. Il est connu \'egalement que pour tout $w \in \WKM$ il existe $N(w) \in \Z_{\geq 0}$ tel que $\puH_w^{\mathscr{G}} = \underline{H}_w$ pour tout nombre premier $p \geq N(w)$; mais l'entier $N(w)$ n'est pas connu, et semble difficile \`a d\'ecrire (sauf dans des cas tr\`es particuliers).

De m\^eme que les polyn\^omes de Kazhdan--Lusztig se d\'efinissent \`a partir de la base de Kazhdan--Lusztig de $\mathcal{H}_\WKM$, on d\'efinit les \emph{$p$-polyn\^omes de Kazhdan--Lusztig} $(\ph^{\mathscr{G}}_{x,y} : x,y \in \WKM)$ comme les coefficients des \'el\'ements de la base $p$-canonique dans la base standard, c'est-\`a-dire par l'\'egalit\'e
\[
\puH^{\mathscr{G}}_y = \sum_{x \in \WKM} \ph^{\mathscr{G}}_{x,y} \cdot H_x
\]
pour $y \in \WKM$.
Contrairement aux polyn\^omes de Kazhdan--Lusztig usuels, ces polyn\^omes n'appartiennent pas n\'ecessairement \`a $\Z[v]$ (mais plut\^ot \`a $\Z[v,v^{-1}]$). Par contre, d'apr\`es~\eqref{eqn:phi-parity}, ils ont des coefficients positifs ou nuls.

\begin{rmq}
Le groupe $\WKM$ est le m\^eme pour $\mathscr{G}$ et pour le groupe dual $\mathscr{G}^\vee$ consid\'er\'e au~\S\ref{ss:Koszul-constructible}. Dans tous les exemples donn\'es ci-dessus, on a $\puH^{\mathscr{G}}_w = \puH^{\mathscr{G}^\vee}_w$. Cependant, cette \'egalit\'e est fausse en g\'en\'eral; voir~\cite{jw} pour des exemples explicites.
\end{rmq}

Bien s\^ur, on peut consid\'erer des constructions similaires pour la vari\'et\'e de drapeaux oppos\'ee $\Flag_{\mathscr{G}}^\op$ \`a la place de $\Flag_{\mathscr{G}}$. On a alors un isomorphisme
\[
\varphi_{\mathscr{G},\op}^\bk : [\Dmix_{\mathrm{Br}}(\Flag_{\mathscr{G}}^\op,\bk)] \simto \mathcal{H}_\WKM
\]
envoyant $[\dmix_{\mathscr{G},\op,w}\{n\}]$ sur $v^n H_w$. En utilisant le fait que les objets $\mathscr{E}_{\mathscr{G},w}$ se ``rel\`event'' dans la cat\'egorie $\mathscr{B}$-\'equivariante, on peut v\'erifier sans difficult\'e qu'on a encore
\begin{equation}
\label{eqn:varphi-Eop}
\puH^{\mathscr{G}}_w = \varphi_{\mathscr{G},\op}^\bk(\mathscr{E}^\mix_{\mathscr{G},\op,w})
\end{equation}
pour tout $w \in \WKM$. (Au vu des commentaires \`a la fin du~\S\ref{ss:Dmix}, cela revient \`a dire que $\ph^{\mathscr{G}}_{x,y} = \ph^{\mathscr{G}}_{x^{-1},y^{-1}}$ pour tous $x,y \in \WKM$.)

\subsection{Dualit\'e de Koszul et base $p$-canonique}
\label{ss:Koszul-pcan}

On consid\`ere maintenant le groupe de Kac--Moody $\mathscr{G}^\vee$ qui est dual de $\mathscr{G}$ au sens de Langlands, sa vari\'et\'e de drapeaux oppos\'ee $\Flag^\vee_{\op}$, et l'\'equivalence $\kappa$ du Th\'eor\`eme~\ref{thm:mkd}. 

\begin{lem}
\label{lem:koszul-iota}
L'automorphisme de $\mathcal{H}_{\WKM}$ donn\'e par $\varphi^{\bk}_{\mathscr{G}^\vee,\mathrm{op}} \circ [\kappa] \circ (\varphi^\bk_{\mathscr{G}})^{-1}$ coincide avec l'involution d'anneau donn\'ee par
\[
\iota(v^n H_w) = (-v)^{-n} H_w
\]
pour $n \in \Z$ et $w \in \WKM$.
\end{lem}

\begin{proof}
La formule d\'ecoule du fait que
\[
\kappa(\dmix_{\mathscr{G},w} \{n\}) \cong \Delta^{\mix}_{\mathscr{G}^\vee, \op, w} \langle n \rangle = \Delta^{\mix}_{\mathscr{G}^\vee,\op,w} \{ -n \}[n]
\]
pour $n \in \Z$ et $w \in \WKM$.
\end{proof}

Le r\'esultat suivant est une traduction de~\eqref{eqn:kappa-E-T} dans les groupes de Grothendieck. (On peut consid\'erer ce fait comme la ``trace combinatoire'' de la dualit\'e de Koszul.)

\begin{prop}
\label{prop:caracteres-basculant-constr}
Pour tout $w \in \WKM$ on a
\[
\varphi_{\mathscr{G}^\vee,\op}^{\bk}([\mathscr{T}^{\mix}_{\mathscr{G}^\vee,\op,w}]) = \iota (\puH^{\mathscr{G}}_w).
\]
\end{prop}

\begin{rmq}
\begin{enumerate}
\item
Consid\'erons le cas particulier o\`u $p=0$. Dans ce cas, comme expliqu\'e au~\S\ref{ss:base-p-can}, on a ${}^0 \hspace{-1pt} \underline{H}_w^{\mathscr{G}} = \underline{H}_w$ pour tout $w \in \WKM$. Dans les conventions de Kazhdan--Lusztig~\cite{kl0}, la base $(\underline{H}_w : w \in \WKM)$ est not\'ee $(C'_w : w \in \WKM)$. Kazhdan et Lusztig consid\`erent \'egalement une autre base, not\'ee $(C_w : w \in \WKM)$. Avec nos notations, on a $C_w = \iota(C'_w)$. En d'autres termes, pour $p$ g\'en\'eral, la base $(\iota (\puH^{\mathscr{G}}_w) : w \in \WKM)$ peut \^etre consid\'er\'ee comme une ``version $p$-canonique'' de ``l'autre'' base de Kazhdan--Lusztig.
\item
Toujours dans le cas particulier o\`u $p=0$, les ``caract\`eres'' des faisceaux pervers mixtes basculants sur les vari\'et\'es de drapeaux ont \'et\'e d\'etermin\'es par Yun~\cite{yun}. Dans ce cas particulier, notre formule est identique \`a la sienne. (Yun utilise une autre notion de ``faisceaux pervers mixtes'' ; voir~\cite{ar:kdsf} pour le lien avec notre d\'efinition.)
\item
Revenons au cas o\`u $p$ est quelconque. En utilisant~\eqref{eqn:varphi-Eop} pour le groupe $\mathscr{G}^\vee$, et le fait que $\iota$ est une involution, on trouve aussi que
$\varphi_{\mathscr{G}}^{\bk}([\mathscr{T}^{\mix}_{\mathscr{G},w}]) = \iota (\puH^{\mathscr{G}^\vee}_w)$
pour tout $w \in \WKM$. En comparant avec la formule de la Proposition~\ref{prop:caracteres-basculant-constr} (appliqu\'ee au groupe $\mathscr{G}^\vee$), il s'ensuit que
\[
\varphi_{\mathscr{G}}^{\bk}([\mathscr{T}^{\mix}_{\mathscr{G},w}]) = \varphi_{\mathscr{G},\op}^{\bk}([\mathscr{T}^{\mix}_{\mathscr{G},\op,w}])
\]
pour tout $w \in \WKM$. Cette \'egalit\'e peut \'egalement se d\'emontrer directement (c'est-\`a-dire sans utiliser le Th\'eor\`eme~\ref{thm:mkd}) en utilisant les ``objets libre-mono\-dromiques'' consid\'er\'es dans~\cite{amrw1}.
\end{enumerate}
\end{rmq}

\subsection{Le cas particulier des vari\'et\'es de drapeaux affines}
\label{ss:Koszul-affine}

Consid\'erons le cadre du~\S\ref{ss:drap-affine}, et supposons de plus que $G$ est quasi-simple et simplement connexe. Alors $\Fl_G$ est la vari\'et\'e de drapeaux d'un groupe de Kac--Moody, et on peut donc lui appliquer les r\'esultats consid\'er\'es au~\S\ref{ss:Koszul-constructible} et au~\S\ref{ss:base-p-can}. Le ``groupe de Weyl'' $\WKM$ correspondant est le groupe de Weyl affine non \'etendu associ\'e \`a $G$ (ou au groupe dual, selon les conventions choisies) et qui se d\'ecrit comme $\Waff = W \ltimes \Z R^\vee$, o\`u $\Z R^\vee$ est le r\'eseau des coracines de $(G,T)$.

Cependant, le groupe de Kac--Moody dual du groupe affine non tordu associ\'e \`a $G$ n'est pas toujours du m\^eme type : il s'agit d'un groupe de Kac--Moody affine, mais tordu si $G$ n'est pas simplement lac\'e. Cependant, si on suppose que $p$ est tr\`es bon pour $G$ (et, comme toujours, diff\'erent de $2$), la cat\'egorie des complexes \`a parit\'e sur la vari\'et\'e de drapeaux de ce groupe dual s'identifie canoniquement \`a celle des complexes \`a parit\'e sur
$\Fl_G^{\op}$ ;
voir~\cite[\S 7.1]{amrw2} pour plus de d\'etails. L'alg\`ebre de Hecke associ\'ee sera not\'ee $\mathcal{H}^G_{\mathrm{aff}}$, et sa base $p$-canonique $(\puH_w : w \in \Waff)$.

Plut\^ot que d'\'enoncer les r\'esultats du~\S\ref{ss:Koszul-pcan} dans ce cadre, nous allons les \'etendre au cas des groupes r\'eductifs g\'en\'eraux. Consid\'erons donc le contexte du~\S\ref{ss:drap-affine}, o\`u $G$ est maintenant un groupe reductif connexe complexe quelconque. Le ``groupe de Weyl affine \'etendu'' associ\'e est d\'efini par $\Wext=W \ltimes X_*(T)$. Il contient le groupe $\Waff := W \ltimes \Z R^\vee$ comme sous-groupe distingu\'e (o\`u, comme ci-dessus, $\Z R^\vee$ est le r\'eseau des coracines de $(G,T)$). Comme au~\S\ref{ss:coh-rep} le groupe $\Waff$ admet une structure naturelle de groupe de Coxeter, dont la fonction de longueur v\'erifie
\[
\ell(w \cdot t_\lambda)=\sum_{\substack{\alpha \in R^+ \\  w(\alpha) \in R^+}} |\langle \lambda,
\alpha \rangle | + \sum_{\substack{\alpha \in
      R^+ \\ w(\alpha) \in -R^+}}
  |1 + \langle \lambda, \alpha \rangle |.
\]
Cette formule permet d'\'etendre $\ell$ \`a $\Wext$. De plus $\Omega := \{w \in \Wext \mid \ell(w)=0\}$ est un sous-groupe ab\'elien de $\Wext$, isomorphe \`a $X_*(T) / \Z R^\vee$, dont l'action sur $\Waff$ (par conjugaison) pr\'eserve les g\'en\'erateurs de Coxeter, et tel que la multiplication induit un isomorphisme
\[
\Omega \ltimes \Waff \simto \Wext.
\]

\begin{rmq}
Dans le contexte consid\'er\'e au~\S\ref{ss:coh-rep}, avec le groupe ``$G$'' consid\'er\'e ici \'egal au groupe $G^\vee$ dual de Langlands de $\dot{\bG}$, le groupe de Weyl affine \'etendu s'identifie canoniquement avec celui not\'e de la m\^eme mani\`ere au~\S\ref{ss:coh-rep}. (Dans ce cas particulier le quotient $X^*(T) / \Z R$ est sans torsion ; cette hypoth\`ese n'est pas n\'ecessaire pour les r\'esultats consid\'er\'es ici.)
\end{rmq}

On notera
$\mathcal{H}^G_{\mathrm{ext}}$ l'alg\`ebre de Hecke affine ``\'etendue'', c'est-\`a-dire la $\Z[v,v^{-1}]$-alg\`ebre engendr\'ee par des \'el\'ements $(H_w : w \in \Wext)$, soumis aux relations
\begin{enumerate}
\item
$(H_s + v)(H_s - v^{-1})=0$ si $s \in \Waff$ et $\ell(s)=1$;
\item
$H_x \cdot H_y = H_{xy}$ si $x,y \in \Wext$ et $\ell(xy)=\ell(x)+\ell(y)$.
\end{enumerate}
L'application $\omega \mapsto H_\omega$ induit un morphisme de groupes de $\Omega$ vers le groupe des \'el\'ements inversibles de $\mathcal{H}^G_{\mathrm{ext}}$. D'autre part, l'alg\`ebre de Hecke $\mathcal{H}^G_{\mathrm{aff}}$ du groupe de Coxeter $\Waff$ est naturellement une sous-alg\`ebre de $\mathcal{H}^G_{\mathrm{ext}}$, et la multiplication induit un isomorphisme
\[
 \Omega \ltimes \mathcal{H}^G_{\mathrm{aff}} \simto \mathcal{H}^G_{\mathrm{ext}}
\]
(o\`u l'action de $\Omega$ sur $\mathcal{H}^G_{\mathrm{aff}}$ est induite par l'action sur $\Waff$). On notera encore $\iota$ l'involution d'anneau de $\mathcal{H}^G_{\mathrm{ext}}$ d\'etermin\'ee par $\iota(v^n H_w)=(-v)^{-n} H_w$ pour $n \in \Z$ et $w \in \Wext$.

On ``\'etend'' la base $p$-canonique de $\mathcal{H}^G_{\mathrm{aff}}$ en une base de $\mathcal{H}^G_{\mathrm{ext}}$ en posant
\[
\puH_w := H_\omega \cdot \puH_y \quad \text{si $w=\omega y$ avec $\omega \in \Omega$ et $y \in \Waff$.}
\]
(On peut v\'erifier que si $w=y' \omega'$ avec $y' \in \Waff$ et $\omega' \in \Omega$, alors on a aussi $\puH_w = \puH_{y'} \cdot H_{\omega'}$.)

Du point de vue g\'eom\'etrique, on peut consid\'erer comme au~\S\ref{ss:drap-affine} les vari\'et\'es de drapeaux affines $\Fl_G$ et $\Fl_G^\op$, leurs cat\'egories d\'eriv\'ees mixtes $\Dmix_{\mathrm{Br}}(\Fl_G,\bk)$ et $\Dmix_{\mathrm{Br}}(\Fl^\op_G,\bk)$, et les isomorphismes naturels
\[
\varphi^{\bk}_{\Fl_G} : [\Dmix_{\mathrm{Br}}(\Fl_G,\bk)] \simto \mathcal{H}^G_{\mathrm{ext}}, \qquad
\varphi^{\bk}_{\Fl_G^\op} : [\Dmix_{\mathrm{Br}}(\Fl^\op_G,\bk)] \simto \mathcal{H}^G_{\mathrm{ext}}
\]
d\'etermin\'es par $\varphi^{\bk}_{\Fl_G}(\dmix_{\Fl_G,w}\{n\}) = v^n H_w$ et $\varphi^{\bk}_{\Fl^\op_G}(\dmix_{\Fl^\op_G,w}\{n\}) = v^n H_w$ pour $n \in \Z$ et $w \in \Wext$.

\begin{prop}
\label{prop:pcan-reductif}
Supposons que $p$ est tr\`es bon pour $G$ et diff\'erent de $2$.
Pour tout $w \in \Wext$ on a
\[
\varphi^{\bk}_{\Fl_G}([\mathscr{E}_{\Fl_G,w}^\mix])=\puH_w, \quad \varphi^{\bk}_{\Fl_G^\op}([\mathscr{T}^\mix_{\Fl_G^\op, w}]) = \iota(\puH_w).
\]
\end{prop}

\begin{proof}
On se ram\`ene facilement au cas o\`u $w \in \Waff$. Dans ce cas les formules voulues d\'ecoulent du Th\'eor\`eme~\ref{thm:mkd}, puisque la composante connexe $\Fl^\circ_G$ de $\Fl_G$ contenant le point de base s'identifie \`a la vari\'et\'e de drapeaux affine du recouvrement simplement connexe du sous-groupe d\'eriv\'e de $G$, qui lui-m\^eme est un produit de groupes quasi-simples et simplement connexes (voir~\cite[\S 7.1]{amrw2} pour plus de d\'etails).
\end{proof}

\subsection{Dualit\'e de Koszul constructible parabolique}
\label{ss:koszul-parabolique-aff}

Dans~\cite{amrw2} nous d\'emontrons \'egalement une version ``parabolique'' du Th\'eor\`eme~\ref{thm:mkd}. Pour simplifier, ici nous n'\'enoncerons ce r\'esultat que dans un cas particulier, dans le cadre des vari\'et\'es de drapeaux affines. 

On continue de supposer que $G$ est un groupe r\'eductif connexe complexe quelconque, et on
consid\`ere comme au~\S\ref{ss:drap-affine} sa ``Grassmannienne affine oppos\'ee''
$\Gr_G^{\op} := G_{\mathscr{O}} \backslash G_{\mathscr{K}}$.
Cette vari\'et\'e admet une d\'ecomposition de Bruhat en terme d'orbites pour l'action du sous-groupe d'Iwahori $I$ induite par multiplication \`a droite sur $G_{\mathscr{K}}$; ses strates sont param\'etr\'ees par le sous-ensemble $\Wextmin \subset \Wext$ form\'e des \'el\'ements $w$ qui sont de longueur minimale dans $W \cdot w$. Comme pour les vari\'et\'es de drapeaux des groupes de Kac--Moody, on lui associe une cat\'egorie d\'eriv\'ee mixte, qu'on notera $\Dmix_{\mathrm{Br}}(\Gr_G^\op,\bk)$. Les objets standards et costandards dans cette cat\'egorie seront not\'es $\dmix_{\Gr_G^\op,w}$ et $\nmix_{\Gr_G^\op,w}$ respectivement ($w \in \Wextmin$).

Fixons maintenant un corps $\F$ alg\'ebriquement clos de caract\'eristique positive $\ell \neq p$, et supposons qu'il existe une racine primitive $\ell$-i\`eme de l'unit\'e dans $\bk$ (qu'on fixe). Notons $G_{\F}$ le $\F$-groupe alg\'ebrique r\'eductif connexe dont la donn\'ee radicielle est la m\^eme que celle de $(G,T)$. Notons \'egalement $B_\F \subset G_\F$ le sous-groupe de Borel et $T_\F \subset B_\F$ le tore maximal correspondant \`a $B$ et $T$. Nous noterons $B_\F^+$ le sous-groupe de Borel de $G_\F$ qui est oppos\'e \`a $B_\F$ par rapport \`a $T_\F$, et $U^+_\F$ son radical unipotent. En fixant une trivialisation de chacun des sous-espaces radiciels de $U^+_\F$ correspondant \`a une racine simple, et en composant avec le morphisme somme, on obtient un ``caract\`ere additif non d\'eg\'en\'er\'e''
\[
\chi : U^+_\F \to \mathbb{G}_{\mathrm{a},\F}.
\]
Notre choix de racine primitive $\ell$-i\`eme de l'unit\'e dans $\bk$ d\'efinit un faisceau d'Artin--Schreier sur $\mathbb{G}_{\mathrm{a},\F}$, qu'on notera $\mathscr{L}_{\mathrm{AS}}$. 

Consid\'erons maintenant la vari\'et\'e de drapeaux affine
\[
\Fl_{G_\F} = G_{\F ( \hspace{-1pt} (z) \hspace{-1pt} )} / I_\F
\]
d\'efinie comme pr\'ec\'edemment, en rempla{\c c}ant $\C$ par $\F$. On consid\'erera \'egalement le sous-groupe d'Iwahori $I^+_\F \subset G_{\F [ \hspace{-1pt} [z] \hspace{-1pt} ]}$ associ\'e \`a $B_\F^+$, dont on notera $I^+_{\mathrm{u},\F}$ le radical pro-unipotent. Il existe une surjection naturelle $I^+_{\mathrm{u},\F} \to U^+_\F$, dont on notera $\chi_I$ la compos\'ee avec $\chi$.
Nous noterons alors
\[
\Db_{\IW}(\Fl_G,\bk)
\]
la cat\'egorie d\'eriv\'ee des $\bk$-faisceaux \'etales sur $\Fl_{G_\F}$ qui sont $(I^+_{\mathrm{u},\F}, \chi_I^* \mathscr{L}_{\mathrm{AS}})$-\'equi\-var\-iants. (Ici, la notation $\IW$ est un symbole pour ``Iwahori--Whittaker''.) Dans cette cat\'egorie on peut consid\'erer les objets \`a parit\'e, puis d\'efinir la cat\'egorie d\'eriv\'ee mixte associ\'ee $\Dmix_{\IW}(\Fl,\bk)$ et sa t-structure perverse. Les orbites de $I^+_{\mathrm{u},\F}$ sur $\Fl$ sont param\'etr\'ees naturellement par $\Wext$; celles qui supportent un syst\`eme local $(I^+_{\mathrm{u},\F}, \chi_I^* \mathscr{L}_{\mathrm{AS}})$-\'equivariant non nul sont celles param\'etr\'ees par les \'el\'ements de $\Wextmin$. Les objets standards et costandards associ\'es \`a $w \in \Wextmin$ dans $\Dmix_{\IW}(\Fl,\bk)$ seront not\'es $\dmix_{\IW,w}$ et $\nmix_{\IW,w}$ respectivement.

Dans~\cite[Theorem~7.4]{amrw2}, sous l'hypoth\`ese o\`u $G$ est quasi-simple et simplement connexe, on construit une dualit\'e de Koszul formelle
\begin{equation}
\label{eqn:Koszul-IW-Gr}
\Db_{\IW}(\Fl,\bk) \simto \Dmix_{\mathrm{Br}}(\Gr_G^\op,\bk)
\end{equation}
dont la bijection associ\'ee sur les ensembles de param\`etres est 
$\id_{\Waffmin}$.

Cette dualit\'e admet l'interpr\'etation ``combinatoire'' suivante, qui s'\'etend en fait au cas des groupes r\'eductifs arbitraires. Consid\'erons l'alg\`ebre de Hecke $\mathcal{H}_W$ associ\'ee \`a $W$ (pour sa structure de groupe de Coxeter d\'etermin\'ee par $B$). Alors $\mathcal{H}_W$ est une sous-alg\`ebre de $\mathcal{H}^G_\ext$, et elle poss\`ede deux actions naturelles sur $\Z[v,v^{-1}]$: l'action ``triviale'', pour laquelle $H_w$ agit par multiplication par $v^{-\ell(w)}$ (pour tout $w \in W$), et l'action ``signe'', pour laquelle $H_w$ agit par multiplication par $(-v)^{\ell(w)}$ (pour tout $w \in W$). Les modules correspondants seront not\'es $\Z[v,v^{-1}]_{\mathrm{triv}}$ et $\Z[v,v^{-1}]_{\mathrm{sgn}}$. On consid\`ere alors les modules ``sph\'erique'' et ``antisph\'erique'' de $\mathcal{H}^G_\ext$ d\'efinis par
\[
\mathcal{M}^{\mathrm{sph}}_\ext := \Z[v,v^{-1}]_{\mathrm{triv}} \otimes_{\mathcal{H}_W} \mathcal{H}^G_\ext \quad \text{et} \quad \mathcal{M}^{\mathrm{asph}}_\ext := \Z[v,v^{-1}]_{\mathrm{sgn}} \otimes_{\mathcal{H}_W} \mathcal{H}^G_\ext
\]
respectivement. L'automorphisme $\iota$ induit un isomorphisme $\mathcal{M}^{\mathrm{sph}}_\ext \simto \mathcal{M}^{\mathrm{asph}}_\ext$, qu'on notera encore $\iota$, et qui v\'erifie
\begin{equation}
\label{eqn:iota-M-N}
\iota(x \cdot h)=\iota(x) \cdot \iota(h)
\end{equation}
pour $x \in \mathcal{M}^{\mathrm{sph}}_\ext$ et $h \in \mathcal{H}^G_\ext$.

Chacun de ces modules poss\`ede une ``base standard'' param\'etr\'ee par $\Wextmin$, et d\'efinie par
\[
M_w := 1 \otimes H_w \in \mathcal{M}^{\mathrm{sph}}_\ext, \qquad N_w := 1 \otimes H_w \in \mathcal{M}^{\mathrm{asph}}_\ext.
\]
Ils poss\`edent \'egalement des bases $p$-canoniques $(\puM_w : w \in \Wextmin)$ et $(\puN_w : w \in \Wextmin)$ qui sont reli\'ees \`a la base canonique de $\mathcal{H}^G_\ext$ de la fa{\c c}on suivante. (Ces bases peuvent \^etre d\'efinies ind\'ependemment des relations qui suivent, mais elles sont caract\'eris\'ees par celles-ci.) Notons $\xi : \mathcal{H}^G_\ext \to \mathcal{M}^{\mathrm{asph}}_{\ext}$ l'application naturelle (avec $\xi(H)=1 \otimes H$) et $\eta : \mathcal{M}^{\mathrm{sph}}_\ext \to \mathcal{H}^G_\ext$ l'application d\'etermin\'ee par $\eta(1 \otimes H) = \underline{H}_{w_0} \cdot H$, o\`u $w_0$ est l'\'el\'ement de plus grande longueur dans $W$. Alors pour $w \in \Wext$ on a
\begin{equation}
\label{eqn:xi-pcan}
\xi(\puH_w) = \begin{cases}
\puN_w & \text{si $w \in \Wextmin$;} \\
0 & \text{sinon}
\end{cases}
\end{equation}
et pour $w \in \Wextmin$ on a
\begin{equation}
\label{eqn:eta-pcan}
\eta(\puM_x) = \puH_{w_0 x}.
\end{equation}
Ces bases d\'eterminent des $p$-polyn\^omes de Kazhdan--Lusztig sph\'eriques et antisph\'eriques respectivement, d\'efinis par les \'egalit\'es
\[
\puM_y = \sum_{x \in \Wextmin} \pmm_{x,y} \cdot M_x, \qquad \puN_y = \sum_{x \in \Wextmin} \pn_{x,y} \cdot N_x.
\]
(Comme dans le cas des bases $p$-canoniques des alg\`ebres de Hecke au~\S\ref{ss:base-p-can}, ces bases d\'ependent du groupe $G$ et pas seulement du groupe $\Wext$. Mais pour all\'eger les notations nous ne ferons pas appara\^itre cette d\'ependance.)

On a des isomorphismes canoniques
\begin{equation}
\label{eqn:Msph-Masph-cat}
[\Dmix_{\mathrm{Br}}(\Gr_G^\op,\bk)] \simto \mathcal{M}_\ext^{\mathrm{sph}}, \qquad [\Db_{\IW}(\Fl,\bk)] \simto \mathcal{M}_\ext^{\mathrm{asph}}
\end{equation}
envoyant $[\dmix_{\Gr_G^\op,w}\{n\}]$ sur $v^n M_w$ et $[\dmix_{\IW,w}\{n\}]$ sur $v^n N_w$ respectivement. Les images des objets \`a parit\'e ind\'ecomposables associ\'es \`a $w \in \Wextmin$ dans $\Dmix_{\mathrm{Br}}(\Gr_G^\op,\bk)$ et $\Db_{\IW}(\Fl,\bk)$ seront not\'es $\mathscr{E}^\mix_{\Gr_G^\op,w}$ et $\mathscr{E}^\mix_{\IW,w}$ respectivement.
Pour $w \in\Wextmin$, on notera \'egalement $\mathscr{T}^\mix_{\Gr_G^\op,w}$ et $\mathscr{T}^\mix_{\IW,w}$ les faisceaux pervers basculants associ\'es \`a $w$ dans $\Dmix_{\mathrm{Br}}(\Gr_G^\op,\bk)$ et $\Db_{\IW}(\Fl,\bk)$ respectivement.

\begin{rmq}
Les formules~\eqref{eqn:xi-pcan} et~\eqref{eqn:eta-pcan} ont des interpr\'etations en termes de complexes \`a parit\'e sur les vari\'et\'es de drapeaux affines : voir~\cite[Corollary~11.10]{rw} pour~\eqref{eqn:xi-pcan} et~\cite[Lemma~A.5]{acr} pour~\eqref{eqn:eta-pcan}.
\end{rmq}

Le r\'esultat suivant d\'ecoule des propri\'et\'es de l'\'equivalence~\eqref{eqn:Koszul-IW-Gr}, suivant les m\^emes arguments que pour la Proposition~\ref{prop:pcan-reductif}.

\begin{prop}
\label{prop:character-tilting}
Via les identifications~\eqref{eqn:Msph-Masph-cat}, les classes $[\mathscr{E}^\mix_{\Gr_G^\op,w}]$ et $[\mathscr{T}^\mix_{\Gr_G^\op,w}]$ correspondent \`a $\puM_w$ et $\iota^{-1}(\puN_w)$ respectivement, et les classes $[\mathscr{E}^\mix_{\IW,w}]$ et $[\mathscr{T}^\mix_{\IW,w}]$ correspondent \`a $\puN_w$ et $\iota(\puM_w)$ respectivement.
\end{prop}

\begin{rmq}
\label{rmq:basculants-mix-red}
Au~\S\ref{ss:drap-affine} on a consid\'er\'e la cat\'egorie d\'eriv\'ee mixte $\Dmix_{(\Iw)}(\Gr_G,\bk)$ des faisceaux sur la Grassmannienne affine ``usuelle'' $\Gr_G = G_{\mathscr{K}} / G_{\mathscr{O}}$ munie de la stratification par les orbites du sous-groupe d'Iwahori \emph{positif} $\Iw$. Cette cat\'egorie est reli\'ee \`a cat\'egorie $\Dmix_{\mathrm{Br}}(\Gr_G^\op,\bk)$ (d\'efinie pour la stratification par les orbites du sous-groupe d'Iwahori \emph{n\'egatif}) consid\'er\'ee ici de la fa{\c c}on suivante. Consid\'erons un anti-automorphisme de $G$ qui se restreint \`a l'identit\'e sur $T$, et qui \'echange les sous-groupes de Borel positif et n\'egatif.
Cet anti-automorphisme induit un anti-automorphisme de $G_\mathscr{K}$ qui stabilise $G_\mathscr{O}$, et donc un isomorphisme de ind-vari\'et\'es $\Gr_G \simto \Gr_G^\op$. Cet isomorphisme \'echange la $\Iw$-orbite param\'etr\'ee par $\lambda \in X_*(T)$ et la $I$-orbite param\'etr\'ee par $w_\lambda$; elle induit donc une \'equivalence de cat\'egories
\[
\Dmix_{(\Iw)}(\Gr_G,\bk) \simto \Dmix_{\mathrm{Br}}(\Gr_G^\op,\bk)
\]
envoyant $\dmix_{\Gr_G,\lambda}$ sur $\dmix_{\Gr_G^\op,w_\lambda}$, $\nmix_{\Gr_G,\lambda}$ sur $\nmix_{\Gr_G^\op,w_\lambda}$, $\mathscr{E}^\mix_{\Gr_G,\lambda}$ sur $\mathscr{E}^\mix_{\Gr_G^\op,w_\lambda}$ et $\mathscr{T}^\mix_{\Gr_G,\lambda}$ sur $\mathscr{T}^\mix_{\Gr_G^\op,w_\lambda}$. On en d\'eduit un isomorphisme
\begin{equation}
\label{eqn:Msph-Iw-Gr}
[\Dmix_{(\Iw)}(\Gr_G,\bk)] \simto \mathcal{M}_\ext^{\mathrm{sph}}
\end{equation}
envoyant $[\dmix_{\Gr_G,\lambda}\{n\}]$ sur $v^n M_{w_\lambda}$, $[\mathscr{E}^\mix_{\Gr_G,\lambda}]$ sur $\puM_{w_\lambda}$ et $[\mathscr{T}^\mix_{\Gr_G,\lambda}]$ sur $\iota^{-1}(\puN_{w_\lambda})$, pour tout $\lambda \in X_*(T)$.
\end{rmq}

\subsection{Base $p$-canonique et faisceaux exotiques}
\label{ss:p-exotique}

Revenons aux notations du~\S\ref{ss:ets}. La base $p$-canonique du module $\mathcal{M}_\ext^{\mathrm{sph}}$ encode \'egalement la combinatoire des objets basculants et \`a parit\'e pour la t-structure exotique sur $\Db\Coh^{\GCoh \times \Gm}(\tcN_\GCoh)$. En effet, en combinant l'isomorphisme~\eqref{eqn:Msph-Iw-Gr} avec celui induit par l'\'equivalence du Th\'eor\`eme~\ref{thm:coh-const}, on obtient un isomorphisme
\[
[\Db\Coh^{\GCoh \times \Gm}(\tcN_\GCoh)] \simto \mathcal{M}_\ext^{\mathrm{sph}}
\]
envoyant $[\Delta^\Coh_\lambda \langle n \rangle]$ sur $v^n M_{w_\lambda}$, $[\mathcal{T}_\lambda]$ sur $\puM_{w_\lambda}$ et $[\mathcal{E}_\lambda]$ sur $\iota^{-1}(\puN_{w_\lambda})$, pour tout $\lambda \in \bX$ et tout $n \in \Z$. (Ici l'alg\`ebre de Hecke consid\'er\'ee est celle associ\'ee au groupe r\'eductif complexe $G^\vee$ dual de $\GCoh$ au sens de Langlands.)
Sous cet isomorphisme, l'action de $\mathcal{H}_\ext$ est induite par l'action cat\'egorique du groupe de tresses affine \'etudi\'ee dans~\cite{br,riche}; plus pr\'ecis\'ement, l'action (\`a droite) de $H_w$ est induite par le foncteur not\'e $\mathscr{J}_{(T_{w^{-1}})^{-1}}$ dans~\cite{mr:ets}.

\section{Applications en th\'eorie des repr\'esentations}
\label{sec:applications}

\subsection{Formule de caract\`eres pour les modules basculants}
\label{ss:rw-form}

Notre motivation principale pour d\'evelopper les travaux expos\'es au~\S\ref{sec:exemples} \'etait une application \`a un calcul de caract\`eres dans la th\'eorie des repr\'esentations des groupes r\'eductifs, qu'on pr\'esente maintenant.

Comme au~\S\ref{ss:coh-rep}, soit $\bG$ un groupe alg\'ebrique r\'eductif dont le sous-groupe d\'eriv\'e est simplement connexe, soit $\bB \subset \bG$ un sous-groupe de Borel, et soit $\bT \subset \bB$ un tore maximal. Si on consid\`ere les constructions du~\S\ref{ss:koszul-parabolique-aff} pour le groupe $G=G^\vee$ qui est dual de $\dot{\bG}$ au sens de Langlands, le groupe $\Waff$ s'interpr\`ete comme le ``groupe de Weyl affine'' associ\'e \`a $\bG$ comme dans~\cite{jantzen}. L'alg\`ebre de Hecke affine \'etendue associ\'ee sera not\'ee simplement $\mathcal{H}_\ext$. On supposera que $p>h$, o\`u $h$ est le nombre de Coxeter de $\bG$.

La formule suivante a \'et\'e conjectur\'ee par G. Williamson et le second auteur dans~\cite{rw}. Dans cet \'enonc\'e on consid\`ere la base $p$-canonique du module antisph\'erique d\'efinie comme au~\S\ref{ss:koszul-parabolique-aff}, pour le  groupe $G^\vee$, et on note $(\mathsf{T}_w : \mathsf{N}_y)$ le nombre d'occurences de $\mathsf{N}_y$ dans une filtration costandard de $\mathsf{T}_w$.

\begin{conj}
\label{conj:rw}
Pour tous $y,w \in \Wextmin$ on a
\[
(\mathsf{T}_w : \mathsf{N}_y) = \pn_{y,w}(1).
\]
\end{conj}

\begin{rmq}\phantomsection
\label{rmq:conj-rW}
\begin{enumerate}
\item
Plus pr\'ecis\'ement, dans~\cite{rw} cette formule est conjectur\'ee sous l'hypoth\`ese o\`u $y,w$ sont dans $\Wextmin$. Cependant, en utilisant les propri\'et\'es usuelles des foncteurs de translation (voir~\cite{jantzen}) on voit que les deux versions sont \'equivalentes.
\item
La formule \'enonc\'ee dans la Conjecture~\ref{conj:rw} est tr\`es proche de (et inspir\'ee par) la formule de multiplicit\'es pour les modules basculants des groupes quantiques de Lusztig en une racine de l'unit\'e conjectur\'ee par Soergel dans~\cite{soergel-comb-tilting} puis d\'emontr\'ee peu apr\`es par lui-m\^eme dans~\cite{soergel-char-tilt} (voir \'egalement~\cite{soergel-char-for}). Une extension de cette formule au cas des repr\'esentations modulaires des groupes r\'eductifs a \'et\'e propos\'ee par Andersen (voir notamment~\cite{andersen}). Mais la Conjecture~\ref{conj:rw} est diff\'erente de celle d'Andersen dans la mesure o\`u elle s'exprime en termes des $p$-polyn\^omes de Kazhdan--Lusztig, et non des polyn\^omes de Kazhdan--Lusztig ordinaires. D'autre part, elle n'impose aucune condition sur $w$. (Cette propri\'et\'e est li\'ee au fait que les $p$-polyn\^omes de Kazhdan--Lusztig antisph\'eriques ``int\`egrent'' la formule de Donkin modulaire, alors que leurs analogues ordinaires n'int\`egrent ``que'' sa version quantique; voir le~\S\ref{ss:formule-donkin} ci-dessous pour plus de d\'etails.)
\item
Une remarque d'Andersen (bas\'ee sur des travaux ant\'erieurs de Jantzen et de Donkin) montre que si on connait les multiplicit\'es $(\mathsf{T}_w : \mathsf{N}_y)$ pour certaines valeurs de $w$, on peut en d\'eduire une formule de caract\`eres pour les modules simples $\mathsf{L}_x$ (pour tous les $x \in \Wextmin$); voir~\cite[\S 1.8]{rw} pour des d\'etails et des r\'ef\'erences. Cependant, la formule obtenue suivant cette approche s'av\`ere assez peu \'eclairante. Pour une formule plus satisfaisante, on pourra consulter~\cite{rw-sim}.
\item
\label{it:conj-rw-Groth}
Consid\'erons l'isomorphisme de groupes ab\'eliens
\[
\Z \otimes_{\Z[v,v^{-1}]}  \mathcal{M}_\ext^{\mathrm{asph}} \simto [\Rep_\varnothing(\bG)]
\]
(o\`u $\Z$ est vu comme $\Z[v,v^{-1}]$-module via $v \mapsto 1$)
envoyant $1 \otimes N_w$ sur $[\mathsf{N}_w]=[\mathsf{M}_w]$ pour tout $w \in \Wextmin$. L'action de $\Z \otimes_{\Z[v,v^{-1}]} \mathcal{H}_\ext$ sur le membre de gauche se r\'ealise en termes du membre de droite de la fa{\c c}on suivante : pour tout $s \in \Waff$ tel que $\ell(s)=1$, la multiplication by $\underline{H}_s$ correspond au morphisme induit par le foncteur $\Theta_s$\footnote{Avec les notations de~\cite{prinblock}, on a $\Theta_s:=T_{\{s\}}^\varnothing \circ T_\varnothing^{\{s\}}$.} de ``croisement des murs'' associ\'e \`a $s$. La Conjecture~\ref{conj:rw} peut se reformuler en disant que cet isomorphisme envoie $1 \otimes \puN_w$ sur $[\mathsf{T}_w]$, pour tout $w \in \Wextmin$.
\end{enumerate}
\end{rmq}

L'article~\cite{rw} contient un plan de preuve de la Conjecture~\ref{conj:rw} (suivant des id\'ees de ``repr\'esentations cat\'egoriques'' dues \`a Rouquier et Khovanov--Lauda), ainsi que sa mise en oeuvre effective dans le cas particulier o\`u $\bG=\mathrm{GL}_n(\bk)$. Cependant, cette approche semble tr\`es difficile \`a mettre en pratique pour d'autres groupes. L'application principale des travaux pr\'esent\'es au~\S\ref{sec:exemples} est une preuve g\'en\'erale de cette conjecture, bas\'ee sur des id\'ees diff\'erentes inspir\'ees de certains travaux de Bezrukavnikov (avec divers collaborateurs) consacr\'es aux groupes quantiques (et notamment~\cite{abg}).

Nous obtenons finalement le r\'esultat suivant.

\begin{thm}
\label{thm:conj-characters}
La conjecture~{\rm \ref{conj:rw}} est vraie.
\end{thm}

\begin{proof}
En combinant les foncteurs des Th\'eor\`emes~\ref{thm:coh-const} et~\ref{thm:prinblock}
 on obtient un foncteur de d\'egraduation (au sens du~\S\ref{ss:degrad}, et par rapport \`a $\langle 1 \rangle$) de $\Dmix_{(\Iw)}(\Gr_{G^\vee},\bk)$ vers $\Db \Rep_\varnothing(\bG)$. Ce foncteur envoie $\nmix_{\Gr_{G^\vee},\lambda}$ sur $\mathsf{N}_{w_\lambda}$, $\dmix_{\Gr_{G^\vee},\lambda}$ sur $\mathsf{M}_{w_\lambda}$, et $\mathscr{T}^\mix_{\Gr_{G^\vee},\lambda}$ sur $\mathsf{T}_{w_\lambda}$, pour tout $\lambda \in \bX$. On en d\'eduit que pour tous $w,y \in \Wextmin$ on a
\[
(\mathsf{T}_w : \mathsf{N}_y) = (\mathsf{T}_w : \mathsf{M}_y) = \sum_{n \in \Z} (\mathscr{T}^\mix_{\Gr_{G^\vee},\lambda} : \dmix_{\Gr_{G^\vee},\mu} \langle n \rangle),
\]
o\`u $\lambda$ et $\mu$ sont les poids tels que $w=w_\lambda$ et $y=w_\mu$.
Le r\'esultat voulu d\'ecoule alors de la Remarque~\ref{rmq:basculants-mix-red}.
\end{proof}

\subsection{Support des modules basculants et conjecture de Humphreys}

On continue dans le cadre du~\S\ref{ss:rw-form}.  Soit $\Fr: \bG \to \dot{\bG}$ le morphisme de Frobenius, et notons $\bG_1$ son noyau (un $\bk$-schéma en groupes fini). 
Dans ce paragraphe nous pr\'esentons une application des constructions du~\S\ref{ss:coh-rep} au calcul de la cohomologie
\[
\coh^\bullet(\bG_1,M) = \Ext^\bullet_{\bG_1}(\bk,M)
\]
de $\bG_1$ à coefficients dans un $\bG$-module $M$ (ou plus g\'en\'eralement dans un complexe de $\bG$-modules).  Dans le cas où $M$ est un $\bG$-module basculant, 
cette approche permet de d\'emontrer
une conjecture de Humphreys sur le support de cette cohomologie.

Notons que pour tout $M$ dans $\Db \Rep(\bG)$, l'espace vectoriel gradu\'e $\coh^\bullet(\bG_1,M)$ admet une action naturelle de $\bG$, qui se factorise via une action de $\dot \bG = \bG/\bG_1$.

\begin{prop}
\label{prop:cohom-glob}
Pour tout $\cF \in \Db\Coh^{\dot\bG \times \Gm}(\tcN_{\dot\bG})$ et tout $k \in \Z$, il existe un isomorphisme canonique et $\dot\bG$-équivariant
\[
\coh^k(\bG_1, \Xi(\cF)) \cong \bigoplus_{n \in \Z}R^{n+k}\Gamma(\tcN_{\dot\bG}, \cF)_{-n}.
\]
\end{prop}

Ici $\Xi$ est le foncteur du Théorème~\ref{thm:prinblock}, et le ``$-n$'' en indice dans le membre de droite désigne la composante graduée de degré $-n$. (Notons que les espaces vectoriels $R^i\Gamma(\tcN_{\dot\bG},\cF)$ ont une structure naturelle de $\dot\bG \times \Gm$-module. La graduation provient de l'action de $\Gm$.)

\begin{proof}
Par adjonction, pour tout $\bG$-module $M$, on a
\[
\Ext^\bullet_{\bG_1}(\bk, M) \cong \Ext^\bullet_{\bG}(\bk, M \otimes \Fr^*\bk[\dot\bG]),
\]
où $\bk[\dot\bG]$ est l'anneau des fonctions sur $\dot \bG$.  Ensuite, nous exploitons l'amplification suivante du Théorème~\ref{thm:prinblock}: si $V$ est un $\dot \bG$-module, alors pour tout $\cF \in \Db\Coh^{\dot\bG \times \Gm}(\tcN)$ il existe un isomorphisme naturel
\[
\Xi(\cF \otimes V) \cong \Xi(\cF) \otimes \Fr^*(V)
\]
(voir~\cite{prinblock}).
Supposons désormais que $M = \Xi(\cF)$. Puisque $\Xi$ est un foncteur de d\'egraduation pour $\langle -1 \rangle [1]$ (voir le Lemme~\ref{lem:koszul-degrad}), on a
\begin{align*}
\Ext^k_{\bG_1}(\bk, M) 
&\cong \Ext^k_{\bG}(\Xi(\cO_{\tcN_{\dot\bG}}), \Xi(\cF) \otimes \Fr^*\bk[\dot\bG]) \\
&\cong \bigoplus_{n \in \Z} \Hom_{\Db\Coh^{\dot\bG \times \Gm}(\tcN_{\dot\bG})}( \cO_{\tcN_{\dot\bG}}, \cF \otimes \bk[\dot\bG]\la -n\ra[n+k]) \\
&\cong \bigoplus_{n \in \Z} \Hom_{\Db\Rep(\dot\bG \times \Gm)}(\bk, R\Gamma(\tcN_{\dot\bG},\cF)\la -n\ra[n+k] \otimes \bk[\dot\bG]).
\end{align*}
La proposition est alors une conséquence du fait que pour tout $\dot\bG$-module $M$, l'espace $\Hom_{\dot\bG}(\bk, M \otimes \bk[\dot\bG])$ s'identifie canoniquement à l'espace vectoriel sous-jacent \`a $M$.
\end{proof}

Soit $\cN_{\dot\bG}$ le cône nilpotent dans l'algèbre de Lie de $\dot\bG$, et notons
\[
\pi: \tcN_{\dot\bG} \to \cN_{\dot\bG}
\]
la \emph{résolution de Springer}.  Rappelons que c'est une résolution \emph{rationelle}; en d'autres termes on a $\pi_*\cO_{\tcN_{\dot\bG}} \cong \cO_{\cN_{\dot\bG}}$, ou, puisque $\cN_{\dot\bG}$ est une vari\'et\'e affine, $R\Gamma(\tcN_{\dot\bG}, \cO_{\tcN_{\dot\bG}}) \cong \bk[\cN_{\dot\bG}]$ (voir~\cite[Theorem~2]{klt}).  Par conséquent, dans le cas où $\cF = \cO_{\tcN_{\dot\bG}}$, la Proposition~\ref{prop:cohom-glob} nous donne un isomorphisme 
\[
\coh^\bullet(\bG_1,\bk) \cong \bk[\cN_{\dot\bG}],
\]
dont on v\'erifie facilement qu'il s'agit d'un isomorphisme d'anneaux (o\`u le membre de gauche est muni du produit de Yoneda.)
On retrouve ainsi un célèbre résultat d'Andersen--Jantzen~\cite{aj} (d\'emontr\'e \'egalement par Friedlander--Parshall~\cite{fp} sous des hypoth\`eses l\'eg\`erement plus fortes).

Pour tout $\bG$-module $M$, la cohomologie $\coh^\bullet(\bG_1,M)$ est munie d'une structure de module sur $\coh^\bullet(\bG_1,\bk)$, et donc sur $\bk[\cN_{\dot\bG}]$, et par l\`a peut s'interpr\'eter comme un faisceau quasi-cohérent sur $\cN_{\dot \bG}$. Le support de ce dernier (une sous-variété fermée de $\cN_{\dot\bG}$) est appelé la \emph{variété de support relative de $M$}. La \emph{variété de support} ordinaire (i.e., non relative) est définie de la même façon en remplaçant $\coh^\bullet(\bG_1,M)$ par $\Ext^\bullet_{\bG_1}(M,M)$.  Les idées de la preuve de la Proposition~\ref{prop:cohom-glob} s'adaptent aisément pour démontrer l'énoncé suivant.

\begin{prop}
Pour tout $\cF \in \Db\Coh^{\dot\bG \times \Gm}(\tcN_{\dot\bG})$, la variété de support relative de $\Xi(\cF)$ est le support de l'objet $\pi_*\cF \in \Db\Coh^{\dot\bG \times \Gm}(\cN_{\dot\bG})$.
\end{prop}

Considérons maintenant les variétés de support relatives des $\bG$-modules basculants $\mathsf{T}_w$ avec $w \in \Waff$.  Il n'est pas difficile de voir que si $\lambda \in \Z R$ n'est pas antidominant, alors $\coh^\bullet(\bG_1, \mathsf{T}_{w_\lambda}) = 0$, et donc sa variété de support relative est vide. La \emph{conjecture de Humphreys relative} propose une description de la variété de support relative de $\mathsf{T}_{w_\lambda}$ pour $\lambda \in (-\bX^+) \cap \Z R$.\footnote{La conjecture telle qu'énoncée dans~\cite{humphreys} concerne les variétés de support ordinaires, mais celles-ci sont \'etroitement liées aux variétés de support relatives : voir notamment~\cite[Lemma~8.11 et Remark~9.4]{ahr}.} Pour expliquer cette conjecture, rappelons que $\Waff$ se répartit en \emph{cellules bilatères}, et que ces derni\`eres sont en bijection naturelle avec les $\bG$-orbites dans $\cN_{\dot\bG}$~\cite{lusztig2}.  Pour $\lambda \in \bX$, posons
\begin{align*}
\bc_\lambda &:= \text{la cellule bilatère qui contient $w_\lambda$}, \\
C_{\bc_\lambda} &:= \text{la $\bG$-orbite dans $\cN_{\dot\bG}$ qui correspond à $\bc_\lambda$.}
\end{align*}
Ces notations permettent d'énoncer le théorème suivant, qui est l'un des résultats principaux de~\cite{ahr}. 

\begin{thm}[Conjecture de Humphreys relative]
Si la caractéristique $p$ de $\bk$ est assez grande, ou si $\bG = \mathrm{SL}_n$ et $p>n+1$, alors pour tout $\lambda \in (-\bX^+) \cap \Z R$ la variété de support relative de $\mathsf{T}_{w_\lambda}$ est $\overline{C_{\bc_\lambda}}$.
\end{thm}

\begin{rmq}
\begin{enumerate}
\item
Dans~\cite{ahr} on d\'emontre \'egalement une version non relative de ce r\'esultat, pour $p$ assez grand.
\item
Dans le cas $\bG = \mathrm{SL}_n$ (et $p>n+1$), ce th\'eor\`eme est obtenu comme corollaire de sa version non relative, d\'emontr\'ee par W. Hardesty dans~\cite{hardesty}.
\item
Dans le cas g\'en\'eral, les m\'ethodes utilis\'ees dans~\cite{ahr} ne permettent pas de d\'eterminer explicitement une borne sur $p$ au-del\`a de laquelle la conjecture est vraie.
\end{enumerate}
\end{rmq}

\subsection{Une ``formule de Donkin'' pour la base $p$-canonique antisph\'erique}
\label{ss:formule-donkin}

Continuons avec les notations du~\S\ref{ss:rw-form}.
Pour $\lambda \in \bX$, on notera $\vartheta_\lambda$ l'\'el\'ement de Bernstein de $\mathcal{H}_\ext$ associ\'e, qui est d\'efini par
\[
\vartheta_\lambda = (H_{t_\mu}) \cdot (H_{t_\nu})^{-1} = (H_{t_\nu})^{-1} \cdot H_{t_\mu}
\]
o\`u $\lambda=\mu-\nu$ avec $\mu,\nu \in \bX^+$ ; voir par exemple~\cite{lusztig}. (Il est bien connu que cet \'el\'ement ne d\'epend pas du choix de $\mu$ et $\nu$.) Pour tout $V$ dans $\Rep(\dot\bG)$, on posera
\[
\vartheta_V := \sum_{\lambda \in \bX} \dim(V_\lambda) \cdot \vartheta_\lambda.
\]
(Ici, $V_\lambda$ d\'esigne le sous-espace de poids $\lambda$ dans $V$, o\`u on a identifi\'e $\bX$ aux caract\`eres de $\dot{\bT}$ comme au~\S\ref{ss:coh-rep}. Cette somme est bien s\^ur finie, et a donc un sens.) Puisque $\dim(V_{w\lambda})=\dim(V_\lambda)$ pour tous $\lambda \in \bX$ et $w \in W$, des r\'esultats de Bernstein et Lusztig (voir~\cite{lusztig}) montrent que cet \'el\'ement est central dans $\mathcal{H}_\ext$. En particulier, on consid\'erera ces \'el\'ements quand $V=\dot{\mathsf{T}}(\lambda)$ est le $\dot\bG$-module basculant ind\'ecomposable de plus haut poids $\lambda$.

Choisissons un poids $\varsigma \in \bX$ tel que $\langle \varsigma, \alpha^\vee \rangle = 1$ pour toute racine simple $\alpha$. (Un tel poids existe gr\^ace \`a notre hypoth\`ese que le sous-groupe d\'eriv\'e de $\bG$ est simplement connexe.)
Nous dirons qu'un \'el\'ement $x \in \Wextmin$ est \emph{restreint} si le poids dominant $x \cdot_p 0$ est restreint. (Bien s\^ur, cette d\'efinition ne d\'epend pas de la caract\'eristique $p$ sur laquelle on travaille, sous notre hypoth\`ese o\`u $p>h$.)
Le r\'esultat suivant est un analogue 
de la ``formule de Donkin'' pour les $\bG$-modules basculants (voir~\cite[\S E.9]{jantzen}).

\begin{prop}
\label{prop:Donkin-pcan}
Supposons que $p \geq 2h-1$. Alors pour tout $x \in \Wextmin$ tel que
$t_{-\varsigma} x$ appartient \`a $\Wextmin$ et est restreint, et pour tout $\lambda \in \bX^+$, on a
\[
\puN_{t_\lambda \cdot x} = \puN_x \cdot \vartheta_{\dot{\mathsf{T}}(\lambda)}.
\]
\end{prop}

\begin{rmq}
\begin{enumerate}
\item
Il n'est pas difficile de voir que notre hypoth\`ese sur $x$ ne d\'epend pas du choix de $\varsigma$.
\item
Puisque $p$ est premier, l'hypoth\`ese ``$p \geq 2h-1$'' est \'equivalente \`a demander que $p \geq 2h-2$ et $p>h$. Ces derni\`eres seront les deux conditions dont nous aurons besoin.
\item
L'hypoth\`ese que $p \geq 2h-1$ n'est probablement pas n\'ecessaire. Elle est utilis\'ee dans la preuve pour se ramener \`a la formule de Donkin en th\'eorie des repr\'esentations ; mais une approche g\'eom\'etrique (qui serait de toute fa{\c c}on plus satisfaisante) devrait permettre de s'en passer.
\end{enumerate}
\end{rmq}

Avant de donner la preuve de la Proposition~\ref{prop:Donkin-pcan} nous devons faire quelques rappels (pour lesquels l'hypoth\`ese sur $p$ n'est pas n\'ecessaire). Consid\'erons la cat\'egorie $\Perv_{G^\vee_\mathscr{O}}(\Gr_{G^\vee},\bk)$ des $\bk$-faisceaux pervers $G^\vee_\mathscr{O}$-\'equivariants sur $\Gr_{G^\vee}$. Cette cat\'egorie peut \^etre munie d'un ``produit de convolution'' $\star^{G^\vee_\mathscr{O}}$ qui en fait une cat\'egorie mono\"idale, et l'\emph{\'equivalence de Satake g\'eom\'etrique} fournit une \'equivalence de cat\'egories mono\"idales
\[
\mathsf{Sat} : (\Perv_{G^\vee_\mathscr{O}}(\Gr_{G^\vee},\bk), \star^{G^\vee_\mathscr{O}}) \simto (\Rep(\dot\bG),\otimes).
\]
(Voir~\cite{mv} pour la preuve originale dans cette g\'en\'eralit\'e, et~\cite{bar} pour une version plus d\'etaill\'ee et des remarques historiques. Cette \'equivalence ne n\'ecessite aucune hypoth\`ese sur $p$, ni sur $\bG$.)

Si on suppose que $p$ est tr\`es bon pour $\bG$, alors la th\'eorie des complexes \`a parit\'e de~\cite{jmw} s'applique dans la cat\'egorie d\'eriv\'ee \'equivariante $\Db_{G^\vee_\mathscr{O}}(\Gr_{G^\vee},\bk)$. Puisque les $G^\vee_\mathscr{O}$-orbites dans $\Gr_{G^\vee}$ sont param\'etr\'ees par $\bX^+$, les complexes \`a parit\'e ind\'ecomposables dans $\Db_{G^\vee_\mathscr{O}}(\Gr_{G^\vee},\bk)$ sont param\'etr\'es par $\bX^+ \times \Z$; on notera $\mathscr{E}^\mathrm{sph}_\lambda$ l'objet associ\'e \`a $(\lambda,0)$. Les r\'esultats de~\cite{mr} montrent que $\mathscr{E}^\mathrm{sph}_\lambda$ est un faisceau pervers, et que $\mathsf{Sat}(\mathscr{E}^\mathrm{sph}_\lambda)$ est isomorphe \`a
$\dot{\mathsf{T}}(\lambda)$. (Voir~\cite{jmw2} pour une preuve ant\'erieure, sous des hypoth\`ese plus fortes.)

Continuons de supposer que $p$ est tr\`es bon pour $\bG$. Pour tout $\lambda \in \bX^+$, la convolution \emph{\`a gauche} avec $\mathscr{E}_\lambda^{\mathrm{sph}}$ d\'efinit un endofunctor de la cat\'egorie $\Db_{\mathrm{Br}}(\Gr_{G^\vee}^\op,\bk)$ qui stabilise les complexes \`a parit\'e (gr\^ace aux r\'esultats de~\cite[\S 4.1]{jmw}); elle induit donc un foncteur
\begin{equation}
\label{eqn:convolution-Esph}
\mathscr{E}_\lambda^{\mathrm{sph}} \star^{G^\vee_\mathscr{O}} (-) : \Dmix_{\mathrm{Br}}(\Gr_{G^\vee}^\op,\bk) \to \Dmix_{\mathrm{Br}}(\Gr_{G^\vee}^\op,\bk).
\end{equation}

\begin{lem}
\label{lem:conv-Elambda-theta}
Via le premier isomorphisme de~\eqref{eqn:Msph-Masph-cat}, l'endomorphisme induit par le foncteur~\eqref{eqn:convolution-Esph} sur $[\Dmix_{\mathrm{Br}}(\Gr_{G^\vee}^\op,\bk)]$ est l'action de $\vartheta_{\dot{\mathsf{T}}(\lambda)} \in \mathcal{H}_\ext$ sur $\mathcal{M}_\ext^{\mathrm{sph}}$.
\end{lem}

\begin{proof}
Les deux applications consid\'er\'ees sont des endomorphismes de $\mathcal{H}_\ext$-modules \`a droite. Puisque $\mathcal{M}_\ext^{\mathrm{sph}}$ est un module cyclique, il suffit donc de v\'erifier qu'elles co\"incident sur l'\'el\'ement $M_e$. 

Pour tout $\mu \in \bX^+$, on notera $x_\mu$ l'\'el\'ement de plus grande longueur dans $Wt_\mu W$. (Cet \'el\'ement est not\'e $n_\mu$ dans~\cite{lusztig}.)
Avec cette notation,
l'image de $M_e$ par le morphisme induit par~\eqref{eqn:convolution-Esph} est $[\mathscr{E}^{\mathrm{sph}}_\lambda] = [\mathscr{E}^\mix_{\Gr_G^\op,w_0 x_\lambda}]$, c'est-\`a-dire $\puM_{w_0 x_\lambda}$. En comparant l'\'equivalence de Satake g\'eom\'etrique sur $\bk$ et en caract\'eristique $0$, et en notant $\dot{\mathsf{N}}(\mu)$ le $\dot\bG$-module induit de plus haut poids $\mu$, on voit que
\[
\puM_{w_0 x_\lambda} = \sum_{\mu \in \bX^+} (\dot{\mathsf{T}}(\lambda) : \dot{\mathsf{N}}(\mu)) \cdot \underline{M}_{w_0 x_\mu}
\]
(voir par exemple~\cite[Remark~4.2]{jmw2}). D'un autre c\^ot\'e, on peut interpr\'eter~\cite[Proposition~8.6]{lusztig} comme disant que pour tout $\mu \in \bX^+$ on a
\[
\underline{M}_{w_0 x_\mu} = M_e \cdot \vartheta_{\dot{\mathsf{N}}(\mu)}.
\]
On en d\'eduit que
\[
\puM_{w_0 x_\lambda} = \sum_{\mu \in \bX^+} (\dot{\mathsf{T}(\lambda)} : \dot{\mathsf{N}}(\mu)) \cdot M_e \cdot \vartheta_{\dot{\mathsf{N}}(\mu)} = M_e \cdot \vartheta_{\dot{\mathsf{T}}(\lambda)},
\]
ce qui conclut la preuve.
\end{proof}

\begin{proof}[D\'emonstration de la Proposition~{\rm \ref{prop:Donkin-pcan}}]
On a vu au cours de la preuve du Th\'eor\`eme~\ref{thm:conj-characters} qu'il existe un foncteur de d\'egraduation $\Dmix_{(\Iw)}(\Gr_{G^\vee}, \bk) \to \Db\Rep_\varnothing(\bG)$. Ce foncteur est t-exact pour la t-structure perverse sur sa source et la t-structure tautologique sur son but. Il induit donc un foncteur exact
\[
\mathsf{F} : \Perv^\mix_{(\Iw)}(\Gr_{G^\vee}, \bk) \to \Rep_\varnothing(\bG).
\]

Dans~\cite[\S 11]{prinblock} on consid\`ere le foncteur de convolution \`a droite avec $\mathscr{E}_\lambda^{\mathrm{sph}}$ sur $\Dmix_{(\Iw)}(\Gr_{G^\vee}, \bk)$ (d\'efini comme ci-dessus pour~\eqref{eqn:convolution-Esph}). On v\'erifie que ce foncteur est t-exact, et induit donc un foncteur entre cat\'egories de faisceaux pervers, et que de plus $\mathsf{F}$ a les propri\'et\'es suivantes :
\begin{enumerate}
\item
$\mathsf{F}(\mathscr{F} \star^{G^\vee_{\mathscr{O}}} \mathscr{E}^{\mathrm{sph}}_\lambda) \cong \mathsf{F}(\mathscr{F}) \otimes_\bk \Fr^* \bigl( \dot{\mathsf{T}}(\lambda) \bigr)$ pour $\mathscr{F}$ dans $\Perv^\mix_{(\Iw)}(\Gr_{G^\vee}, \bk)$ et $\lambda \in \bX^+$;
\item
un object $\mathscr{F}$ de $\Perv^\mix_{(\Iw)}(\Gr_{G^\vee}, \bk)$ est basculant (resp.~basculant et ind\'ecomposable) ssi son image $\mathsf{F}(\mathscr{F})$ est un $\bG$-module basculant (resp.~un $\bG$-module basculant et ind\'ecomposable).
\end{enumerate}

Sous nos hypoth\`eses un r\'esultat de Donkin (voir~\cite[\S II.E.9]{jantzen}) assure que
\[
\mathsf{T}_x \otimes \Fr^* \bigl( \dot{\mathsf{T}}(\lambda) \bigr) \cong \mathsf{T}_{t_\lambda x}
\]
pour $x$ comme dans l'\'enonc\'e et $\lambda \in \bX^+$.
En utilisant les propri\'et\'es de $\mathsf{F}$ rappel\'ees ci-dessus,
on en d\'eduit que
\[
\mathscr{T}^\mix_{\Gr_{G^\vee},\mu} \star^{G^\vee_{\mathscr{O}}} \mathscr{E}^{\mathrm{sph}}_\lambda \cong \mathscr{T}^\mix_{\Gr_{G^\vee},\nu},
\]
o\`u $\mu,\nu \in \bX$ sont les poids tels que $w_\mu = x$ et $w_\nu = t_\lambda x$. On utilise ensuite la Remarque~\ref{rmq:basculants-mix-red} pour ``transf\'erer'' cet isomorphisme en un isomorphisme
\[
\mathscr{E}^{\mathrm{sph}}_\lambda \star^{G^\vee_{\mathscr{O}}} \mathscr{T}^\mix_{\Gr_{G^\vee}^\op,x} \cong \mathscr{T}^\mix_{\Gr_{G^\vee}^\op,t_\lambda x}.
\]
En utilisant le Lemme~\ref{lem:conv-Elambda-theta} et la Proposition~\ref{prop:character-tilting}, cet isomorphisme implique que
\[
\iota^{-1}(\puN_x) \cdot \vartheta_{\dot{\mathsf{T}}(\lambda)} = \iota^{-1}(\puN_{t_\lambda x}).
\]
En appliquant $\iota$ et en utilisant~\eqref{eqn:iota-M-N} et le fait que $\iota(\vartheta_\nu) =\vartheta_\nu$ pour tout $\nu \in \bX$, on en d\'eduit l'\'egalit\'e voulue.
\end{proof}

Une fois l'\'egalit\'e de la Proposition~\ref{prop:Donkin-pcan} \'etablie,
on peut la ``relever'' en un \'enonc\'e sur les complexes \`a parit\'e de la fa{\c c}on suivante. Dans cet \'enonc\'e, on note
\[
\mathsf{Z} : \Perv_{G^\vee_{\mathscr{O}}}(\Gr_{G^\vee},\bk) \to \Perv_{I^\vee}(\Fl_{G^\vee},\bk)
\]
(o\`u $I^\vee$ est le sous-groupe d'Iwahori not\'e $I$ au~\S\ref{ss:Koszul-affine}, pour notre groupe $G^\vee$) le ``foncteur central'' construit par Gaitsgory dans~\cite{gaitsgory} (ou sa variante consid\'er\'ee par Zhu dans~\cite{zhu-conj}). D'autre part, la cat\'egorie d\'eriv\'ee $I^\vee$-\'equivariante $\Db_{I^\vee}(\Fl_{G^\vee},\bk)$ poss\`ede un produit de convolution qui en fait une cat\'egorie mono\"idale, et qu'on notera $\star^{I^\vee}$. Cette cat\'egorie agit naturellement par convolution \`a droite sur $\Db_{\IW}(\Fl_{G^\vee},\bk)$ ; le bifoncteur associ\'e sera not\'e $\star^{I^\vee}$ \'egalement.

\begin{prop}
\label{prop:Donkin-parity}
Supposons que $p \geq 2h-1$, et soient $x$ et $\lambda$ comme dans la Proposition~{\rm \ref{prop:Donkin-pcan}}. Alors on a
\[
\mathscr{E}_{\IW,x} \star^{I^\vee} \mathsf{Z}(\mathscr{E}^{\mathrm{sph}}_\lambda) \cong \mathscr{E}_{\IW,t_\lambda x}.
\]
\end{prop}

\begin{rmq}
\begin{enumerate}
\item
Comme pour la Proposition~\ref{prop:Donkin-pcan}, l'hypoth\`ese sur $p$ dans la Proposition~\ref{prop:Donkin-parity} n'est probablement pas n\'ecessaire. En fait ici on va d\'eduire la Proposition~\ref{prop:Donkin-parity} de la Proposition~\ref{prop:Donkin-pcan}. Une strat\'egie plus satisfaisante serait de d\'emontrer la Proposition~\ref{prop:Donkin-parity} par des arguments g\'eom\'etriques, et sous une hypoth\`ese plus faible, et d'en d\'eduire la Proposition~\ref{prop:Donkin-pcan}.
\item
Signalons qu'il n'est \emph{pas} vrai que $\mathscr{E}_{\IW,x} \star^{I^\vee} \mathsf{Z}(\mathscr{E}^{\mathrm{sph}}_\lambda)$ est un complexe \`a parit\'e pour tous $x \in \Wextmin$ et $\lambda \in \bX^+$.
\end{enumerate}
\end{rmq}

\begin{proof}
Notons $E$ le sous-$\Z[v,v^{-1}]$-module de $\mathcal{M}^{\mathrm{asph}}_\ext$ engendr\'e par les \'el\'ements $\puN_x$ avec $x \in \Wextmin$ tel que $t_{-\varsigma} x \in \Wextmin$. Alors $E$ est un sous-$\mathcal{H}_\ext$-module de $\mathcal{M}^{\mathrm{asph}}_\ext$. En effet, pour justifier cela il suffit de montrer que pour tout $x$ comme ci-dessus et pour toute r\'eflexion simple $s$, $\puN_x \cdot \underline{H}_s$ appartient \`a $E$. Mais $\puN_x \cdot \underline{H}_s$ est la classe d'un complexe \`a parit\'e (consid\'er\'e comme un complexe dans $\Dmix_{\IW}(\Fl_{G^\vee},\bk)$ concentr\'e en degr\'e $0$) ; ses coefficients dans la base $(\puN_y : y \in \Wextmin)$ sont donc des polyn\^omes \`a coefficients positifs ou nuls. Donc, pour d\'emontrer notre affirmation, il suffit de voir que les valeurs en $1$ des coefficients de $\puN_x \cdot \underline{H}_s$ sur les \'el\'ements $\puN_y$ tels que $t_{-\varsigma} y$ n'appartient \emph{pas} \`a $\Wextmin$ s'annulent. En utilisant 
la Remarque~\ref{rmq:conj-rW}\eqref{it:conj-rw-Groth}, ce qu'on doit d\'emontrer est donc que
si $x$ est comme ci-dessus et si $\Theta_s$ est le foncteur de ``croisement des murs'' associ\'e \`a $s$, 
alors $\Theta_s(\mathsf{T}_x)$ est une somme directe d'objets $\mathsf{T}_y$ avec $y \in \Wextmin$ tel que $t_{-\varsigma} y \in \Wextmin$. Mais les modules basculants ind\'ecomposables $\mathsf{T}_y$ avec $y$ v\'erifiant cette condition sont exactement ceux qui sont projectifs comme $\bG_1$-modules ; voir~\cite[Lemma~E.8]{jantzen}. Il suffit donc de remarquer que cette propri\'et\'e est stable par les foncteurs de translation, ce qui d\'ecoule de la discussion de~\cite[\S II.9.22]{jantzen}.

Cette discussion montre \'egalement que $E$ est un $\mathcal{H}_\ext$-module cyclique, engendr\'e par $\puN_{t_\varsigma}=[\mathscr{E}_{\IW,t_{\varsigma}}]$ ; plus pr\'ecis\'ement, tout complexe \`a parit\'e de la forme $\mathscr{E}_{\IW,x}$ avec $x \in \Wextmin$ tel que $t_{-\varsigma} x \in \Wextmin$ apparait comme facteur direct d'un complexe \`a parit\'e de la forme $\mathscr{E}_{\IW,t_{\varsigma}} \star^{I^\vee} \mathscr{F}$ avec $\mathscr{F}$ un complexe \`a parit\'e $I^\vee$-\'equivariant sur $\Fl_{G^\vee}$.
En utilisant~\cite[Theorem~7.3]{zhu-conj} on voit que
\[
\mathscr{E}_{\IW,t_{\varsigma}} \star^{I^\vee} \mathscr{F} \star^{I^\vee} \mathsf{Z}(\mathscr{E}^{\mathrm{sph}}_\lambda) \cong \mathscr{E}_{\IW,t_{\varsigma}} \star^{I^\vee} \mathsf{Z}(\mathscr{E}^{\mathrm{sph}}_\lambda) \star^{I^\vee} \mathscr{F}.
\]
D'autre part, d'apr\`es~\cite[Proposition~4.18]{bgmrr} l'objet $\mathscr{E}_{\IW,t_{\varsigma}} \star^{I^\vee} \mathsf{Z}(\mathscr{E}^{\mathrm{sph}}_\lambda)$ est un complexe \`a parit\'e. On en d\'eduit qu'il en est de m\^eme pour $\mathscr{E}_{\IW,t_{\varsigma}} \star^{I^\vee} \mathscr{F} \star^{I^\vee} \mathsf{Z}(\mathscr{E}^{\mathrm{sph}}_\lambda)$, et donc pour $\mathscr{E}_{\IW,x} \star^{I^\vee} \mathsf{Z}(\mathscr{E}^{\mathrm{sph}}_\lambda)$.

Ce qu'on vient donc de montrer implique en particulier que le foncteur $(-) \star^{I^\vee} \mathsf{Z}(\mathscr{E}^{\mathrm{sph}}_\lambda)$ induit un endomorphisme de $\mathcal{H}_\ext$-module de $E$. Puisque, d'apr\`es~\cite[Proposition~4.18]{bgmrr}, on a
\[
[\mathscr{E}_{\IW,t_{\varsigma}} \star^{I^\vee} \mathsf{Z}(\mathscr{E}^{\mathrm{sph}}_\lambda)] = [\mathscr{E}_{\IW,t_{\lambda+\varsigma}}] = \puN_{t_{\lambda+\varsigma}} = \puN_{t_\varsigma} \cdot \vartheta_{\dot{\mathsf{T}}(\lambda)},
\]
cet endomorphisme co\"incide avec l'action de $\vartheta_{\dot{\mathsf{T}}(\lambda)}$. Donc, pour $x$ comme dans l'\'enonc\'e, on a
\[
[\mathscr{E}_{\IW,x} \star^{I^\vee} \mathsf{Z}(\mathscr{E}^{\mathrm{sph}}_\lambda)] = \puN_x \cdot \vartheta_{\dot{\mathsf{T}}(\lambda)} = \puN_{t_\lambda x} = [\mathscr{E}_{\IW,t_\lambda x}].
\]
On en d\'eduit l'isomorphisme de la proposition.
\end{proof}

\subsection{Une ``formule de Steinberg'' pour la base $p$-canonique sph\'erique}
\label{ss:formule-steinberg}

Consid\'erons un \'el\'ement
$\varsigma \in \bX$
comme au~\S\ref{ss:formule-donkin}.
Dans l'article~\cite{rw-sim}, G. Williamson et le second auteur consid\`erent l'unique morphisme de $\mathcal{H}_\ext$-modules \`a droite $\varphi : \mathcal{M}^{\mathrm{sph}}_\ext \to \mathcal{M}^{\mathrm{asph}}_\ext$ envoyant $M_e$ sur $\underline{N}_{t_\varsigma}$. Ils d\'emontrent que ce morphisme est injectif et v\'erifie
\[
\varphi(\puM_w) = \puN_{t_{\varsigma} w}
\]
pour tout $w \in \Wextmin$, sous l'hypoth\`ese o\`u $p$ est bon pour $\bG$. En utilisant cet r\'esultat, la Proposition~\ref{prop:Donkin-pcan} est \'equivalente au r\'esultat suivant, qu'on peut voir comme une ``formule de Steinberg'' pour la base $p$-canonique sph\'erique.

\begin{prop}
\label{prop:Steinberg-pcan}
Supposons que $p \geq 2h-1$. Alors pour tout $x \in \Wextmin$ restreint et tout $\lambda \in \bX^+$, on a
\[
\puM_{t_\lambda \cdot x} = \puM_x \cdot \vartheta_{\dot{\mathsf{T}}(\lambda)}.
\]
\end{prop}

Une fois cette \'egalit\'e \'etablie, comme pour la Proposition~\ref{prop:Donkin-parity} on peut obtenir les isomorphismes suivants (sous l'hypoth\`ese $p \geq 2h-1$) :
\begin{enumerate}
\item
\label{it:steinberg-parity}
dans $\Db_{\mathrm{Br}}(\Gr_{G^\vee}^\op,\bk)$, si $x \in \Wextmin$ est restreint et si $\lambda \in \bX^+$ alors
\[
\mathscr{E}^{\mathrm{sph}}_\lambda \star^{G^\vee_{\mathscr{O}}} \mathscr{E}_{\Gr_{G^\vee}^\op,x} \cong \mathscr{E}_{\Gr_{G^\vee}^\op,t_\lambda x};
\]
\item
\label{it:steinberg-tilting}
dans $\Db \Coh^{\dot \bG \times \Gm}(\tcN_{\dot \bG})$, si $\mu \in \bX$ est tel que $w_\mu$ est restreint et si $\lambda \in \bX^+$, alors
\[
\mathcal{T}_\mu \otimes \dot{\mathsf{T}}(\lambda) \cong \mathcal{T}_\nu,
\]
o\`u $\nu \in \bX$ est le poids tel que $w_\nu = t_\lambda w_\mu$.
\end{enumerate}
(Pour le cas~\eqref{it:steinberg-parity}, le fait que notre objet est un complexe \`a parit\'e d\'ecoule des r\'esultats de~\cite[\S 4.1]{jmw}, et on utilisera le Lemme~\ref{lem:conv-Elambda-theta}. Pour le cas~\eqref{it:steinberg-tilting}, le fait que notre objet est un faisceau exotique basculant d\'ecoule de~\cite[Proposition~4.10]{mr:ets}, et on peut identifier l'action du foncteur $\dot{\mathsf{T}}(\lambda) \otimes (-)$ sur le groupe de Grothendieck en utilisant~\cite[Proposition~4.6]{mr:ets} et la relation entre les $q$-analogues de Lusztig et les polyn\^omes de Kazhdan--Lusztig, voir~\cite{kato}. Alternativement, on peut d\'eduire ce cas du cas~\eqref{it:steinberg-parity} en utilisant l'\'equivalence du Th\'eor\`eme~\ref{thm:coh-const}, dont la version construite dans~\cite{arider} est compatible avec l'\'equivalence de Satake au sens appropri\'e.)

\begin{rmq}
\begin{enumerate}
\item
Les isomorphismes \'enonc\'es en~\eqref{it:steinberg-parity} et~\eqref{it:steinberg-tilting} ci-dessus peuvent \^etre consid\'er\'es comme des analogues de la ``formule de Steinberg'' pour les $\bG$-modules simples; voir~\cite[Proposition~II.3.16]{jantzen}. (Cependant, ils font intervenir un produit tensoriel avec un module \emph{basculant} et non un module simple.)
\item
Comme pour la Proposition~\ref{prop:Donkin-parity}, il serait plus satisfaisant de d\'emontrer les isomorphismes ci-dessus par des arguments g\'eom\'etriques. Pour des coefficients de caract\'eristique $0$, un analogue de l'isomorphisme pr\'esent\'e en~\eqref{it:steinberg-parity} est d\'emontr\'e dans~\cite[Theorem~1.3.5]{abbgm}.
\end{enumerate}
\end{rmq}

\addtocontents{toc}{\protect\addvspace{1.5em}}


\end{document}